\documentclass[11pt]{article}

\usepackage{amsmath,amssymb,amsfonts,enumitem,amsthm,accents,pstricks,pst-node,pstricks-add,mathrsfs,xy,graphicx}
\xyoption{all}

\numberwithin{equation}{section}

\newtheorem{thm}{Theorem}[section]
\newtheorem{lmm}[thm]{Lemma}
\newtheorem{prp}[thm]{Proposition}
\newtheorem{crl}[thm]{Corollary}
\theoremstyle{definition}
\newtheorem{dfn}[thm]{Definition}
\newtheorem{eg}[thm]{Example}

\def\BE#1{\begin{equation}\label{#1}}
\def\EE{\end{equation}}
\def\eref#1{(\ref{#1})}

\def\lra{\longrightarrow}
\def\xlra#1{\xrightarrow{#1}}
\def\Lra{\Longrightarrow}
\def\Llra{\Longleftrightarrow}
\def\lhra{\ensuremath{\lhook\joinrel\relbar\joinrel\rightarrow}}

\def\ov#1{\overline{#1}}

\def\wt#1{\widetilde{#1}}
\def\sf#1{\textsf{#1}}

\def\sm#1{\begin{small}{#1}\end{small}}

\def\wh#1{\widehat{#1}}

\def\be{\beta}

\def\io{\iota}

\def\na{\nabla}
\def\om{\omega}
\def\si{\sigma}

\def\vph{\varphi}

\def\ze{\zeta}

\def\De{\Delta}

\def\Om{\Omega}

\def\bI{\mathbb I}
\def\cA{\mathcal A}
\def\C{\mathbb C}

\def\fD{\mathfrak D}

\def\bI{\mathbb I}

\def\cN{\mathcal N}
\def\cO{\mathcal O}
\def\P{\mathbb P}
\def\cP{\mathcal P}
\def\fR{\mathfrak R}
\def\R{\mathbb R}
\def\cR{\mathcal R}

\def\bS{\mathbb S}

\def\X{\mathbf X}

\def\Z{\mathbb Z}

\def\jo{\jmath}

\def\fc{\mathfrak c}

\def\fI{\mathfrak i}

\def\AK{\textnormal{AK}}
\def\Aux{\textnormal{Aux}}
\def\Aut{\textnormal{Aut}}
\def\codim{\textnormal{codim}}

\def\Dom{\textnormal{Dom}}

\def\id{\textnormal{id}}
\def\Im{\textnormal{Im}}

\def\nd{\textnormal{d}}

\def\supp{\textnormal{supp}}

\def\Symp{\textnormal{Symp}}

\def\ver{\textnormal{ver}}

\def\eset{\emptyset}
\def\prt{\partial}
\def\1{\mathbf 1}

\def\bu{\bullet}

\begin{document}

\title{Normal Crossings Singularities for Symplectic Topology, II}
\author{Mohammad F.~Tehrani\thanks{Partially supported 
by the Homological Mirror Symmetry Collaboration grant}\,, 
Mark McLean\thanks{Partially supported by NSF grants 1508207 and 1811861}\,, and 
Aleksey Zinger\thanks{Partially supported by NSF grants 1500875 and 1901979}}

\date{\today}

\maketitle

\begin{abstract}
\noindent
In recent work, we introduced topological notions of simple normal crossings symplectic divisor 
and variety, 
showed that they are equivalent, in a suitable sense, to the corresponding geometric notions,
and established a topological smoothability criterion for~them. 
The present paper extends these notions to arbitrary normal crossings singularities, 
from both local and global perspectives, and
shows that they are also equivalent to the corresponding geometric notions.
In subsequent papers, we extend our smoothability criterion to 
arbitrary normal crossings symplectic varieties and construct 
a variety of geometric structures associated with normal crossings singularities
in algebraic geometry.
\end{abstract}

\tableofcontents

\section{Introduction}
\label{intro_sec}

\noindent
Normal crossings (NC) divisors and varieties are the most basic and important classes 
of singular objects in complex algebraic and K\"ahler geometries.
An \sf{NC divisor} in a smooth variety~$X$ is a subvariety~$V$ 
locally defined by an equation of the form
\BE{Local1_e} z_1\cdots z_k = 0\EE
in a holomorphic coordinate chart $(z_1,\ldots,z_n)$ on~$X$. 
From a global perspective, an NC divisor is the image of an  immersion with 
 transverse self-intersections.
A \sf{simple normal crossings} (or~\sf{SC}) \sf{divisor} 
is a global transverse union of smooth divisors, i.e. 
$$V=\bigcup_{i\in S}\!V_i \subset X$$
with the \sf{singular locus}
$$V_\prt=\bigcup_{\begin{subarray}{c}I\subset S\\ |I|=2\end{subarray}}\!\!\!V_{I}\subset V,
\qquad\hbox{where}\quad V_I=\bigcap_{i\in I}\!V_i\quad \forall~I\!\subset\!S.$$
An~\sf{NC variety} of complex dimension $n$ is a variety $X_{\eset}$ 
that can be locally embedded as an NC divisor in~$\C^{n+1}$. 
Thus, every sufficiently small open set~$U_\eset$ in  $X_{\eset}$ 
can be written~as 
$$U_\eset=\bigg(\bigsqcup_{i\in S}\!U_i\bigg)\!\Big/\!\!\!\sim,
\qquad  U_{ij}\approx U_{ji}\quad \forall~i,j\!\in\!S,~i\!\neq\!j,$$
where $\{U_{ij}\}_{j\in S-i}$ is an SC divisor in a smooth component~$U_i$ of~$U_\eset$. 
An \sf{SC~variety} 
is a global transverse union of smooth varieties $\{X_i\}_{i\in S}$ 
along SC divisors $\{X_{ij}\}_{j\in S-i}$ in~$X_i$,~i.e.
\BE{Xesetdfn_e}
X_{\eset}=\bigg(\bigsqcup_{i\in S}\!X_i\bigg)\!\Big/\!\!\!\sim,\qquad  
X_{ij}\approx X_{ji}\quad \forall~i,j\!\in\!S,~i\!\neq\!j,\EE
with the \sf{singular locus}
\BE{Xprtdfn_e}
X_{\prt}=\bigcup_{\begin{subarray}{c}i,j\in S\\ i\neq j\end{subarray}}
\!\!\!X_{ij}\subset X_\eset. \EE
A two-dimensional 3-fold SC variety is shown in Figure~\ref{P2cut_fig}.\\

\noindent
In parallel with his introduction of $J$-holomorphic curve techniques 
in symplectic topology in~\cite{Gr}, 
Gromov asked about the feasibility of introducing notions of 
singular (sub-)varieties of higher dimension suitable for this field; 
see \cite[p343]{GrBook}. 
Important developments since then, such as symplectic sum constructions \cite{Gf,MW}, 
degeneration and decomposition formulas for Gromov-Witten invariants
\cite{Tian,CH,LR,Jun2,Brett}, log Gromov-Witten theory \cite{GS,AC}, 
affine symplectic geometry \cite{MAff,McLean}, 
homological mirror symmetry~\cite{Sheridan}, and 
a new perspective on Atiyah-Floer Conjecture \cite{DF} suggest 
the need for (soft) symplectic notions of NC divisors and varieties 
that are equivalent, in a suitable sense, to the corresponding (rigid) geometric notions.\\

\noindent
A {\it smooth} symplectic divisor~$V$ in a symplectic manifold $(X,\om)$
is an almost K\"ahler divisor with respect to an $\om$-compatible almost complex structure~$J$
which is integrable in the normal direction to~$V$.
Furthermore, the projection
\BE{AKtoXproj_e}\AK(X,V)\lra \Symp(X,V), \qquad (\om,J)\lra \om,\EE
from the space of such pairs $(\om,J)$ to the space of symplectic forms~$\om$ on~$X$
which restrict to symplectic forms on~$V$ is a weak homotopy equivalence.
This property of~\eref{AKtoXproj_e}, rather than its surjectivity, 
is fundamental to applications of smooth divisors in symplectic topology.
In~\cite{SympDivConf}, we propose\\

\begin{minipage}{6in}\label{DefPhil_minip}
{\it to treat NC symplectic divisors/varieties up to deformation equivalence,
showing that each deformation equivalence class has a subspace of 
sufficiently ``nice" representatives so that an appropriate analogue
of~\eref{AKtoXproj_e} is a weak homotopy equivalence}.\\
\end{minipage}

\noindent
We also prove that unions of so-called {\it positively intersecting collections}
of smooth symplectic divisors provide for a notion of SC symplectic divisors
compatible with this prospective and lead to a compatible notion of SC symplectic varieties.
An overview of the program initiated in~\cite{SympDivConf} and of its potential applications
in symplectic topology and algebraic geometry appears in~\cite{SympNCSumm}.\\

\noindent
The present paper extends Definitions~\ref{SCD_dfn} and~\ref{SCC_dfn} 
of SC symplectic divisors and varieties,
the notions of regularizations for them, and 
the main theorems of~\cite{SympDivConf} to arbitrary NC divisors and varieties. 
This is done from both local and global perspectives, 
which are better suited for different types of applications.
In the local perspectives of Definitions~\ref{NCD_dfn} and~\ref{NCC_dfn},
\sf{NC symplectic divisors} and \sf{varieties} are spaces 
that are locally  SC symplectic divisors and varieties, respectively.
In the global perspectives of Proposition~\ref{NCD_prp} and Section~\ref{NCCgl_subs},
NC symplectic divisors and varieties are images of immersions with transverse 
self-intersections.
Regularizations are key to  many applications of divisors in symplectic topology, including 
the symplectic sum constructions of \cite{Gf,MW,SympSumMulti}.
They in particular ensure the existence of almost complex structures~$J$ on~$X$ that are ``nice"
along divisors in smooth and NC symplectic varieties.
Such almost complex structures in turn play a central role in Gromov-Witten theory 
and in its interplay with algebraic geometry and string theory,
for example.\\

\noindent
After recalling the notions of regularizations for SC symplectic divisors and varieties
in Section~\ref{SC_sec},
we define NC divisors and their regularizations from a local perspective
in Section~\ref{NCDloc_subs} and from a global perspective in Section~\ref{NCDgl_subs}.
The two perspectives are  equivalent by Proposition~\ref{NCD_prp}
and the last paragraph of Section~\ref{NCDgl_subs}. 
Theorem~\ref{NCD_thm} extends \cite[Theorem~2.13]{SympDivConf}
to arbitrary NC symplectic divisors.
We define NC varieties and their regularizations from a local perspective
in Section~\ref{NCCloc_subs} and from a global perspective in
Section~\ref{NCCgl_subs}.
The two perspectives are shown to be equivalent in Section~\ref{NCCcomp_subs}. 
Theorem~\ref{NCC_thm} extends \cite[Theorem~2.17]{SympDivConf} 
to arbitrary NC symplectic varieties. 
Section~\ref{eg_subs} provides examples of non-SC normal crossings divisors and varieties.
In Section~\ref{NCCpf_sec}, we deduce Theorem~\ref{NCC_thm} from
the proof of \cite[Theorem~2.17]{SympDivConf} using the local perspective of 
Section~\ref{NCCloc_subs} and a seemingly weaker, but equivalent, version 
of the notion of regularization of \cite[Definition~2.15(1)]{SympDivConf}.

\section{Simple crossings divisors and varieties}
\label{SC_sec}

\noindent
We begin by introducing the most commonly used notation.
For a set $S$, denote by $\cP(S)$ the collection of subsets of~$S$ and
by $\cP^*(S)\!\subset\!\cP(S)$ the collection of nonempty subsets.
If in addition $i\!\in\!S$, let 
$$\cP_i(S)=\big\{I\!\in\!\cP(S)\!:\,i\!\in\!S\big\}.$$
For $N\!\in\!\Z^{\ge0}$, let 
$$[N]=\{1,\ldots,N\}, \quad \cP(N)\!=\!\cP\big([N]\big), \quad 
\cP^*(N)\!=\!\cP^*\big([N]\big).$$
For $i\!\in\![N]$, let $\cP_i(N)\!=\!\cP_i([N])$.\\

\noindent
For $k\!\in\!\Z^{\ge0}$, denote by $\bS_k$ the $k$-th symmetric group. 
For $k'\!\in\![k]$, we identify $\bS_{k'}$ with the subgroup of~$\bS_k$
consisting of the permutations of $[k']\!\subset\![k]$ and
denote by $\bS_{[k]-[k']}\!\subset\!\bS_k$ the subgroup of~$\bS_k$
consisting of the permutations of $[k]\!-\![k']$.
In particular, 
$$\bS_{k'}\times\!\bS_{[k]-[k']}\subset \bS_k\,.$$
For each $\si\!\in\!\bS_k$ and $i\!\in\![k]$, let $\si_i\!\in\!\bS_{k-1}$
be the permutation obtained from the bijection
$$[k]\!-\!\{i\}\lra [k]\!-\!\{\si(i)\}, \qquad j\lra \si(j),$$
by identifying its domain and target with $[k\!-\!1]$ in the order-preserving fashions.\\

\noindent
If $\cN\!\lra\!V$ is a vector bundle, $\cN'\!\subset\!\cN$, and $V'\!\subset\!V$, we define
$$\cN'|_{V'}=\cN|_{V'}\cap\cN'\,.$$
Let $\bI\!=\![0,1]$.

\subsection{Notation and definitions}
\label{SCdfn_subs}

\noindent
Let $X$ be a (smooth) manifold. 
For any submanifold $V\!\subset\!X$, let
$$\cN_XV\equiv \frac{TX|_V}{TV}\lra V$$
denote the normal bundle of~$V$ in~$X$.
For a collection $\{V_i\}_{i\in S}$ of submanifolds of~$X$ and $I\!\subset\!S$, let
$$V_I\equiv \bigcap_{i\in I}\!V_i\subset X\,.$$
Such a collection 
is called \sf{transverse} if any subcollection $\{V_i\}_{i\in I}$ of these submanifolds
intersects transversely, i.e.~the homomorphism
\BE{TransVerHom_e}
T_xX\oplus\bigoplus_{i\in I}T_xV_i\lra \bigoplus_{i\in I}T_xX, \qquad
\big(v,(v_i)_{i\in I}\big)\lra (v\!+\!v_i)_{i\in I}\,,\EE
is surjective for all $x\!\in\!V_I$. 
By the Inverse Function Theorem,
each subspace $V_I\!\subset\!X$ is then a submanifold of~$X$ 
of codimension
$$\codim_XV_I=\sum_{i\in I}\codim_XV_i$$
and the homomorphisms 
\BE{cNorient_e2}\begin{split}
\cN_XV_I\lra \bigoplus_{i\in I}\cN_XV_i\big|_{V_I}\quad&\forall~I\!\subset\!S,\qquad
\cN_{V_{I-i}}V_I\lra \cN_XV_i\big|_{V_I} \quad\forall~i\!\in\!I\!\subset\!S,\\
&\bigoplus_{i\in I-I'}\!\!\!\cN_{V_{I-i}}V_I\lra 
\cN_{V_{I'}}V_I \quad\forall~I'\!\subset\!I\!\subset\!S
\end{split}\EE
induced by inclusions of the tangent bundles are isomorphisms.\\

\noindent
As detailed in \cite[Section~2.1]{SympDivConf}, 
a transverse collection $\{V_i\}_{i\in S}$ of oriented submanifolds of
an oriented manifold~$X$
of even codimensions  induces an orientation on each submanifold $V_I\!\subset\!X$
with $|I|\!\ge\!2$; we call it \sf{the intersection orientation of~$V_I$}.
If $V_I$ is zero-dimensional, it is a discrete collection of points in~$X$
and the homomorphism~\eref{TransVerHom_e} is an isomorphism at each point $x\!\in\!V_I$;
the intersection orientation of~$V_I$ at $x\!\in\!V_I$
then corresponds to a plus or minus sign, depending on whether this isomorphism
is orientation-preserving or orientation-reversing.
We call the original orientations of 
$X\!=\!V_{\eset}$ and $V_i\!=\!V_{\{i\}}$ \sf{the intersection orientations}
of these submanifolds~$V_I$ of~$X$ with $|I|\!<\!2$.\\

\noindent
Suppose $(X,\om)$ is a symplectic manifold and $\{V_i\}_{i\in S}$ is a transverse collection 
of submanifolds of~$X$ such that each $V_I$ is a symplectic submanifold of~$(X,\om)$.
Each $V_I$ then carries an orientation induced by $\om|_{V_{I}}$,
which we call the \sf{$\om$-orientation}.
If $V_I$ is zero-dimensional, it is automatically a symplectic submanifold of~$(X,\om)$;
the $\om$-orientation of~$V_I$ at each point $x\!\in\!V_I$ corresponds to the plus sign 
by definition.
By the previous paragraph, the $\om$-orientations of~$X$ and~$V_i$ with $i\!\in\!I$
also induce intersection orientations on all~$V_I$.

\begin{dfn}\label{SCD_dfn}
Let $(X,\om)$ be a symplectic manifold.
A \sf{simple crossings} (or \sf{SC}) \sf{symplectic divisor} 
in~$(X,\om)$ is a finite transverse collection 
$\{V_i\}_{i\in S}$ of closed submanifolds of~$X$ of codimension~2 such that 
$V_I$ is a symplectic submanifold of~$(X,\om)$ for every $I\!\subset\!S$
and the intersection and $\om$-orientations of~$V_I$ agree.
\end{dfn}

\noindent
An SC symplectic divisor $\{V_i\}_{i\in S}$ with $|S|\!=\!1$ is 
a smooth symplectic divisor in the usual sense. 
If $(X,\om)$ is a 4-dimensional symplectic manifold, 
a finite transverse collection  $\{V_i\}_{i\in S}$ of closed submanifolds of~$X$ 
of codimension~2 is an SC symplectic divisor if all points of the pairwise intersections
$V_{i_1}\!\cap\!V_{i_2}$ with $i_1\!\neq\!i_2$ are positive. 
By \cite[Example~2.7]{SympDivConf}, it is not sufficient to consider 
the deepest (non-empty) intersections in higher~dimensions.

\begin{dfn}\label{SCdivstr_dfn}
Let $X$ be a manifold and $\{V_i\}_{i\in S}$ be a finite transverse collection of 
closed submanifolds of~$X$ of codimension~2.
A \sf{symplectic structure on $\{V_i\}_{i\in S}$ in~$X$} is a symplectic form~$\om$ 
on~$X$ such that $V_I$ is a symplectic submanifold of $(X,\om)$ for all $I\!\subset\!S$.
\end{dfn}

\noindent
For $X$ and $\{V_i\}_{i\in S}$ as in Definition~\ref{SCdivstr_dfn}, 
we denote by $\Symp(X,\{V_i\}_{i\in S})$ the space of all symplectic structures 
on $\{V_i\}_{i\in S}$ in~$X$ and by 
$$\Symp^+\big(X,\{V_i\}_{i\in S}\big)\subset \Symp\big(X,\{V_i\}_{i\in S}\big)$$
the subspace of the symplectic forms~$\om$ such that $\{V_i\}_{i\in S}$
is an SC symplectic divisor in~$(X,\om)$.
The latter is a union of topological components of the former.

\begin{dfn}\label{TransConf_dfn1}
Let $N\!\in\!\Z^+$.
An \sf{$N$-fold transverse configuration} is a tuple $\{X_I\}_{I\in\cP^*(N)}$
of manifolds such that $\{X_{ij}\}_{j\in[N]-i}$ is a transverse collection 
of submanifolds of~$X_i$ for each $i\!\in\![N]$ and
$$X_{\{ij_1,\ldots,ij_k\}}\equiv \bigcap_{m=1}^k\!\!X_{ij_m}
=X_{ij_1\ldots j_k}\qquad\forall~j_1,\ldots,j_k\in[N]\!-\!i.$$
A \sf{symplectic structure} on an $N$-fold transverse configuration~$\X$
such that $X_{ij}$ is a closed submanifold of~$X_i$ of codimension~2
for all $i,j\!\in\![N]$ distinct is a~tuple 
$$(\om_i)_{i\in[N]}\in 
\prod_{i=1}^N\Symp\big(X_i,\{X_{ij}\}_{j\in[N]-i}\big)$$
such that $\om_{i_1}|_{X_{i_1i_2}}\!=\!\om_{i_2}|_{X_{i_1i_2}}$ for all $i_1,i_2\!\in\![N]$.
\end{dfn}

\noindent
For an $N$-fold transverse configuration~$\X$ as in Definition~\ref{TransConf_dfn1}, 
we define the spaces \hbox{$X_{\eset}\!\supset\!X_{\prt}$} as in~\eref{Xesetdfn_e}
and~\eref{Xprtdfn_e}.
If in addition $X_{ij}$ is a closed submanifold of~$X_i$ of codimension~2
for all $i,j\!\in\![N]$ distinct,
let $\Symp(\X)$ denote the space of all symplectic structures 
on $\X$ and
$$\Symp^+\big(\X\big)= \Symp\big(\X\big)
\cap \prod_{i=1}^N\Symp^+\big(X_i,\{X_{ij}\}_{j\in[N]-i}\big)\,.$$
Thus, if $(\om_i)_{i\in[N]}$ is an element of $\Symp^+(\X)$,
then $\{X_{ij}\}_{j\in[N]-i}$ is an SC symplectic divisor in $(X_i,\om_i)$
for each $i\!\in\![N]$.

\begin{dfn}\label{SCC_dfn}
Let $N\!\in\!\Z^+$.
An \sf{$N$-fold simple crossings} (or \sf{SC}) \sf{symplectic configuration} 
is a~tuple 
\BE{SCCdfn_e}\X=\big((X_I)_{I\in\cP^*(N)},(\om_i)_{i\in[N]}\big)\EE
such that $\{X_I\}_{I\in\cP^*(N)}$ is an $N$-fold transverse configuration,
$X_{ij}$ is a closed submanifold of~$X_i$ of codimension~2
for all $i,j\!\in\![N]$ distinct, and
$(\om_i)_{i\in[N]}\in\Symp^+(\X)$.
The \sf{SC symplectic variety associated~to} such a tuple~$\X$ 
is the pair~$(X_{\eset},(\om_i)_{i\in[N]})$.
\end{dfn}

\begin{figure}
\begin{pspicture}(-3,-2)(11,2)
\psset{unit=.3cm}
\psline[linewidth=.1](15,-2)(22,-2)\psline[linewidth=.1](15,-2)(15,5)
\psline[linewidth=.1](15,-2)(10.5,-6.5)\pscircle*(15,-2){.3}
\rput(19.5,2.5){\sm{$\wh\P^2$}}\rput(11.5,-1){\sm{$\wh\P^2$}}\rput(17,-5){\sm{$\wh\P^2$}}
\rput(21.5,-2.9){\sm{$E$}}\rput(21.5,-1.1){\sm{$\bar{L}$}}
\rput(14.2,4.5){\sm{$\bar{L}$}}\rput(15.7,4.5){\sm{$E$}}
\rput(10,-5.7){\sm{$E$}}\rput(12,-6.1){\sm{$\bar{L}$}}
\rput(15.8,-1.2){\sm{$P$}}
\end{pspicture}
\caption{A 3-fold NC variety}
\label{P2cut_fig}
\end{figure}
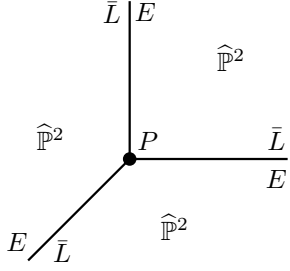

\noindent
A two-dimensional 3-fold SC configuration and associated NC variety
are shown in Figure~\ref{P2cut_fig}.
In this figure, $\wh\P^2$ denotes the blowup of 
the complex projective space~$\P^2$ at a point~$p$ and
$E,\ov{L}\!\subset\!\wh\P^2$ are the exceptional divisor and 
the proper transform of a line through~$p$, respectively.
In particular, we take 
$$X_{i,i-1}=\ov{L}\subset X_i=\wh\P^2, \quad 
X_{i,i+1}=E\subset X_i=\wh\P^2 \qquad\forall~i\in\Z_3\approx\{1,2,3\}$$
and choose an identification $\ov{L}\!\approx\!E$.

\subsection{Regularizations}
\label{SCDregul_subs}

\noindent
We next recall the notions of regularizations for a submanifold $V\!\subset\!X$, 
a symplectic submanifold with a split normal bundle,
a transverse collection $\{V_i\}_{i\in S}$ of submanifolds of
a manifold~$X$  with a symplectic structure~$\om$ on $(X,\{V_i\}_{i\in S})$,
and an SC symplectic configuration~$\X$. 
A regularization in the sense of Definition~\ref{TransCollregul_dfn}\ref{sympregul_it} 
for $\{V_i\}_{i\in S}$ in~$(X,\om)$
symplectically models a neighborhood of $x\!\in\!V_I$ in~$X$ on a neighborhood
of the zero section~$V_I$ in the normal bundle~$\cN_XV_I$ split as in~\eref{cNorient_e2}
with a standardized symplectic form.
A regularization for~$\X$ in the sense of Definition~\ref{TransConfregul_dfn}\ref{SCCreg_it} 
is a compatible collection of regularizations for the collections $\{X_{ij}\}_{j\in[N]-i}$
of submanifolds of~$X_i$.\\

\noindent
If $B$ is a manifold, possibly with boundary,
we call a family $(\om_t)_{t\in B}$ of 2-forms on~$X$ \sf{smooth} 
if the 2-form~$\wt\om$ on~$B\!\times\!X$ given~by
$$\wt\om_{(t,x)}(v,w)=
\begin{cases}
\om_t|_x(v,w),&\hbox{if}~v,w\!\in\!T_xX;\\
0,&\hbox{if}~v\!\in\!T_tB;
\end{cases}$$
is smooth. 
Smoothness for families of other objects is defined similarly.\\

\noindent
For a vector bundle $\pi\!:\cN\!\lra\!V$, we denote by $\ze_{\cN}$ 
\sf{the radial vector field} on the total space of~$\cN$; it is given~by
$$\ze_{\cN}(v)=(v,v)\in\pi^*\cN =T\cN^{\ver} \lhra T\cN\,.$$
Let $\Om$ be a fiberwise 2-form on~$\cN\!\lra\!V$.
A connection~$\na$ on~$\cN$ induces a projection $T\cN\!\lra\!\pi^*\cN$ and thus
determines an extension~$\Om_{\na}$ of~$\Om$ to a 2-form on 
the total space of~$\cN$.
If $\om$ is a closed 2-form on~$V$, the 2-form
\BE{ombund_e2}
\wt\om \equiv \pi^*\om+\frac12\nd\io_{\ze_{\cN}}\Om_{\na}
\equiv \pi^*\om+\frac12\nd\big(\Om_{\na}(\ze_{\cN},\cdot)\big)\EE
on the total space of $\cN$ is then closed and
restricts to~$\Om$ on $\pi^*\cN\!=\!T\cN^{\ver}$.
If $\om$ is a symplectic form on~$V$ and $\Om$ is a fiberwise symplectic form on~$\cN$,
then~$\wt\om$ is a symplectic form on a neighborhood of~$V$ in~$\cN$.\\

\noindent
We call $\pi\!:(L,\rho,\na)\!\lra\!V$ a \sf{Hermitian line bundle} if
$V$ is a manifold, $L\!\lra\!V$ is a smooth complex line bundle,
$\rho$ is a Hermitian metric on~$L$, 
and $\na$ is a $\rho$-compatible connection on~$L$.
We use the same notation~$\rho$ to denote the square of the norm function on~$L$
and the Hermitian form on~$L$ which is $\C$-antilinear in the second input.
Thus,
$$\rho(v)\equiv\rho(v,v), \quad 
\rho(\fI v,w)=\fI\rho(v,w)=-\rho(v,\fI w) 
\qquad\forall~(v,w)\!\in\!L\!\times_V\!L.$$
Let $\rho^{\R}$  denote the real part of the form~$\rho$.\\

\noindent
A Riemannian metric on an oriented  real vector bundle \hbox{$L\!\lra\!V$} of rank~2
determines a complex structure on the fibers of~$V$.
A \sf{Hermitian structure} on an oriented  real vector bundle \hbox{$L\!\lra\!V$} of rank~2
is a pair $(\rho,\na)$ such that $(L,\rho,\na)$ is a Hermitian line bundle
with the complex structure~$\fI_{\rho}$ determined by the Riemannian metric~$\rho^{\R}$.
If $\Om$ is a fiberwise symplectic form on an oriented vector bundle \hbox{$L\!\lra\!V$} of rank~2,
an \sf{$\Om$-compatible Hermitian structure} on~$L$ is a Hermitian structure $(\rho,\na)$ on~$L$ 
such that
$\Om(\cdot,\fI_{\rho}\cdot)=\rho^{\R}(\cdot,\cdot)$.\\

\noindent
Let $(L_i,\rho_i,\na^{(i)})_{i\in I}$ be a finite collection of Hermitian line
bundles over~$V$.
If each $(\rho_i,\na^{(i)})$ is compatible with a fiberwise symplectic form~$\Om_i$ on~$L_i$
and
$$(\cN,\Om,\na)\equiv\bigoplus_{i\in I}\big(L_i,\Om_i,\na^{(i)}\big),$$
then the 2-form~\eref{ombund_e2} is given~by 
\BE{ombund_e}
\wt\om=\om_{(\rho_i,\na^{(i)})_{i\in I}}
\equiv  \pi^*\om+\frac12
\bigoplus_{i\in I} \pi_{I;i}^*\nd\big((\Om_i)_{\na^{(i)}}(\ze_{L_i},\cdot)\big),\EE
where $\pi_{I;i}\!:\cN\!\lra\!L_i$ is the component projection map.\\

\noindent
If in addition $\Psi\!:V'\!\lra\!V$ is a smooth map and 
$(L_i',\rho_i',\na'^{(i)})_{i\in I}$ is a finite collection of Hermitian line
bundles over~$V'$, we call a (fiberwise) vector bundle isomorphism
$$\wt\Psi\!: \bigoplus_{i\in I'}L'_i\lra \bigoplus_{i\in I}L_i$$
covering~$\Psi$ a \sf{product Hermitian isomorphism} if
$$\wt\Psi\!: (L_i',\rho_i',\na'^{(i)}) \lra 
\Psi^*(L_i,\rho_i,\na^{(i)})$$
is an isomorphism of Hermitian line bundles over~$V'$ for every $i\!\in\!I$.

\begin{dfn}\label{smreg_dfn}
Let $X$ be a manifold and $V\!\subset\!X$ be a  submanifold
with normal bundle $\cN_XV\!\lra\!V$. 
A \sf{regularization for~$V$ in~$X$} is a diffeomorphism $\Psi\!:\cN'\!\lra\!X$
from a neighborhood of~$V$ in~$\cN_XV$ onto a neighborhood of~$V$ in~$X$ such
that $\Psi(x)\!=\!x$ and the isomorphism
$$ \cN_XV|_x=T_x^{\ver}\cN_XV \lhra T_x\cN_XV
\stackrel{\nd_x\Psi}{\lra} T_xX\lra \frac{T_xX}{T_xV}\equiv\cN_XV|_x$$
is the identity for every $x\!\in\!V$.
\end{dfn}

\noindent
If $(X,\om)$ is a symplectic manifold and $V$ is a  symplectic submanifold in~$(X,\om)$,
then $\om$ induces a fiberwise symplectic form~$\om|_{\cN_XV}$ on
the normal bundle~$\cN_XV$ of~$V$ in~$X$ via the isomorphism
$$\pi_{\cN_XV}\!:
\cN_XV\equiv \frac{TX|_V}{TV}\approx TV^{\om}
\equiv \big\{v\!\in\!T_xX\!:\,x\!\in\!V,\,\om(v,w)\!=\!0~\forall\,w\!\in\!T_xV\big\}\,.$$
We denote the restriction of~$\om|_{\cN_XV}$ to a subbundle $L\!\subset\!\cN_XV$
by~$\om|_L$.

\begin{dfn}\label{sympreg1_dfn}
Let $X$ be a  manifold, $V\!\subset\!X$ be a  submanifold, and
$$\cN_XV=\bigoplus_{i\in I}L_i$$
be a fixed splitting into oriented rank~2 subbundles. 
If $\om$ is a symplectic form on~$X$ such that $V$ is a symplectic submanifold
and $\om|_{L_i}$ is nondegenerate for every $i\!\in\!I$, then
an \sf{$\om$-regularization for~$V$ in~$X$} is a tuple $((\rho_i,\na^{(i)})_{i\in I},\Psi)$, 
where $(\rho_i,\na^{(i)})$ is an $\om|_{L_i}$-compatible Hermitian structure on~$L_i$
for each $i\!\in\!I$ and $\Psi$ is a regularization for~$V$ in~$X$, such that 
$$\Psi^*\om=\om_{(\rho_i,\na^{(i)})_{i\in I}}\big|_{\Dom(\Psi)}.$$
\end{dfn}

\vspace{.1in}

\noindent
Suppose $\{V_i\}_{i\in S}$ is a transverse collection of codimension~2 submanifolds of~$X$.
For each $I\!\subset\!S$, the last isomorphism in~\eref{cNorient_e2} provides 
a natural decomposition 
$$\pi_I\!:\cN_XV_I\!=\!\bigoplus_{i\in I}\cN_{V_{I-i}}V_I  \lra V_I$$
of the normal bundle of~$V_I$ in~$X$ into oriented rank~2 subbundles. 
We take this decomposition as given for the purposes of applying Definition~\ref{sympreg1_dfn}.
If in addition $I'\!\subset\!I$, let
\BE{cNIIprdfn_e}\pi_{I;I'}\!:\cN_{I;I'}\equiv 
\bigoplus_{i\in I-I'}\!\!\!\cN_{V_{I-i}}V_I=\cN_{V_{I'}}V_I\lra V_I\,.\EE
There are canonical identifications
\BE{cNtot_e}\cN_{I;I-I'}=\cN_XV_{I'}|_{V_I}, \quad
\cN_XV_I=\pi_{I;I'}^*\cN_{I;I-I'}=\pi_{I;I'}^*\cN_XV_{I'}
\qquad\forall~I'\!\subset\!I\!\subset\![N].\EE
The first equality in the second statement above
is used in particular in~\eref{overlap_e}. 

\begin{dfn}\label{TransCollReg_dfn}
Let $X$ be a manifold and $\{V_i\}_{i\in S}$ be a transverse collection 
of submanifolds of~$X$.
A \sf{system of regularizations for}  $\{V_i\}_{i\in S}$ in~$X$ is a~tuple 
$(\Psi_I)_{I\subset S}$, where $\Psi_I$ is a regularization for~$V_I$ in~$X$
in the sense of Definition~\ref{smreg_dfn}, such~that
\BE{Psikk_e}
\Psi_I\big(\cN_{I;I'}\!\cap\!\Dom(\Psi_I)\big)=V_{I'}\!\cap\!\Im(\Psi_I)\EE
for all $I'\!\subset\!I\!\subset\!S$.
\end{dfn}

\noindent
Given a system of regularizations as in Definition~\ref{TransCollReg_dfn}
and $I'\!\subset\!I\!\subset\!S$, let
$$\cN_{I;I'}' = \cN_{I;I'}\!\cap\!\Dom(\Psi_I), \qquad
\Psi_{I;I'}\equiv \Psi_I\big|_{\cN_{I;I'}'}\!: \cN_{I;I'}'\lra V_{I'}\,.$$
The map $\Psi_{I;I'}$ is a regularization for $V_I$ in~$V_{I'}$.
As explained in \cite[Section~2.2]{SympDivConf}, $\Psi_I$ determines
an isomorphism
\BE{wtPsiIIdfn_e}
 \fD\Psi_{I;I'}\!:  \pi_{I;I'}^*\cN_{I;I-I'}\big|_{\cN_{I;I'}'}
\lra \cN_XV_{I'}\big|_{V_{I'}\cap\Im(\Psi_I)}\EE
of vector bundles covering~$\Psi_{I;I'}$ and
respecting the natural decompositions of 
$\cN_{I;I-I'}\!=\!\cN_XV_{I'}|_{V_I}$ and $\cN_XV_{I'}$.
By the last assumption in Definition~\ref{smreg_dfn}, 
$$\fD\Psi_{I;I'}\big|_{\pi_{I;I'}^*\cN_{I;I-I'}|_{V_I}}\!=\!\id\!:\,
\cN_{I;I-I'}\lra \cN_XV_{I'}|_{V_I}$$
under the canonical identification of $\cN_{I;I-I'}$ with $\cN_XV_{I'}|_{V_I}$.

\begin{dfn}\label{TransCollregul_dfn}
Let $X$ be a manifold and  $\{V_i\}_{i\in S}$ be a transverse 
collection of submanifolds of~$X$. 
\begin{enumerate}[label=(\arabic*),leftmargin=*]

\item A \sf{regularization for $\{V_i\}_{i\in S}$ in~$X$} 
is a system of regularizations $(\Psi_I)_{I\subset S}$ 
for $\{V_i\}_{i\in S}$ in~$X$ such~that
\BE{overlap_e}
\fD\Psi_{I;I'}\big(\Dom(\Psi_I)\big)
=\Dom(\Psi_{I'})\big|_{V_{I'}\cap\Im(\Psi_I)}, \quad
\Psi_I=\Psi_{I'}\circ\fD\Psi_{I;I'}|_{\Dom(\Psi_I)}\EE
for all $I'\!\subset\!I\!\subset\!S$.

\item\label{sympregul_it} 
Suppose in addition that $V_i$ is a closed submanifold of~$X$ of codimension~2 for every $i\!\in\!S$
and \hbox{$\om\!\in\!\Symp(X,\{V_i\}_{i\in S})$}. 
An \sf{$\om$-regularization for $\{V_i\}_{i\in S}$ in~$X$}  is a~tuple
$$(\cR_I)_{I\subset S} \equiv  
\big((\rho_{I;i},\na^{(I;i)})_{i\in I},\Psi_I\big)_{I\subset S}$$
such that $\cR_I$ is an $\om$-regularization for~$V_I$ in~$X$ for each $I\!\subset\!S$,
$(\Psi_I)_{I\subset S}$ is a regularization for $\{V_i\}_{i\in S}$ in~$X$,
and the induced vector bundle isomorphisms~\eref{wtPsiIIdfn_e}
are product Hermitian isomorphisms for all $I'\!\subset\!I\!\subset\!S$.

\end{enumerate}
\end{dfn}

\vspace{.1in}

\noindent
If $(\cR_I)_{I\subset S}$ is a regularization for $\{V_i\}_{i\in S}$ in~$X$, 
then 
\BE{SCDcons_e2}
\Psi_{I;I''}=\Psi_{I';I''}\circ\fD\Psi_{I;I'}\big|_{\cN_{I;I''}'}\,, 
\quad
\fD\Psi_{I;I''}=\fD\Psi_{I';I''}\circ
\fD\Psi_{I;I'}\big|_{\pi_{I;I''}^*\cN_{I;I-I''}|_{\cN_{I;I''}'}}\EE
for all $I''\!\subset\!I'\!\subset\!I\!\subset\!S$.\\

\noindent
Suppose $\{X_I\}_{I\in\cP^*(N)}$ is a transverse configuration in the sense 
of Definition~\ref{TransConf_dfn1}.
For each $I\!\in\!\cP^*(N)$ with $|I|\!\ge\!2$, let
$$\pi_I\!:\cN X_I\equiv \bigoplus_{i\in I}\cN_{X_{I-i}}X_I\lra X_I\,.$$
If in addition $I'\!\subset\!I$, let
$$\pi_{I;I'}\!:\cN_{I;I'}\equiv 
  \bigoplus_{i\in I-I'}\!\!\!\cN_{X_{I-i}}X_I\lra X_I\,.$$
By the last isomorphism in~\eref{cNorient_e2} with $X\!=\!X_i$ for any $i\!\in\!I'$ and 
$\{V_j\}_{j\in S}\!=\!\{X_{ij}\}_{j\in[N]-i}$, 
$$	\cN_{I;I'}=\cN_{X_{I'}}X_I \qquad\forall~I'\!\subset\!I\!\subset\![N],~I'\!\neq\!\eset.$$
Similarly to~\eref{cNtot_e}, there are canonical identifications
$$\cN_{I;I-I'}=\cN X_{I'}|_{X_I}, \quad
\cN X_I=\pi_{I;I'}^*\cN_{I;I-I'}=\pi_{I;I'}^*\cN X_{I'}
\qquad\forall~I'\!\subset\!I\!\subset\![N];$$
the first and last identities above hold if $|I'|\!\ge\!2$.

\begin{dfn}\label{TransConfregul_dfn}
Let $N\!\in\!\Z^+$ and $\X\!=\!\{X_I\}_{I\in\cP^*(N)}$ be a transverse configuration
as in Definition~\ref{TransConf_dfn1}.
\begin{enumerate}[label=(\arabic*),leftmargin=*]

\item A \sf{regularization for $\X$} is a tuple 
$(\Psi_{I;i})_{i\in I\subset[N]}$, 
where for each $i\!\in\!I$ the tuple $(\Psi_{I;i})_{I\in\cP_i(N)}$
is a regularization for $\{X_{ij}\}_{j\in[N]-i}$ in~$X_i$ in the sense of
Definition~\ref{TransCollregul_dfn}, such that
\BE{SCCregCond_e0}
\Psi_{I;i_1}\big|_{\cN_{I;i_1i_2}\cap\Dom(\Psi_{I;i_1})}
=\Psi_{I;i_2}\big|_{\cN_{I;i_1i_2}\cap\Dom(\Psi_{I;i_2})}\EE
for all $i_1,i_2\!\in\!I\!\subset\![N]$.

\item\label{SCCreg_it} 
Suppose in addition that $X_{ij}$ is a closed submanifold of~$X_i$ of codimension~2
for all $i,j\!\in\![N]$ distinct and $(\om_i)_{i\in[N]}\!\in\!\Symp^+(\X)$. 
An \sf{$(\om_i)_{i\in[N]}$-regularization for~$\X$} is a~tuple
\BE{SCCregdfn_e0}
\fR\equiv (\cR_I)_{I\in\cP^*(N)} \equiv
\big(\rho_{I;i},\na^{(I;i)},\Psi_{I;i}\big)_{i\in I\subset[N]}\EE
such that $(\Psi_{I;i})_{i\in I\subset[N]}$ is a  regularization 
for $\X$ and for each $i\!\in\![N]$ the~tuple
$$\big((\rho_{I;j},\na^{(I;j)})_{j\in I-i},\Psi_{I;i}\big)_{I\in\cP_i(N)}$$
is an $\om_i$-regularization for $\{X_{ij}\}_{j\in[N]-i}$ in $X_i$
in the sense of Definition~\ref{TransCollregul_dfn}\ref{sympregul_it}.

\end{enumerate}
\end{dfn}

\vspace{.1in}

\noindent
For a smooth family $(\om_t)_{t\in B}$ of symplectic forms 
in $\Symp(X,\{V_i\}_{i\in S})$,
Definition~\ref{TransCollregul_dfn}\ref{sympregul_it} naturally extends 
to provide a notion of 
 \sf{$(\om_t)_{t\in B}$-family of regularizations for $\{V_i\}_{i\in S}$ in~$X$};
see \cite[Definition~2.12(2)]{SympDivConf}.
For a smooth family of symplectic structures $(\om_{t;i})_{t\in B,i\in[N]}$ 
on~$\X$, 
Definition~\ref{TransConfregul_dfn}\ref{SCCreg_it} similarly extends 
to provide a notion of \sf{$(\om_{t;i})_{t\in B,i\in[N]}$-family of regularizations for~$\X$};
see \cite[Definition~2.15(2)]{SympDivConf}.
The first extension topologizes the set $\Aux(X,\{V_i\}_{i\in S})$ 
of pairs $(\om,(\cR_I)_{I\subset S})$ consisting of a symplectic structure~$\om$ 
on $\{V_i\}_{i\in S}$ in~$X$
and an $\om$-regularization $(\cR_I)_{I\subset S}$ for $\{V_i\}_{i\in S}$ in~$X$.
The second extension topologizes the set $\Aux(\X)$ of pairs $((\om_i)_{i\in[N]},\fR)$ 
consisting of a symplectic structure $(\om_i)_{i\in[N]}$ on~$\X$
and an $(\om_i)_{i\in[N]}$-regularization~$\fR$ for~$\X$.\\

\noindent
The existence of regularizations requires the symplectic divisors $V_i\!\subset\!X$
and $X_{ij}\!\subset\!X_i$ to meet $\om$-orthogonally and $\om_i$-orthogonally,
respectively, which is rarely the case.
However, \cite[Theorems~2.13,2.17]{SympDivConf} imply~that the projections 
\BE{AuxtoSympNC_e}\begin{aligned}
\Aux\big(X,\{V_i\}_{i\in S}\big) &\lra \Symp^+\big(X,\{V_i\}_{i\in S}\big),
& (\om,\fR)&\lra\om,\\
\Aux(\X) &\lra \Symp^+(\X),
& \big((\om_i)_{i\in[N]},\fR\big)&\lra(\om_i)_{i\in[N]},
\end{aligned}\EE
are weak homotopy equivalences and thus ensure a virtual kind of existence
whenever $\{V_i\}_{i\in S}$ is an SC symplectic divisor in
the sense of Definition~\ref{SCD_dfn} and
$(X_{\eset},(\om_i)_{i\in[N]})$ is an NC symplectic variety 
in the sense of Definition~\ref{SCC_dfn}.

\section{Normal crossings symplectic divisors}
\label{NCD_sec}

\noindent
NC~divisors are spaces that are locally SC~divisors.
This local perspective makes it fairly straightforward to define NC divisors and notions of regularizations.
NC~divisors can also be viewed as analogues of SC~divisors for immersions instead of submanifolds.
This global perspective leads to a more succinct notion of regularizations for 
NC~divisors and fits better with some applications.

\subsection{Local perspective}
\label{NCDloc_subs}

\noindent
Definitions~\ref{NCD_dfn} and~\ref{NCDregul_dfn} below  
locally correspond to Definitions~\ref{SCD_dfn} and~\ref{TransCollregul_dfn}\ref{sympregul_it}, respectively.

\begin{dfn}\label{NCsubsp_dfn}
Let $X$ be a manifold.
A subspace $V\!\subset\!X$ is a \sf{normal crossings} (or \sf{NC}) \sf{divisor} 
if for every $x\!\in\!X$ there exist an open neighborhood~$U$ of~$x$ in~$X$ and 
a finite transverse collection $\{V_i\}_{i\in S}$ of closed submanifolds  of~$U$
of codimension~2 such~that 
$$V\cap U= \bigcup_{i\in S}\!V_i\,.$$
\end{dfn}

\begin{dfn}\label{NCD_dfn}
Let $(X,\om)$ be a symplectic manifold.
A subspace  $V\!\subset\!X$ is an \sf{NC symplectic divisor in~$(X,\om)$} if 
for every $x\!\in\!X$ there exist $U$ and $\{V_i\}_{i\in S}$ as
in Definition~\ref{NCsubsp_dfn} such that
 $\{V_i\}_{i\in S}$ is an SC symplectic divisor in~$(U,\om|_U)$.
\end{dfn}

\noindent
By Definition~\ref{NCsubsp_dfn}, 
every NC divisor $V\!\subset\!X$ is a closed subspace.
So is its \sf{singular locus} $V_{\prt}\!\subset\!V$ consisting of the points $x\!\in\!V$
such that there exist $U$ and $\{V_i\}_{i\in S}$ as in Definition~\ref{NCsubsp_dfn} and 
$I\!\subset\![N]$ with $|I|\!=\!2$ and $V_I\!\ni\!x$.
For an NC divisor $V\!\subset\!X$, denote by $\Symp^+(X,V)$ the space 
of all symplectic forms~$\om$ on~$X$ so that $V$ is an NC symplectic divisor in~$(X,\om)$.
An SC symplectic divisor in the sense of Definition~\ref{SCD_dfn}
is an NC symplectic divisor, as we can take $U\!=\!X$ for every $x\!\in\!X$.\\

\noindent
Let $V\!\subset\!X$ be an NC divisor.
For each \sf{chart} $(U,\{V_i\}_{i\in S})$ as in Definition~\ref{NCsubsp_dfn}
and each $x\!\in\!U$, let
$$S_x=\big\{i\!\in\!S\!:\,x\!\in\!V_i\big\}\,.$$
If $(U',\{V_i'\}_{i\in S'})$ is another chart for $V$ in~$X$ and $x\!\in\!U\!\cap\!U'$, 
there exist a neighborhood~$U_x$ of~$x$ in $U\!\cap\!U'$ and a bijection
\BE{NCDoverlap_e0}h_x\!:S_x\lra S_x'  \qquad\hbox{s.t.}\quad
V_i\!\cap\!U_x=V_{h_x(i)}'\!\cap\!U_x~~\forall\,i\!\in\!S_x\,.\EE
We also denote by $h_x$ the induced bijection $\cP(S_x)\!\lra\!\cP(S_x')$.
By~\eref{NCDoverlap_e0}, 
$$\cN_XV_I\big|_{V_I\cap U_x}
=\cN_XV_{I'}'\big|_{V_{I'}'\cap U_x}
\qquad\forall~I\subset S_x,~I'\!=\!h_x(I)\,.$$
Suppose
$$\big(\cR_I\big)_{I\subset S} \equiv  
\big((\rho_{I;i},\na^{(I;i)})_{i\in I},\Psi_I\big)_{I\subset S}, \quad
\big(\cR_I'\big)_{I\subset S'} \equiv  
\big((\rho_{I;i}',\na'^{(I;i)})_{i\in I},\Psi_I'\big)_{I\subset S'}$$
are an $\om|_U$-regularization for $\{V_i\}_{i\in S}$ in~$U$ 
and an $\om|_{U'}$-regularization for $\{V_i'\}_{i\in S'}$ in~$U'$. 
We define 
$$\big(\cR_y\big)_{I\subset S}\cong_X\big(\cR_I'\big)_{I\subset S'}$$
if for every $x\!\in\!U\!\cap\!U'$ there exist $U_x$ and~$h_x$ as in~\eref{NCDoverlap_e0} 
such~that 
\begin{gather*}
\big(\rho_{I;i},\na^{(I;i)}\big)\big|_{V_I\cap U_x}
=\big(\rho_{I';i'}',\na'^{(I';i')}\big)\big|_{V_{I'}\cap U_x},\\
\Psi_I=\Psi_{I'}'
\quad\hbox{on}\quad
\Dom(\Psi_I)|_{V_I\cap U_x}\cap \Dom(\Psi_{I'}')|_{V_{I'}\cap U_x}
\end{gather*}
for all $i\!\in\!I\!\subset\!S_x$, 
$i'\!\equiv\!h_x(i)\in I'\!\equiv\!h_x(I)$.

\begin{dfn}\label{NCDregul_dfn}
Let $X$ be a manifold, $V\!\subset\!X$ be an NC divisor,
and $(U_y,\{V_{y;i}\}_{i\in S_y})_{y\in\cA}$ be a collection of charts for~$V$ in~$X$
as in Definition~\ref{NCsubsp_dfn}.
\begin{enumerate}[label=(\arabic*),leftmargin=*]

\item\label{NCDregul_it1} 
If $\om\!\in\!\Symp^+(X,V)$,  an \sf{$\om$-regularization for $V$ in~$X$} 
is a~collection
\BE{NCDregul_e1} \fR\equiv (\cR_{y;I})_{y\in\cA,I\subset S_y} \equiv  
\big((\rho_{y;I;i},\na^{(y;I;i)})_{i\in I},\Psi_{y;I}\big)_{y\in\cA,I\subset S_y}\EE
such that $(\cR_{y;I})_{I\subset S_y}$ is an $\om|_{U_y}$-regularization 
for $\{V_{y;i}\}_{i\in S_y}$ in~$U_y$ for each $y\!\in\!\cA$  and 
\BE{NCDregul_e2}\big(\cR_{y;I}\big)_{I\subset S_y}\cong_X\big(\cR_{y';I}\big)_{I\subset S_{y'}}
\qquad \forall\,y,y'\!\in\!\cA\,.\EE

\item\label{NCDregul_it2}  If $B$ is a manifold, possibly with boundary, and
$(\om_t)_{t\in B}$ is a smooth family of symplectic forms in $\Symp^+(X,V)$,
 an \sf{$(\om_t)_{t\in B}$-family of regularizations for $V$ in~$X$}
is a smooth family of~tuples $(\cR_{t;y;I})_{t\in B,y\in\cA,I\subset S_y}$
such that $(\cR_{t;y;I})_{y\in\cA,I\subset S_y}$ 
is an $\om_t$-regularization for $V$ in~$X$ for each $t\!\in\!B$  
and $(\cR_{t;y;I})_{t\in B,I\subset S_y}$
is an $(\om_t|_{U_y})_{t\in B}$-family of regularizations 
for $\{V_{y;i}\}_{i\in S_y}$ in~$U_y$  for each $y\!\in\!\cA$.

\end{enumerate}
\end{dfn}

\vspace{.1in}

\noindent
Suppose $X$, $V$, and $(U_y,\{V_{y;i}\}_{i\in S_y})_{y\in\cA}$ are
as in Definition~\ref{NCDregul_dfn} and $(\om_t)_{t\in B}$ is a family 
of symplectic forms in $\Symp^+(X,V)$.
We define two $(\om_t)_{t\in B}$-families of regularizations  
for~$V$ in~$X$  to be \sf{equivalent},
$$\big(\cR_{t;y;I}^{(1)}\big)_{t\in B,y\in\cA,I\subset S_y}\cong
\big(\cR_{t;y;I}^{(2)}\big)_{t\in B,y\in\cA,I\subset S_y},$$
if they agree on the level of germs.
This means that for every $y\!\in\!\cA$ the families 
$$(\cR_{t;y;I}^{(1)}\big)_{t\in B,I\subset S_y} \qquad\hbox{and}\qquad
\big(\cR_{t;y;I}^{(2)}\big)_{t\in B,I\subset S_y}$$
of regularizations for $\{V_{y;i}\}_{i\in S_y}$ in~$U_y$
agree on the level of germs as formally defined  just before
\cite[Theorem~2.13]{SympDivConf}.

\begin{thm}\label{NCD_thm}
Let $X$, $V$, and $(U_y,\{V_{y;i}\}_{i\in S_y})_{y\in\cA}$ be
as in Definition~\ref{NCDregul_dfn} and
$X^*\!\subset\!X$ be an open subset such that $\ov{X^*}\!\cap\!V_{\prt}\!=\!\eset$.
Suppose
\begin{enumerate}[label=$\bullet$,leftmargin=*]
\item $B$ is a compact manifold, possibly with boundary, 

\item $N(\prt B),N'(\prt B)$ are neighborhoods of $\prt B$ in~$B$
such that $\ov{N'(\prt B)}\!\subset\!N(\prt B)$,

\item $(\om_t)_{t\in B}$ is a  family of symplectic forms 
in $\Symp^+(X,V)$,

\item  $(\cR_{t;y;I})_{t\in N(\prt B),y\in\cA,I\subset S_y}$ is 
an $(\om_t)_{t\in N(\prt B)}$-family of regularizations for $V$ in~$X$.
\end{enumerate}
Then there exist a smooth family $(\mu_{t,\tau})_{t\in B,\tau\in\bI}$ of
1-forms on~$X$ such~that 
\begin{gather*}
\mu_{t,0}=0, \quad 
\supp\big(\mu_{\cdot,\tau}\big)\subset 
\big(B\!-\!N'(\prt B)\big)\!\times\!(X\!-\!X^*)
\qquad \forall~t\!\in\!B,\,\tau\!\in\!\bI,\\
\big(\om_{t,\tau}\equiv\om_t\!+\!\nd\mu_{t,\tau}\big)
\in \Symp^+\big(X,V\big) \qquad \forall~t\!\in\!B,\,\tau\!\in\!\bI,
\end{gather*}
and  an $(\om_{t,1})_{t\in B}$-family 
$(\wt\cR_{t;y;I})_{t\in B,y\in\cA,I\subset S_y}$ of regularizations 
 for $V$ in~$X$ such~that
$$(\wt\cR_{t;y;I})_{t\in N'(\prt B),y\in\cA,I\subset S_y} \cong 
(\cR_{t;y;I})_{t\in N'(\prt B),y\in\cA,I\subset S_y}\,.$$
\end{thm}

\vspace{.1in}

\noindent
Definition~\ref{NCDregul_dfn}\ref{NCDregul_it2} topologizes the set $\Aux(X,V)$
of pairs $(\om,\fR)$ consisting of a symplectic structure~$\om$ 
on an NC divisor~$V$ in~$X$
and an $\om$-regularization $\fR$ for~$V$ in~$X$.
Theorem~\ref{NCD_thm} above, which is the direct analogue of \cite[Theorem~2.13]{SympDivConf}
for arbitrary NC divisors, implies that the first projection in~\eref{AuxtoSympNC_e}
is a weak homotopy equivalence in this general setting as well.
Similarly to the situation with \cite[Theorems~2.13,2.17]{SympDivConf},
Theorem~\ref{NCD_thm} is implied by Theorem~\ref{NCC_thm};
see the paragraph after \cite[Theorem~2.13]{SympDivConf} and Example~\ref{NCDvsC_eg}.

\subsection{Global perspective}
\label{NCDgl_subs}

\noindent
We now give an equivalent global description of the notions introduced in
Section~\ref{NCDloc_subs}.
We do so by viewing an NC symplectic divisor in the sense of 
Definition~\ref{NCD_dfn} as the image of a transverse immersion~$\io$ with certain properties.\\

\noindent
For any map $\io\!:\wt{V}\!\lra\!X$ and $k\!\in\!\Z^{\geq 0}$, let
\BE{tViotak_e}
\wt{V}_{\io}^{(k)}=\big\{(x,\wt{v}_1,\ldots,\wt{v}_k)\!\in\!X\!
\times\!(\wt{V}^k\!-\!\De_{\wt{V}}^{(k)})\!:\,\io(\wt{v}_i)\!=\!x~\forall\,i\!\in\![k]\big\},\EE
where $\De_{\wt{V}}^{(k)}\!\subset\!\wt{V}^k$ is the big diagonal
(at least two of the coordinates are the same).
Define
\begin{gather}\label{iotak_e}  
\io_k\!:\wt{V}_{\io}^{(k)}\lra  X, \qquad \io_k(x,\wt{v}_1,\ldots,\wt{v}_k)=x,\\
\label{Xiotak_e}
V_{\io}^{(k)}=\io_k(\wt{V}_{\io}^{(k)})
=\big\{x\!\in\!X\!:\,\big|\io^{-1}(x)\big|\!\ge\!k\big\}.
\end{gather}
For example,
$$\wt{V}_{\io}^{(0)},V_{\io}^{(0)}=X, \qquad
\wt{V}_{\io}^{(1)}\approx\wt{V}, \qquad  V_{\io}^{(1)}=\io(\wt{V}).$$

\vspace{.2in}

\noindent
For $k',k\!\in\!\Z^{\ge0}$ and $i\!\in\!\Z^+$ with $i,k'\!\le\!k$, define
\begin{alignat}{2}
\label{whiokk_e}
\wt\io_{k;k'}\!:\wt{V}_{\io}^{(k)}&\lra\wt{V}_{\io}^{(k')}, &\quad 
\wt\io_{k;k'}(x,\wt{v}_1,\ldots,\wt{v}_k)&=(x,\wt{v}_1,\ldots,\wt{v}_{k'}),\\
\label{iokcjdfn_e}
\wt\io_{k;k-1}^{(i)}\!: \wt{V}_{\io}^{(k)}&\lra\wt{V}_{\io}^{(k-1)}, &\quad
\wt\io_{k;k-1}^{(i)}(x,\wt{v}_1,\ldots,\wt{v}_k)&=(x,\wt{v}_1,\ldots,\wt{v}_{i-1},\wt{v}_{i+1},\ldots,\wt{v}_k),\\
\label{iokjdfn_e}
\wt\io_k^{(i)}\!:  \wt{V}_{\io}^{(k)} &\lra \wt{V}, &\quad 
\wt\io_k^{(i)}(x,\wt{v}_1,\ldots,\wt{v}_k)&=\wt{v}_i.
\end{alignat}
For example, 
\begin{alignat*}{2}
\wt\io_{k;k'}\!=\!\wt\io_{k'+1;k'}^{(k'+1)}\!\circ\!\ldots\!\circ\!\wt\io_{k;k-1}^{(k)}
\!&:\wt{V}_{\io}^{(k)}\lra\wt{V}_{\io}^{(k')}, &\qquad
\wt\io_{k;1}\!\approx\!\wt\io_k^{(1)}\!&:
\wt{V}_{\io}^{(k)}\lra\wt{V}_{\io}^{(1)}\!\approx\!\wt{V}, \\
\wt\io_{k;0}\!=\!\io_k\!&:\wt{V}_{\io}^{(k)}\lra\wt{V}_{\io}^{(0)}\!=\!X, &\qquad
\wt\io_{1;0}\!\approx\!\io\!&:\wt{V}_{\io}^{(1)}\!\approx\!\wt{V}\lra X.
\end{alignat*}
We define an $\bS_k$-action on $\wt{V}_{\io}^{(k)}$ by requiring that 
\BE{SkVk_e} \wt\io_k^{(i)}=\wt\io_k^{(\si(i))}\!\circ\!\si\!: 
\wt{V}_{\io}^{(k)} \lra \wt{V}\EE
for all $\si\!\in\!\bS_k$ and $i\!\in\![k]$.
The diagrams
\BE{Vkdiag_e}\begin{split}
\xymatrix{\wt{V}_{\io}^{(k)~{}} \ar[rr]^{\wt\io_k^{(i)}} \ar[d]^{\wt\io_{k;k-1}^{(i)}}  
\ar@/_2pc/[dd]_{\wt\io_{k;k'}}  \ar@/^1pc/[rrdd]^{\io_k}
&& \wt{V} \ar[dd]^{\io}  \\
\wt{V}_{\io}^{(k-1)} \ar[rrd]^<<<<<<{\!\io_{k-1}}  
\ar@{-->}[d]^{\wt\io_{k-1;k'}} &&  \\
\wt{V}_{\io}^{(k')}  \ar[rr]^{\io_{k'}}& &X}
\end{split} \hspace{.5in}
\begin{split}
\xymatrix{\wt{V}_{\io}^{(k)}\ar[rr]^{\si} \ar[d]_{\wt\io_{k;k-1}^{(i)}}&&
\wt{V}_{\io}^{(k)}\ar[d]^{\wt\io_{k;k-1}^{(\si(i))}} \\
\wt{V}_{\io}^{(k-1)} \ar[rr]^{\si_i}&&  \wt{V}_{\io}^{(k-1)}}
\end{split}\EE
of solid arrows then commute; the entire first diagram commutes if $i\!>\!k'$.\\ 

\noindent
A smooth map $\io\!:\wt{V}\lra\!X$ is an \sf{immersion} if
the differential $\nd_x\io$ of~$\io$ at~$x$ is injective for all $x\!\in\!\wt{V}$.
This implies~that 
$$\codim\,\io\equiv\dim X-\dim V\ge0.$$
Such a map has a well-defined normal bundle,
$$\cN\io\equiv \io^*TX\big/\Im(\nd\io)\lra \wt{V}\,.$$
If $\io$ is a closed immersion, then the subsets $V_{\io}^{(k)}\!\subset\!X$
and $\wt{V}_{\io}^{(k)}\!\subset\!X\!\times\!\wt{V}^k$ are closed.\\

\noindent
An immersion $\io\!:\wt{V}\lra\!X$  is \sf{transverse} if the homomorphism
\BE{TransImmVerHom_e}
T_xX\oplus\bigoplus_{i=1}^k T_{\wt{v}_i}\wt{V}\lra \bigoplus_{i=1}^kT_xX, \quad
\big(w,(w_i)_{i\in[k]}\big)\lra \big(w\!+\!\nd_{\wt{v}_i}\io(w_i)\big)_{i\in[k]}\,,\EE
is surjective for all $(x,\wt{v}_1,\ldots,\wt{v}_k)\!\in\!\wt{V}_{\io}^{(k)}$ and $k\!\in\!\Z^+$.
By the Inverse Function Theorem, in such a case
\begin{enumerate}[label=$\bullet$,leftmargin=*]

\item each $\wt{V}_{\io}^{(k)}$ is a submanifold of $X\!\times\!\wt{V}^k$,

\item\label{hatimerssion_it}  the maps $\wt\io_{k;k-1}$ in \eref{whiokk_e} and
the maps~\eref{iokcjdfn_e} are transverse immersions,

\item the homeomorphisms~$\si$ determined by the elements of~$\bS_k$ 
as in~\eref{SkVk_e} are diffeomorphisms.

\end{enumerate}
By the commutativity of the upper and middle triangles in the first diagram in~\eref{Vkdiag_e}, 
the inclusion of $\Im(\nd\io_k)$ into $\wt\io_k^{(i)*}\Im(\nd\io)$ and
the homomorphism~$\nd\io_{k-1}$  induce homomorphisms
\BE{iodecompmaps_e}  
\cN\io_k\lra \io_k^{(i)*}\cN\io, \quad
\cN\wt\io_{k;k-1}^{(i)}\lra \cN\io_k \qquad\forall~i\!\in\![k].\EE
By the Inverse Function Theorem, the resulting homomorphisms
\BE{ImmcNorient_e2}
\cN\io_k\lra \bigoplus_{i\in[k]}\!\wt\io_k^{(i)*}\cN\io \qquad\hbox{and}\qquad
\cN\wt\io_{k;k-1}^{(i)}\lra \wt\io_k^{(i)*}\cN\io~~~\forall\,i\!\in\![k]\EE
are isomorphisms; 
they correspond to the first two isomorphisms in~\eref{cNorient_e2} 
if $\wt{V}$ is the disjoint union of submanifolds $V_i\!\subset\!X$. 
For $\si\!\in\!\bS_k$ and $i\!\in\![k]$,
the homomorphisms~$\nd\si$ and~$\nd\si_i$ of the second diagram in~\eref{Vkdiag_e}
induces an isomorphism
\BE{Djsi_e}D_i\si\!:\cN\wt\io_{k;k-1}^{(i)}\lra \cN\wt\io_{k;k-1}^{(\si(i))}\EE
covering~$\si$.

\begin{lmm}\label{NCD_lmm}
Let $X$ be a manifold.
A subset $V\!\subset\!X$ is an NC divisor in the sense of Definition~\ref{NCsubsp_dfn}
if and only if $V$ is the image of a closed transverse immersion $\io\!:\wt{V}\!\lra\!X$
of codimension~2.
\end{lmm}

\begin{proof}
(1) Let $V\!\subset\!X$ be an NC divisor.
Choose a locally finite open cover $\{U_y\}_{y\in\cA'}$ of~$X$ 
with associated transverse collections $\{V_{y;i}\}_{i\in S_y}$
as in Definition~\ref{NCsubsp_dfn}.
Let 
$$\wt{V}=\bigg(\bigsqcup_{y\in\cA'}\bigsqcup_{i\in S_y}
\{(y,i)\}\!\times\!V_{y;i}\bigg)\Big/\!\!\sim\,,$$
where we identify $(y,i,x)$ with $(y',i',x)$ if there exists a neighborhood~$U$
of~$x$ in $U_y\!\cap\!U_{y'}$ such that $V_{y;i}\!\cap\!U\!=\!V_{y';i'}\!\cap\!U$.
The Hausdorffness of~$X$ implies the Hausdorffness of~$\wt{V}$.
The latter inherits a smooth structure from the smooth structures of
the submanifolds $V_{y;i}\!\subset\!X$ (which necessarily agree on the overlaps).
The smooth~map
$$\io\!: \wt{V}\lra X, \qquad \big[y,i,x]\lra x,$$
is then a well-defined closed transverse immersion of codimension~2.\\

\noindent
(2) Let $\io\!:\wt{V}\!\lra\!X$ be a closed transverse immersion 
of codimension~2.
Given $x\!\in\!X$,  let 
$$\io^{-1}(x)=\big\{\wt{v}_1,\ldots,\wt{v}_k\big\}.$$
By \cite[Proposition~1.35]{Warner} and the closedness of~$\io$,
there exist a neighborhood $U\!\subset\!X$ of~$x$ and  neighborhoods
$\wt{V}_i\!\subset\!\wt{V}$ of~$\wt{v}_i$ with $i\!\in\![k]$ such~that 
$$\io^{-1}(U)=\bigsqcup_{i=1}^k\wt{V}_i\subset\wt{V}$$ 
and $\io|_{\wt{V}_i}$ is an embedding for every $i\!\in\![k]$.
Then, $\{\io(\wt{V}_i)\}_{i\in[k]}$ 
is a finite transverse collection of closed submanifolds of~$U$
of codimension~2 such~that 
$$V\cap U= \bigcup_{i=1}^k\io(\wt{V}_i)\,.$$
Thus, $\io(\wt{V})$ is an NC divisor in~$X$.
\end{proof}

\noindent
If $V\!\subset\!X$ is the NC divisor associated with a closed transverse immersion 
of codimension~2 as in Lemma~\ref{NCD_lmm}, then $V_{\prt}\!=\!V_{\io}^{(2)}$.\\

\noindent
If $\io\!:\wt{V}\!\lra\!X$ is any immersion between oriented manifolds of even dimensions,
the short exact sequence of vector bundles
\BE{ImmcNorient_e1} 0\lra T\wt{V}\stackrel{\nd\io}\lra \io^*TX\lra \cN\io\lra 0\EE
over $\wt{V}$ induces an orientation on~$\cN\io$. 
If in addition $\io$ is a transverse immersion,
the orientation on~$\cN\io$ induced by the orientations of~$X$ and~$\wt{V}$ induces 
an orientation on~$\cN\io_k$ via the first isomorphism in~\eref{ImmcNorient_e2}.
The orientations of~$X$ and~$\cN\io_k$ then induce an orientation on~$\wt{V}_{\io}^{(k)}$
via the short exact sequence~\eref{ImmcNorient_e1} with $\io\!=\!\io_k$ for all $k\!\in\!\Z^+$,
which we call \sf{the intersection orientation of~$\wt{V}_{\io}^{(k)}$}.
For $k\!=\!1$, it agrees with the original orientation of~$\wt{V}$ under the canonical identification 
$\wt{V}_{\io}^{(1)}\!\approx\!\wt{V}$.\\

\noindent
Suppose $(X,\om)$ is a symplectic manifold.
If $\io\!:\wt{V}\!\lra\!X$ is a transverse immersion such that $\io_k^*\om$ 
is a symplectic form on $\wt{V}_{\io}^{(k)}$ for all $k\!\in\!\Z^+$, then
each $\wt{V}_{\io}^{(k)}$ carries an orientation induced by $\io_k^*\om$,
which we  call the $\om$-orientation.
By the previous paragraph, the $\om$-orientations of~$X$ and $\wt{V}$ 
also induce intersection orientations on all~$\wt{V}_{\io}^{(k)}$.
By definition, the intersection and $\om$-orientations of~$\wt{V}_{\io}^{(1)}$
are the same.
The next statement follows readily from Definition~\ref{NCD_dfn} and
the proof of Lemma~\ref{NCD_lmm}.

\begin{prp}\label{NCD_prp}
Suppose $(X,\om)$ is a symplectic manifold, $V\!\subset\!X$ is an NC divisor, and
\hbox{$\io\!:\wt{V}\!\lra\!X$} is its normalization as in Lemma~\ref{NCD_lmm}.
Then  $V$ is an NC symplectic divisor in~$(X,\om)$
if and only if  $\io_k^*\om$ 
is a symplectic form on $\wt{V}_{\io}^{(k)}$ for all $k\!\in\!\Z^+$
and the intersection and $\om$-orientations of~$\wt{V}_{\io}^{(k)}$ are the same.
\end{prp}

\noindent
Suppose $\io\!:\wt{V}\lra\!X$ is a transverse immersion and $k\!\in\!\Z^{\ge0}$.
For $k'\!\in\!\Z^{\ge0}$ with $k'\!\le\!k$, define
\BE{kk'bundles}
\pi_{k;k'}\!:\cN_{k;k'}\io=\!\bigoplus_{i\in [k]-[k']}\!\!\!\!\!\cN\wt\io_{k;k-1}^{(i)}
\lra \wt{V}_{\io}^{(k)}  \quad\hbox{and}\quad
\pi_{k;k'}^c\!:\cN_{k;k'}^c\io=\!\bigoplus_{i\in [k']}\!\cN\wt\io_{k;k-1}^{(i)}
\lra \wt{V}_{\io}^{(k)}\,.\EE
By the commutativity of the first diagram in~\eref{Vkdiag_e}, 
the homomorphisms~$\nd\tilde\io_{k-1;k'}$ and~$\nd\io_{k-1}$ induce homomorphisms
$$\cN_{k;k'}\io\lra \cN\wt\io_{k;k'} \qquad\hbox{and}\qquad 
\cN_{k;k'}^c\io\lra\wt\io_{k;k'}^{\,*}\cN\io_{k'}.$$
By the Inverse Function Theorem, the last two homomorphisms are isomorphisms;
they correspond to the last isomorphism in~\eref{cNorient_e2}
and the first identification in~\eref{cNtot_e}
if $\wt{V}$ is the disjoint union of submanifolds $V_i\!\subset\!X$. 
For each $\si\!\in\!\bS_k$, the isomorphisms~\eref{Djsi_e}  induce an isomorphism
\BE{cNioksplit_e}
D\si=(D_i\si)_{i\in [k]}\!: 
\cN\io_k\approx\cN_{k;0}\io\equiv\bigoplus_{i\in[k]}\!\cN\wt\io_{k;k-1}^{(i)}
\lra \bigoplus_{i\in[k]}\!\cN\wt\io_{k;k-1}^{(\si(i))}\equiv \cN_{k;0}\io\approx \cN\io_k\EE
lifting the action of $\si$ on~$\wt{V}_{\io}^{(k)}$;
the last isomorphism permutes the components of the direct sum.
In particular, the subbundles $\cN_{k;k'}\io$ and $\cN_{k;k'}^c\io$ of $\cN_{k;0}\io$
are  invariant under the action of the subgroup  
\hbox{$\bS_{k'}\!\times\!\bS_{[k]-[k']}$} of~$\bS_k$,
but not under the action of the full group~$\bS_k$.

\begin{dfn}\label{NCsmreg_dfn}
A \sf{regularization} for an immersion $\io\!:\wt{V}\!\lra\!X$ is a smooth map 
\hbox{$\Psi\!:\cN'\!\lra\!X$} from a neighborhood of~$\wt{V}$ in~$\cN\io$ 
such that for every $\wt{v}\!\in\!\wt{V}$,
there exist a neighborhood $U_{\wt{v}}$ of $\wt{v}$ in~$\wt{V}$ so that 
the restriction of $\Psi$ to $\cN'|_{U_{\wt{v}}}$ is a diffeomorphism onto its~image,
$\Psi(\wt{v})\!=\!\io(\wt{v})$, and the homomorphism
$$ \cN\io|_{\wt{v}}=T_{\wt{v}}^{\ver}\cN\io \lhra T_{\wt{v}}\cN\io
\stackrel{\nd_{\wt{v}}\Psi}{\lra} 
T_{\wt{v}}X\lra \frac{T_{\wt{v}}X}{\Im(\nd_{\wt{v}}\io)}\equiv\cN\io|_{\wt{v}}$$
is the identity.
\end{dfn}

\begin{dfn}\label{NCTransCollReg_dfn}
A \sf{system of regularizations for} a transverse immersion $\io\!:\wt{V}\!\lra\!X$
is a tuple $(\Psi_k)_{k\in\Z^{\ge0}}$, where each $\Psi_k$ 
is a regularization for the immersion~$\io_k$, such~that 
\begin{alignat}{2}
\label{NCPsikk_e}\Psi_k\big(\cN_{k;k'}\io\!\cap\!\Dom(\Psi_k)\big)
&=V_{\io}^{(k')}\!\cap\!\Im(\Psi_k)
&\quad&\forall~k\!\in\!\Z^{\ge0},~k'\!\in\![k],\\
\label{NCPsikk_e2}
\Psi_k&=\Psi_k\!\circ\!D\si\big|_{\Dom(\Psi_k)} 
&\quad&\forall~k\!\in\!\Z^{\ge0},\,\si\!\in\!\bS_k.
\end{alignat}
\end{dfn}

\vspace{.1in}

\noindent
The stratification condition~\eref{NCPsikk_e} replaces~\eref{Psikk_e} 
and implies~that there exists a smooth~map
\BE{Psikkprdfn_e}\begin{split}
&\Psi_{k;k'}\!: \cN_{k;k'}'\io\!\equiv\!\cN_{k;k'}\io\!\cap\!\Dom(\Psi_k)\lra\wt{V}_{\io}^{(k')}
\qquad\hbox{s.t.}\\ 
&\quad\Psi_{k;k'}\big|_{\wt{V}_{\io}^{(k)}}=\wt\io_{k;k'}\,,\quad
\Psi_k\big|_{\cN_{k;k'}'\io}=\io_{k'}\!\circ\!\Psi_{k;k'}\,;
\end{split}\EE
see Proposition~1.35 and Theorem~1.32 in~\cite{Warner}.
Similarly to~\eref{wtPsiIIdfn_e}, $\Psi_{k;k'}$ lifts to a (fiberwise) vector bundle isomorphism
\BE{fDPsikk_e}\fD\Psi_{k;k'}\!:  \pi_{k;k'}^*\cN_{k;k'}^c\io|_{\cN_{k;k'}'\io}
\lra\cN\io_{k'}\big|_{\Im(\Psi_{k;k'})}.\EE
This bundle isomorphism preserves the second splittings in~\eref{kk'bundles} and
is $\bS_{k'}$-equivariant and $\bS_{[k]-[k']}$-invariant. 
The condition~\eref{NCoverlap_e} below replaces~\eref{overlap_e} in the present setting.

\begin{dfn}\label{NCTransCollregul_dfn}
A \sf{refined regularization} for a transverse immersion $\io\!:\wt{V}\!\lra\!X$
is a system  $(\Psi_k)_{k\in\Z^{\ge0}}$ of regularizations for~$\io$ 
such~that 
\BE{NCoverlap_e}\fD\Psi_{k;k'}\big(\Dom(\Psi_k)\big)
=\Dom(\Psi_{k'})\big|_{\Im(\Psi_{k;k'})}, \quad
\Psi_k=\Psi_{k'}\circ\fD\Psi_{k;k'}|_{\Dom(\Psi_k)}\EE
whenever $0\!\le\!k'\!\le\!k$ and $k\!\in\!\Z^{\ge0}$.
\end{dfn}

\noindent
Suppose $(X,\om)$ is a symplectic manifold and
$\io\!:\wt{V}\!\lra\!X$ is an immersion so that $\io^*\om$
is a symplectic form on~$V$.
The normal bundle 
$$\cN\io\equiv \frac{\io^*TX}{\Im(\nd\io)}\approx \big(\Im(\nd\io)\big)^{\om}
\equiv \big\{w\!\in\!T_{\io(\wt{v})}X\!:\,\wt{v}\!\in\!\wt{V},\,
\om\big(w,\nd_x\io(w')\big)\!=\!0~
\forall\,w'\!\in\!T_{\wt{v}}V\big\}$$
of~$\io$ then inherits a fiberwise symplectic form~$\om|_{\cN\io}$ from~$\om$.
We denote the restriction of~$\om|_{\cN\io}$ to a subbundle $L\!\subset\!\cN\io$
by~$\om|_L$.

\begin{dfn}\label{NCsympreg1_dfn}
Suppose $(X,\om)$ is a symplectic manifold,
$\io\!:\wt{V}\!\lra\!X$ is an immersion so that $\io^*\om$ is a symplectic form on~$V$,
and
$$\cN\io=\bigoplus_{i\in I}L_i$$
is a fixed splitting into oriented rank~2 subbundles.
If $\om|_{L_i}$ is nondegenerate for every $i\!\in\!I$, then
an \sf{$\om$-regularization for~$\io$} is a tuple $((\rho_i,\na^{(i)})_{i\in I},\Psi)$, 
where $(\rho_i,\na^{(i)})$ is an $\om|_{L_i}$-compatible Hermitian structure on~$L_i$
for each $i\!\in\!I$ and $\Psi$ is a regularization for~$\io$, such that 
$$\Psi^*\om=\big(\io^*\om\big)_{(\rho_i,\na^{(i)})_{i\in I}}\big|_{\Dom(\Psi)}.$$
\end{dfn}

\begin{dfn}\label{NCSCDregul_dfn}
Suppose $(X,\om)$ is a symplectic manifold and
$\io\!:\wt{V}\!\lra\!X$ is a transverse immersion of codimension~2
so that $\io_k^*\om$ is a symplectic form on $\wt{V}_{\io}^{(k)}$ for each $k\!\in\!\Z^+$.
A \sf{refined $\om$-regularization for~$\io$} is a tuple 
\BE{NCSCDregul_e} \fR\equiv\big(\cR_k\big)_{k\in\Z^{\ge0}}\equiv
\big((\rho_{k;i},\na^{(k;i)})_{i\in[k]},\Psi_k\big)_{k\in\Z^{\ge0}}\EE
such that $(\Psi_k)_{k\in\Z^{\ge0}}$ is a refined regularization for~$\io$,
$\cR_k$ is an $\om$-regularization for~$\io_k$
with respect to the splitting~\eref{cNioksplit_e} for every $k\!\in\!\Z^{\ge0}$, 
\BE{NCSCDregul_e2}\big(\rho_{k;i},\na^{(k;i)}\big) = 
\big\{D_i{\si}\big\}^{\!*}\big(\rho_{k;\si(i)},\na^{(k;\si(i))}\big)
\quad\forall\,k\!\in\!\Z^{\ge0},\,\si\!\in\!\bS_k,\,i\!\in\![k],\EE
and the induced vector bundle isomorphisms~\eref{fDPsikk_e}
are product Hermitian isomorphisms for all $k'\!\le\!k$. 
\end{dfn}

\noindent
For a smooth family $(\om_t)_{t\in B}$ of symplectic forms on~$X$ satisfying
the condition of Definition~\ref{NCSCDregul_dfn},
the notion of refined $\om$-regularization naturally extends to a notion of 
\sf{$(\om_t)_{t\in B}$-family of refined regularizations for~$\io$}.
If $\io\!:\wt{V}\!\lra\!X$ corresponds to an NC divisor~$V$ in~$X$ as in Lemma~\ref{NCD_lmm}, 
then a refined $\om$-regularization for~$\io$ in the sense of Definition~\ref{NCSCDregul_dfn}
determines an $\om$-regularization for~$V$ in~$X$ in the sense of 
Definition~\ref{NCDregul_dfn}\ref{NCDregul_it1}.
Conversely, the local regularizations constituting an $\om$-regularization for~$V$ in~$X$
can be patched together to form a refined $\om$-regularization for~$\io$.
These relations also apply to families of regularizations.

\subsection{Resolutions of NC divisors}
\label{EquivalentDfn_subs}

\noindent
In this section, we show that the normalization $\io\!:\!\wt{V}\!\lra\!X$ 
of an NC divisor~$V$ in~$X$ provided by Lemma~\ref{NCD_lmm} is unique and 
has a unique extension to a sequence of closed immersions compatible 
with group actions in the sense of Definition~\ref{NCRes_dfn} below. 
This structure can be used to simplify the notation in some applications; 
see Example~\ref{Succ_eg}.

\begin{dfn}\label{NCRes_dfn}
Let $X$ be a smooth manifold and $V\!\subset\!X$.
A~\sf{resolution of~$V$} is a sequence of closed immersions
\BE{absNC_e}
\ldots \lra X^{(k)}\xlra{f_{k;k-1}}X^{(k-1)}
\xlra{f_{k-1;k-2}}\cdots \xlra{f_{2;1}}  X^{(1)}
\xlra{f_{1;0}} X^{(0)}\!\equiv\!X\EE 
of a fixed codimension $\fc\!\in\!\Z^+$ such~that $f_{1;0}(X^{(1)})\!=\!V$ and
\begin{enumerate}[label=(R\arabic*),leftmargin=*]

\item for every $k\!\in\!\Z^{\ge0}$, 
$X^{(k)}$ is a manifold with a free $\bS_k$-action,

\item\label{Requiv_it} for all $0\!\le\!k'\!\le\!k$, the map 
\BE{absNC_e2} 
f_{k;k'}\!\equiv\!f_{k'+1;k'}\!\circ\!\ldots\!\circ\!f_{k;k-1}\!: X^{(k)}\!\lra\!X^{(k')}\EE
is $\bS_{k'}$-equivariant, and
 
\item\label{inverse_it} for all $0\!\le\!k'\!\le\!k$ and 
$x\!\in\!X^{(k)}\!-\!f_{k+1;k}(X^{(k+1)})$,
$f_{k;k'}^{-1}(f_{k;k'}(x))=\bS_{[k]-[k']}\!\cdot\!x$. 

\end{enumerate}
\end{dfn}

\vspace{.15in}

\noindent
Let $k\!\in\!\Z^{\ge0}$ and $x\!\in\!X^{(k)}$.
Since any resolution~\eref{absNC_e} terminates, the number
$$k(x)\equiv\max\big\{k'\!\in\!\Z^{\ge0}\!:k'\!\ge\!k,\,|f_{k';k}^{-1}(x)|\!\neq\!\eset\big\}
\in\Z^{\ge0}$$
is well-defined.
By \ref{Requiv_it} and~\ref{inverse_it} in Definition~\ref{NCRes_dfn} with $k$ replaced by~$k(x)$, 
$$f_{k;k'}(g\!\cdot\!x)=f_{k;k'}\big(f_{k(x);k}(g\!\cdot\!\wt{x})\big)
\equiv f_{k(x);k'}(g\!\cdot\!\wt{x})=f_{k(x);k'}(\wt{x})=f_{k;k'}(x)$$
for all $g\!\in\!\bS_{[k]-[k']}$ and $\wt{x}\!\in\!f_{k(x);k}^{-1}(x)$.
Thus, the immersion~\eref{absNC_e2} is $\bS_{[k]-[k']}$-invariant.
Since the complement of $f_{k+1;k}(X^{(k+1)})$ in~$X^{(k)}$ is dense
this conclusion also follows directly from~\ref{inverse_it} in Definition~\ref{NCRes_dfn}
and the continuity of $f_{k;k'}$ and of the $\bS_k$-action on~$X^{(k)}$.\\

\noindent
Suppose  $0\!\le\!k'\!\le\!k$ and $x\!\in\!X^{(k)}$. 
If $g,g'\!\in\!\bS_k$, then
\BE{NCRes_e5} f_{k;k'}(g\!\cdot\!x)=f_{k;k'}(g'\!\cdot\!x)
\quad\Llra\quad g'g^{-1}\in\bS_{[k]-[k']}\,.\EE
The {\it if} implication is immediate from the conclusion of the previous paragraph.
The {\it only if} implication with $(k,x)$ replaced by $(k(x),\wt{x})$ as in
the previous paragraph follows from~\ref{inverse_it} in Definition~\ref{NCRes_dfn}
with~$k$ replaced by~$k(x)$.
Along with~\ref{Requiv_it}, this establishes the {\it only if} in~\eref{NCRes_e5}.\\

\noindent
If in addition $k'\!\in\!\Z^+$ and $y\!\in\!X^{(k)}$, then 
\BE{NCRes_e5b}f_{k;k'}(g\!\cdot\!x)=f_{k;k'}(g\!\cdot\!y)
~~\forall\,g\!\in\!\bS_k  \quad\Lra\quad x=y.\EE
This is immediate if $k'\!=\!k$ or if $x$ or $y$ is not in the image $f_{k+1;k}$.
Suppose
$$k'<k, \quad k(x)\le k(y), \quad \wt{x}\in f_{k(x);k}^{-1}(x), 
\quad\hbox{and}\quad \wt{y}\in f_{k(x);k}^{-1}(y).$$ 
The $\bS_k$-equivariance of~$f_{k(x);k}$ and the assumption in~\eref{NCRes_e5b}
then imply~that 
\BE{NCRes_e5c}f_{k(x);k'}(g\!\cdot\!\wt{x})=f_{k(x);k'}(g\!\cdot\!\wt{y})\EE
for all $g\!\in\!\bS_k\!\subset\!\bS_{k(x)}$.
By the $\bS_{[k(x)]-[k']}$-invariance of~$f_{k(x);k'}$, 
\eref{NCRes_e5c} also holds for all $g\!\in\!\bS_{[k(x)]-[k']}$.
Since $\bS_k$ and $\bS_{[k(x)]-[k']}$ generate~$\bS_{k(x)}$,
it follows that~\eref{NCRes_e5c} holds for all $g\!\in\!\bS_{k(x)}$.
Thus, $\wt{x}\!=\!\wt{y}$ and the conclusion in~\eref{NCRes_e5b} still holds.

\begin{eg}\label{NCRes_eg}
Let $\io\!:\wt{V}\!\lra\!X$ be a closed transverse immersion of codimension $\fc\!\in\!\Z^+$.
With the notation as in~\eref{whiokk_e}, the sequence 
\BE{absNCCan_e}
\cdots\lra \wt{V}_{\io}^{(k)}\xlra{\wt\io_{k;k-1}} \wt{V}^{(k-1)}_{\io}\xlra{\wt\io_{k-1;k-2}}
\ldots \xlra{\wt\io_{2;1}} \wt{V}^{(1)}_{\io}
\xlra{\wt\io_{1;0}} \wt{V}^{(0)}_{\io}\!=\!X \EE 
is a resolution of $V\!=\!\io(\wt{V})$.
By Proposition~\ref{2TermToSeq_prp} below, every resolution of~$V$ is canonically isomorphic
to~\eref{absNCCan_e}. 
\end{eg}

\begin{prp}\label{2TermToSeq_prp}
Suppose $X$ is a smooth manifold, $\io\!:\wt{V}\!\lra\!X$ is a closed transverse immersion
of codimension $\fc\!\in\!\Z^+$,
\eref{absNCCan_e} is the associated canonical resolution of $V\!=\!\io(\wt{V})$,
and \eref{absNC_e} is any resolution of~$V$.
Then there exist unique smooth maps~$h_k$
so that the  diagram
\BE{absNC_e10}\begin{split}
\xymatrix{\ldots\ar[r]& X^{(k)}\ar[d]^{h_k}\ar[rr]^{f_{k;k-1}}&&
X^{(k-1)}\ar[d]^{h_{k-1}}\ar[rr]^{f_{k-1;k-2}}&& \cdots \ar[r]^{f_{2;1}}& X^{(1)} 
\ar[d]^{h_1}\ar[r]^>>>>>{f_{1;0}}& X^{(0)}\!=\!X\ar[d]^{\id}\\
\ldots\ar[r] &\wt{V}_{\io}^{(k)}\ar[rr]^{\wt\io_{k;k-1}}&&
\wt{V}_{\io}^{(k-1)} \ar[rr]^{\wt\io_{k-1;k-2}}&& \cdots\ar[r]^{\wt\io_{2;1}} & 
\wt{V}_{\io}^{(1)} \ar[r]^>>>>>{\wt\io_{1;0}}& \wt{V}_{\io}^{(0)}\!=\!X} 
\end{split}\EE
commutes. These maps are $\bS_k$-equivariant diffeomorphisms.
\end{prp}

\begin{proof}
The immersions~$f_{1;0}$ and~$\wt\io_{1;0}$ are of the same codimension~$\fc$
by \cite[Exercise~6]{Warner}.
Thus, the dimensions of the manifolds~$X^{(k)}$ and~$\wt{V}_{\io}^{(k)}$
are the same for each $k\!\in\!\Z^+$.
Since the restriction
\BE{xx1xll_e2a}\wt\io_{k;k-1}\!:
\wt{V}_{\io}^{(k)}\!-\!\wt\io_{k+1;k}(\wt{V}_{\io}^{(k+1)})\lra
\wt\io_{k;k-1}\big(\wt{V}_{\io}^{(k)}\!-\!\wt\io_{k+1;k}\wt{V}_{\io}^{(k+1)})\big)
\subset\wt{V}_{\io}^{(k-1)}\EE
is injective and the dimension of~$\wt{V}_{\io}^{(k+1)}$
is strictly less than the dimension of~$X^{(k)}$,
there can be at most one map~$h_k$ 
so that the left-most square in~\eref{absNC_e10} commutes
with a specific choice of~$h_{k-1}$.
Below we construct~$h_1$ based on general geometric considerations and then
define each $h_k$ with $k\!\ge\!2$ by an explicit formula.\\

\noindent
Since the restriction~\eref{xx1xll_e2a} is injective for $k\!=\!1$, 
there is a unique~map 
\BE{xx1xll_e2b}h_1\!: X^{(1)}\!-\!f_{1;0}^{-1}\big(\wt\io_{2;0}(\wt{V}_{\io}^{(2)})\big)
\lra \wt{V}_{\io}^{(1)}\!-\!\wt\io_{2;1}(\wt{V}_{\io}^{(2)})
\subset\wt{V}_{\io}^{(1)}\EE
so that $f_{1;0}\!=\!\wt\io_{1;0}\!\circ\!h_1$ on the domain of~$h_1$.
Since the restriction~\eref{xx1xll_e2a} is an embedding, 
the map~\eref{xx1xll_e2b} is smooth by \cite[Theorem~1.38]{Warner}.
Since $f_{1;0}$ is an immersion, 
$$f_{1;0}^{-1}\big(\wt\io_{2;0}(\wt{V}_{\io}^{(2)})\big)\subset X^{(1)}$$ 
is the image of 
a smooth map from a manifold of a smaller dimension than~$X^{(1)}$.
Suppose \hbox{$x\!\in\!f_{1;0}^{-1}(\wt\io_{2;0}(\wt{V}_{\io}^{(2)}))$}.
The restriction of~$f_{1;0}$ to a sufficiently small neighborhood~$U_x$ of~$x$ is an embedding
and its image agrees with the image of an embedding of a neighborhood
an element of  $\wt\io_{1;0}^{-1}(f_{1;0}(x))$ in~$\wt{V}_{\io}^{(1)}$.
By \cite[Theorem~1.38]{Warner}, $h_1$ extends smoothly over~$U_x$ as 
a map to $\wt{V}_{\io}^{(1)}$.
We thus obtain a smooth map~$h_1$ so that the right-most square in~\eref{absNC_e10} commutes.
Since~$f_{1;0}$ is a closed immersion which is injective on $X^{(1)}\!-\!f_{2;1}(X^{(2)})$, 
so is~$h_1$.
Since the image of~$f_{1;0}$ contains~$V$, the image of~$h_1$ contains 
$\wt{V}_{\io}^{(1)}\!-\!\wt\io_{2;1}(\wt{V}_{\io}^{(2)})$;
thus, $h_1$ is surjective.
We conclude that~$h_1$ is a diffeomorphism.\\

\noindent
From now on, we identify $(X^{(1)},f_{1;0})$ with 
$(\wt{V}_{\io}^{(1)},\wt\io_{1;0})\!\approx\!(\wt{V},\io)$ via~$h_1$.
For each $k\!\in\!\Z^+$, we then~have
$$f_{k;1}\!:X^{(k)}\lra X^{(1)}\!=\!\wt{V}\,.$$
We denote~by
$$\pi_{k;k-1}\!:X\!\times\!\wt{V}^k\lra X\!\times\!\wt{V}^{k-1}$$
the projection to the first $k$ components; it restricts to the map~\eref{whiokk_e}.
For $i,j\!\in\![k]$, let $\tau_{i,j}\!\in\!\bS_k$ be the transposition interchaning~$i$ and~$j$.\\

\noindent 
For $k\!\in\!\Z^{\ge0}$, define
\BE{xx1xll_e2}
h_k\!: X^{(k)}\lra X\!\times\!\wt{V}^k, \quad 
h_k(x)=\big(f_{k;0}(x),f_{k;1}(\tau_{1,1}(x)),\ldots,
f_{k;1}(\tau_{1,k}(x))\big).\EE
In particular, $h_k$ is an immersion, $h_0\!=\!\id_X$, $h_1\!=\!\id_{\wt{V}}$, 
and the diagram
\BE{absNC_e10a}\begin{split}
\xymatrix{\ldots\ar[r]& X^{(k)}\ar[d]^{h_k}\ar[rr]^{f_{k;k-1}}&&
X^{(k-1)}\ar[d]^{h_{k-1}}\ar[rr]^{f_{k-1;k-2}}&& \cdots \ar[r]^{f_{2;1}}& X^{(1)} 
\ar[d]^{h_1}\ar[r]^{f_{1;0}}& X^{(0)}\ar[d]^{\id}\\
\ldots\ar[r] &X\!\times\!\wt{V}^k\ar[rr]^{\pi_{k;k-1}}&&
X\!\times\!\wt{V}^{k-1} \ar[rr]^{\pi_{k-1;k-2}}&& \cdots\ar[r]^{\pi_{2;1}} & 
X\!\times\!\wt{V} \ar[r]^{\pi_{1;0}}& X^{(0)}}
\end{split}\EE
commutes by $\bS_{k-1}$-equivariance of $f_{k;k-1}$.
Since 
\BE{absNC_e10c} \tau_{1,j}\tau_{1,i}=\tau_{i,j}\tau_{1,j} \qquad
\forall~i,j\!\in\![k]\!-\![1],\,i\!\neq\!j,\EE
the $\bS_{[k]-[1]}$-invariance of~$f_{k;1}$ implies~that 
$$h_k\big(\tau_{1,i}(x)\big)= \tau_{1,i}\!\cdot\!h_k(x)
\qquad\forall\,x\!\in\!X^{(k)},\,i\!\in\![k],\,k\!\in\!\Z^+\,.$$
Thus, $h_k$ is $\bS_k$-equivariant.
By~\eref{NCRes_e5b} with $k'\!=\!1$, $h_k$ is injective for all $k\!\ge\!1$.
Since $f_{k;0}\!\equiv\!\io\!\circ\!f_{k;1}$ is $\bS_k$-invariant, 
$$\io\big(f_{k;1}(\tau_{1,i}(x))\big)=f_{k;0}(x)
\qquad\forall\,x\!\in\!X^{(k)},\,i\!\in\![k],\,k\!\in\!\Z^+\,.$$
By~\eref{NCRes_e5}, $f_{k;1}(\tau_{1,i}(x))\!\neq\!f_{k;1}(\tau_{1,j}(x))$
for all $i\!\neq\!j$.
Thus, $h_k(x)\!\in\!\wt{V}^{(k)}_{\io}$.\\

\noindent
It remains to show that the maps $\Im\,h_k\!=\!\wt{V}_{\io}^{(k)}$.
Since each $h_k$ is an injective immersion, this would imply that each $h_k$ is 
a diffeomorphism.
This is the case for $k\!=\!0,1$.
Suppose $k\!\ge\!2$ and $h_{k'}$ is a diffeomorphism for all $k'\!<\!k$.
Since~$f_{k;k-1}$ is a closed immersion, so is~$h_k$.
By the $\bS_2$-invariance of~$f_{2;0}$ and $\bS_{[k]-[k-2]}$-invariance of~$f_{k;1}$ for $k\!\ge\!3$,
$$h_{k-1}\big(f_{k;k-1}(X^{(k)})\big)\subset
\big\{\wt{v}\!\in\!\wt{V}_{\io}^{(k-1)}\!:
\big|\wt\io_{k-1;k-2}^{-1}(\wt\io_{k-1;k-2}(\wt{v}))\big|\!>\!1\big\}
=\wt\io_{k;k-1}\big(\wt{V}_{\io}^{(k)}\big).$$
The opposite inclusion holds by~\ref{inverse_it} in Definition~\ref{NCRes_dfn}.
Since $\wt\io_{k;k-1}\!\circ\!h_k\!=\!h_{k-1}\!\circ\!f_{k;k-1}$ and~\eref{xx1xll_e2a} is injective, 
it follows that the image of~$h_k$ contains 
$\wt{V}_{\io}^{(k)}\!-\!\wt\io_{k+1;k}(\wt{V}_{\io}^{(k+1)})$.
Thus, $\Im\,h_k\!=\!\wt{V}_{\io}^{(k)}$.
We conclude that~$h_k$ is a diffeomorphism from~$X^{(k)}$ to~$\wt{V}_{\io}^{(k)}$.
\end{proof}

\begin{eg}\label{Succ_eg}
Let $\io\!:\wt{V}\!\lra\!X$ be as in Example~\ref{NCRes_eg} with $\fc\!=\!2$.
For every $m\!\in\!\Z^{\geq 0}$, the $\bS_m$-equivariant immersion 
\BE{ktok-1_e}
\jo\!\equiv\!\wt\io_{m+1;m}\!: 
\wt{W}\!\equiv\!\wt{V}^{(m+1)}_\io \lra Y\!\equiv\!\wt{V}^{(m)}_\io\EE
is the normalization of an NC divisor $W\!\equiv\!\jo(\wt{W})$ in~$Y$.
This divisor is preserved by the $\bS_m$-action on~$\wt{V}^{(m)}$. 
By restricting the $\bS_{m+k}$-actions on $\wt{V}^{(m+k)}_{\io}$
to actions of the subgroups $\bS_{[m+k]-[m]}\!\approx\!\bS_k$, 
we obtain a resolution
\BE{absNCCanShort_e}
\ldots \xlra{\wt\io_{m+k+1;m+k}} \wt{V}^{(m+k)}_{\io} \xlra{\wt\io_{m+k;m+k-1}}\ldots 
\lra \wt{V}^{(m+1)}_{\io}\xlra{\wt\io_{m+1;m}}\wt{V}^{(m)}_{\io}\!=\!Y\EE 
of $W$ as in Definition~\ref{NCRes_dfn}.
By Proposition~\ref{2TermToSeq_prp}, \eref{absNCCanShort_e} is canonically isomorphic 
to the resolution
$$\ldots\xlra{\wt\jo_{k+1;k}} \wt{W}^{(k)}_{\jo}\xlra{\wt\jo_{k;k-1}}\ldots 
\lra \wt{W}^{(1)}_{\jo}\!\approx\!\wt{W} \xlra{\jo}Y,$$
with  $\wt{W}^{(k)}_{\jo}$ and $\wt\jo_{k;k-1}$ defined as in~\eref{tViotak_e} 
and~\eref{whiokk_e}, respectively.
An explicit identification is provided by the diffeomorphisms
$$\phi_k\!:\wt{W}^{(k)}_{\jo}\lra \wt{V}^{(m+k)}_{\io}, \quad
\phi_k\big(y,\wt{w}_1,\ldots,\wt{w}_k\big)=
\big(y,\wt\io_{m+1}^{(m+1)}(\wt{w}_1),\ldots,\wt\io_{m+1}^{(m+1)}(\wt{w}_k)\!\big)
\in \big(X\!\times\!\wt{V}^m\big)\!\times\!\wt{V}^k.$$
From such an identification, we obtain identifications 
\BE{cNapprox_e}
\cN\wt\jo_{k;k-1}^{(i)} \approx  
\cN\wt\io_{m+k;m+k-1}^{(m+i)}\qquad \forall~k\!\in\!\Z^+,~i\!\in\![k],\EE
of the normal bundles to the immersions defined by~\eref{iokcjdfn_e}.
Therefore, a refined $\om$-regularization~$\fR$ for the immersion~$\io$ 
as in Definition~\ref{NCSCDregul_dfn} induces an $\io_m^*\om$-regularization
$$\fR^{(m)}\approx\big((\rho_{m+k;m+i},\na^{(m+k;m+i)})_{i\in[k]},\Psi_{m+k;m}\big)_{k\in\Z^{\ge0}}$$
for the immersion~$\jo$, with $\Psi_{m+k;m}$ as in~\eref{Psikkprdfn_e}.
We use the above identifications in~\cite{SympNCStructures} to construct canonical isomorphisms
$$\io_m^*\cO_{X}(V) \approx\cO_{Y}(W)\otimes\!\bigotimes_{i\in[m]}\!\!\cN\wt\io_{m;m-1}^{(i)},
\quad
\io_m^*TX(-\log V) \approx TY(-\log W)\!\oplus\! \big(Y\!\times\!\C^m\!\big)$$
of vector bundles over~$Y$.
\end{eg}

\section{Normal crossings symplectic varieties}
\label{NCC_sec}

\noindent
NC~varieties are spaces that are locally varieties associated to SC~configurations.
They can also be described as global quotients along images of immersions with compatible
involutions on their domains.
These two perspectives are presented in Sections~\ref{NCCloc_subs} and~\ref{NCCgl_subs},
respectively, and are shown to be equivalent in Section~\ref{NCCcomp_subs}.
An alternative global characterization of regularizations for NC varieties is presented 
in Section~\ref{NCgl2_subs}.
The (virtual) existence theorem for regularizations for NC symplectic varieties,
Theorem~\ref{NCC_thm}, is justified in Section~\ref{NCCpf_sec}.
Section~\ref{eg_subs} presents some examples of NC divisors and varieties.

\subsection{Local perspective}
\label{NCCloc_subs}

\noindent
Suppose
\BE{XXprdfn_e}\X\equiv\big\{X_I\big\}_{I\in\cP^*(N)} \qquad\hbox{and}\qquad 
\X'\equiv\big\{X_I'\big\}_{I\in\cP^*(N')}\EE
are transverse configurations with associated spaces~$X_{\eset}$ and~$X_{\eset}'$
as in~\eref{Xesetdfn_e}.
A homeomorphism
\BE{esetvphdfn_e}\vph\!:X_{\eset}\lra X_{\eset}'\EE
is a \sf{diffeomorphism} if there exists 
a~map \hbox{$h\!:[N]\!\lra\![N']$} such that the~restriction
\BE{vphhdfn_e}\vph\!:X_i\lra X_{h(i)}'\EE
is a diffeomorphism between manifolds for every $i\!\in\![N]$.
This implies that $X_j'\!=\!\eset$ unless $j\!=\!h(i)$ for some $i\!\in\![N]$.
If $\vph$ is a diffeomorphism and $(\om_j')_{j\in[N']}$ is a symplectic structure on~$\X'$,
then 
$$\big(\om_i\big)_{i\in[N]}\equiv\vph^*\big((\om_j')_{j\in[N']}\big)
\equiv  \big(\vph|_{X_i}^*\om_{h(i)}'\big)_{i\in[N]}$$
is a symplectic structure on~$\X$.

\begin{dfn}\label{ImmTransConf_dfn0}
Let $X$ be a topological space. 
\begin{enumerate}[label=(\arabic*),leftmargin=*]

\item\label{NCchart_it} An \sf{NC chart for~$X$ around} a point~$x\!\in\!X$ is a tuple 
$(U,\X,\vph\!:U\!\lra\!X_{\eset})$,
where $U$ is an open neighborhood of $x$ in~$X$, 
$\X$ is a finite transverse configuration with associated quotient space~$X_{\eset}$,
and $\vph$ is a homeomorphism.

\item\label{NCatlas_it} 
An \sf{NC atlas} for $X$ is a maximal collection $(U_y,\X_y,\vph_y)_{y\in\cA}$
of charts for~$X$ such that for all $y,y'\!\in\!\cA'$ 
and $x\!\in\!U_y\!\cap\!U_{y'}$ there exists a neighborhood~$U_{yy';x}$ 
of~$x$ in~$U_y\!\cap\!U_{y'}$ so that the overlap~map
\BE{ImmTransConf_e2}
\vph_{yy';x}\!\equiv\!\vph_y\!\circ\!\vph_{y'}^{-1}\!: 
\vph_{y'}\big(U_{yy';x}\big)\lra \vph_y\big(U_{yy';x}\big)\EE
is a diffeomorphism.
\end{enumerate}
\end{dfn}

\begin{dfn}\label{ImmTransConf_dfn}
An~\sf{NC variety} is a (Hausdorff) paracompact second-countable topological
space~$X$ with an NC atlas $(U_y,\X_y,\vph_y)_{y\in\cA}$ with 
$\X_y\!=\!(X_{y;I})_{I\in\cP^*(N_y)}$ 
such that $X_{y;ij}$ is a closed submanifold of~$X_{y;i}$ of codimension~2
for all $i,j\!\in\![N_y]$ distinct.
\begin{enumerate}[label=(\arabic*),leftmargin=*]

\item\label{NCsympstr_it} A \sf{symplectic structure} on such an~NC variety 
is a tuple $(\om_{y;i})_{y\in\cA,i\in[N_y]}$,
where each $(\om_{y;i})_{i\in[N_y]}$ is a symplectic structure on~$\X_y$,
such~that 
$$\big(\om_{y';i}\big)_{i\in[N_{y'}]}\big|_{\vph_{y'}(U_{yy';x})}
= \vph_{yy';x}^{~*}\big((\om_{y;i})_{i\in[N_y]}\big)$$
for all $y,y'\!\in\!\cA$ and $x\!\in\!U_{yy';x}\!\subset\!U_y\!\cap\!U_{y'}$ 
as in Definition~\ref{ImmTransConf_dfn0}\ref{NCatlas_it}.

\item If $B$ is a manifold, possibly with boundary, 
a tuple $(\om_{t;y;i})_{t\in B,y\in\cA,i\in[N_y]}$
is a \sf{family of symplectic structures} on~$\X$ if 
$(\om_{t;y;i})_{y\in\cA,i\in[N_y]}$ is a symplectic structure on~$\X$ for each $t\!\in\!B$
and $(\om_{t;y;i})_{t\in B,i\in[N_y]}$ is a family of symplectic structures on~$\X_y$
for each~$y\!\in\!\cA$.
\end{enumerate}
\end{dfn}

\begin{dfn}\label{NCC_dfn}
An \sf{NC symplectic variety} is a pair consisting of an~NC variety~$X$
and a symplectic structure $(\om_{y;i})_{y\in\cA,i\in[N_y]}$
on~$X$ as in Definition~\ref{ImmTransConf_dfn}\ref{NCsympstr_it} such~that 
$(\om_{y;i})_{i\in[N_y]}\!\in\!\Symp^+(\X_y)$ for every $y\!\in\!\cA$.
\end{dfn}

\noindent
The symplectic variety~$X_{\eset}$ determined by an SC symplectic configuration~$\X$
as in Definition~\ref{SCC_dfn} is an NC symplectic variety in the sense of 
Definition~\ref{NCC_dfn}.
The corresponding NC~atlas is the maximal collection of NC~charts on~$X_{\eset}$
containing the chart $(X_{\eset},\X,\id_{X_{\eset}})$.\\

\noindent
The \sf{singular locus} of an~NC variety~$X$ as in Definition~\ref{ImmTransConf_dfn}
is the closed subspace
$$X_{\prt}\equiv \bigcup_{y\in\cA}\vph_y^{-1}\big(X_{y;\prt}\big)\subset X,$$
where $X_{y;\prt}\!\subset\!X_{y;\eset}$ is the subspace corresponding to~$\X_y$
as in~\eref{Xprtdfn_e}.
Let
$$\big(X_{\prt}\big)_{\prt}\equiv \bigcup_{y\in\cA}
\bigcup_{\begin{subarray}{c}I\subset[N_y]\\ |I|=3\end{subarray}}
\!\!\!\!\!\vph_y^{-1}\big(X_{y;I}\big)\subset X$$
denote the \sf{singular locus of~$X_{\prt}$}; it is also a closed subset.\\

\noindent
For $k\!\in\!\Z^{\ge0}$, a \sf{$k$-form on~$X$} is a tuple $\mu\!\equiv\!(\mu_{y;i})_{y\in\cA,i\in[N_y]}$,
where each $\mu_{y;i}$ is a $k$-form on~$X_{y;i}$ such~that 
\begin{alignat*}{2}
\mu_{y;i}\big|_{X_{y;ij}}&=\mu_{y;j}\big|_{X_{y;ij}}
&\qquad &\forall\,i,j\!\in\![N_y],\,y\!\in\!\cA, \\
\mu_{y';i}|_{\vph_{y'}(U_{yy';x})} &= \vph_{yy';x}^{~*}\,\mu_{y;h_{yy';x}(i)}
&\qquad &\forall\,i\!\in\![N_{y'}],\,y,y'\!\in\!\cA,\,x\!\in\!U_{yy';x}\!\subset\!U_y\!\cap\!U_{y'}\,,
\end{alignat*}
with $\vph_{yy';x}$ as in Definition~\ref{ImmTransConf_dfn0}\ref{NCatlas_it}
and with $h_{yy';x}$ being the associated map as in~\eref{vphhdfn_e}.
The~tuple
$$ \nd\mu\equiv\big(\nd\mu_{y;i}\big)_{y\in\cA,i\in[N_y]}$$
is then a $(k\!+\!1)$-form on~$X$.
Let
$$\supp(\mu)\equiv \ov{\bigcup_{y\in\cA}\bigcup_{i\in[N_y]} \!\!\!
\vph_y^{-1}\big(\supp(\mu_{y;i})\big)}\subset X$$
denote the \sf{support of~$\mu$}.
Denote by $\Symp^+(X)$ the space of all symplectic
structures~$\om$ on~$X$ so that 
$(X,\om)$ is an NC symplectic variety.\\

\noindent
Let $(\om_i)_{i\in[N]}$ and $(\om_i')_{i\in[N']}$ be symplectic structures on 
transverse configurations $\X$ and~$\X'$, respectively,
as in~\eref{XXprdfn_e}.
Suppose $\vph\!:U\!\lra\!U'$ is a symplectomorphism between open subsets 
of~$X_{\eset}$ and~$X_{\eset}'$ and
$h\!:[N]\!\lra\![N']$ is as in~\eref{vphhdfn_e} with the two sides replaced
by their intersections with~$U$ and~$U'$. 
We also denote by~$h$ the induced map from~$\cP(N)$ to~$\cP(N')$.
The diffeomorphism~$\vph$ determines isomorphisms
$$\nd\vph\!: \cN_{X_i}X_I\big|_{X_I\cap U}\lra \cN_{X_{i'}}X_{I'}\big|_{X_{I'}\cap U'} 
\qquad\forall\,i\!\in\!I\!\subset\![N],~i'\!\equiv\!h(i)\in I'\!\equiv\!h(I).$$
Suppose
$$\big(\cR_I\big)_{I\in\cP^*(N)} \equiv  
\big(\rho_{I;i},\na^{(I;i)},\Psi_{I;i}\big)_{i\in I\subset[N]}
\quad\hbox{and}\quad
\big(\cR_I'\big)_{I\in\cP^*(N')} \equiv  
\big(\rho_{I;i}',\na'^{(I;i)},\Psi_{I;i}'\big)_{i\in I\subset[N']}$$
are an $(\om_i)_{i\in[N]}$-regularization for~$\X$ and 
an $(\om_i')_{i\in[N']}$-regularization for~$\X'$, respectively.
We define 
$$\big(\cR_I\big)_{I\in\cP^*(N)}\cong_{\vph}\big(\cR_I'\big)_{I\in\cP^*(N')}$$
if 
\begin{gather*}
\big(\rho_{I;i},\na^{(I;i)}\big)\big|_{X_I\cap U}
=\big\{\nd\vph\big\}^*\big(\rho_{I';i'}',\na'^{(I';i')}\big),\\
\vph\!\circ\!\Psi_{I;i}
=\Psi_{I';i'}'\!\circ\!\nd\vph
\quad\hbox{on}\quad
\Dom(\Psi_{I;i})\cap \nd\vph^{-1}\big(\Dom(\Psi_{I';i'}')\big)
\end{gather*}
for all $i\!\in\!I\!\subset\![N]$, 
$i'\!\equiv\!h(i)\in I'\!\equiv\!h(I)$.

\begin{dfn}\label{NCCregul_dfn}
Let $X$ be an NC variety with an NC atlas $(U_y,\X_y,\vph_y)_{y\in\cA}$.
\begin{enumerate}[label=(\arabic*),leftmargin=*]

\item\label{NCCregul_it1} 
If $\om\!\equiv\!(\om_{y;i})_{y\in\cA,i\in[N_y]}$ is a symplectic structure on~$X$,
  an \sf{$\om$-regularization for $X$} is a~collection
$$(\fR_y)_{y\in\cA} \equiv (\cR_{y;I})_{y\in\cA,I\in\cP^*(N_y)} $$
such that $\fR_y$ is an $(\om_{y;i})_{i\in[N_y]}$-regularization for~$\X_y$
for each~$y\!\in\!\cA$ and   
\BE{NCCregulover_e}\fR_{y'}\big|_{\vph_{y'}(U_{yy';x})} \cong_{\vph_{yy';x}}
\fR_y\big|_{\vph_y(U_{yy';x})}\EE
for all $y,y'\!\in\!\cA$ and $x\!\in\!U_{yy';x}\!\subset\!U_y\!\cap\!U_{y'}$ 
as in Definition~\ref{ImmTransConf_dfn0}\ref{NCatlas_it}.

\item\label{NCCregul_it2} 
If $B$ is a manifold, possibly with boundary, 
and $(\om_{t;y;i})_{t\in B,y\in\cA,i\in[N_y]}$ is a family of symplectic structures on~$X$,
an \sf{$(\om_{t;y;i})_{t\in B,y\in\cA,i\in[N_y]}$-family of regularizations for~$X$} 
is a tuple $(\fR_{t;y})_{y\in\cA,t\in B}$ such that $(\fR_{t;y})_{y\in\cA}$
is an $(\om_{t;y;i})_{y\in\cA,i\in[N_y]}$-regularization for~$X$ for each~$t\!\in\!B$
and $(\fR_{t;y})_{t\in B}$ is an $(\om_{t;y;i})_{t\in B,i\in[N_y]}$-family of 
regularizations for~$\X_y$ for each~$y\!\in\!\cA$.
\end{enumerate}
\end{dfn}

\vspace{.1in}

\noindent
For  a family $(\om_t)_{t\in B}$ of symplectic structures on an NC variety~$X$,
we define two $(\om_t)_{t\in B}$-families of regularizations  
for~$X$ to be \sf{equivalent} if they agree on the level of germs 
as defined before \cite[Theorem~2.17]{SympDivConf}.

\begin{thm}\label{NCC_thm}
Let $X$ be an NC variety with an NC atlas $(U_y,\X_y,\vph_y)_{y\in\cA}$ and
$X^*\!\subset\!X$ be an open subset such that $\ov{X^*}\!\cap\!(X_{\prt})_{\prt}\!=\!\eset$.
Suppose
\begin{enumerate}[label=$\bullet$,leftmargin=*]
\item $B$, $N(\prt B)$, and $N'(\prt B)$ are as in Theorem~\ref{NCD_thm},

\item $(\om_t)_{t\in B}$ is a  family of symplectic structures 
in $\Symp^+(X)$,

\item $(\fR_{t;y})_{t\in N(\prt B),y\in\cA}$ is an $(\om_t)_{t\in N(\prt B)}$-family
of regularizations for~$X$.

\end{enumerate}
Then there exist a smooth family $(\mu_{t,\tau})_{t\in B,\tau\in\bI}$ of
1-forms on~$X$ such~that 
\BE{SCCom_e}\begin{aligned}
\mu_{t,0}=0, ~~ 
&\supp\big(\mu_{\cdot,\tau}\big)\subset 
\big(B\!-\!N'(\prt B)\big)\!\times\!(X\!-\!X^*),  &\qquad& \forall~t\!\in\!B,\,\tau\!\in\!\bI,\\
&\om_{t,\tau}\!\equiv\!\om_t\!+\!\nd\mu_{t,\tau}\in \Symp^+(X)
&\qquad& \forall~t\!\in\!B,\,\tau\!\in\!\bI,
\end{aligned}\EE
and an $(\om_{t,1})_{t\in B}$-family 
$(\wt\fR_{t;y})_{t\in B,y\in\cA}$ of regularizations  for~$X$ such~that
\BE{SCCom_e2}\big(\wt\fR_{t;y}\big)_{t\in N'(\prt B),y\in\cA} \cong 
\big(\fR_{t;y}\big)_{t\in N'(\prt B),y\in\cA}\,.\EE
\end{thm}

\vspace{.1in}

\noindent
This statement, which is the direct analogue of \cite[Theorem~2.17]{SympDivConf}
for arbitrary NC varieties, implies that the second projection in~\eref{AuxtoSympNC_e}
is a weak homotopy equivalence in this general setting as well.
The proof of \cite[Theorem~2.17]{SympDivConf} is mostly local in nature and readily extends
to Theorem~\ref{NCC_thm}; see Section~\ref{NCCpf_sec} for more details.

\subsection{Global perspective}
\label{NCCgl_subs}

\noindent
Suppose $\wt{X}$ and $\wt{V}$ are sets.
We call a pair
\BE{iopsidfn_e}\big(\io\!:\wt{V}\!\lra\!\wt{X},\psi\!:\wt{V}\!\lra\!\wt{V}\big)\EE
consisting of a map and an involution, 
i.e.~$\psi^2\!=\!\id_{\wt{V}}$, \sf{compatible}~if
\BE{iopsiprop_e1}\begin{split}
&\io\big(\psi(\wt{v})\big)\neq\io(\wt{v})~~\forall\,\wt{v}\!\in\! \wt{V},\quad
\wt{v},\wt{v}'\!\in\!\wt{V},\,\wt{v}\!\neq\!\wt{v}',\,
\io(\wt{v})=\io(\wt{v}')\Lra\io\big(\psi(\wt{v})\big)\neq\io\big(\psi(\wt{v}')\big),\\
&\qquad\psi\big(\io^{-1}\big(\{\wt{x}\}\!\cup\!\io(\psi(\io^{-1}(\wt{x})))\big)\big)
=\io^{-1}\big(\{\wt{x}\}\!\cup\!\io(\psi(\io^{-1}(\wt{x})))\big)
~~\forall\,\wt{x}\!\in\!\wt{X}.
\end{split}\EE
The first condition in~\eref{iopsiprop_e1} is equivalent to 
the image of the~map
\BE{wtXX2_e}\wt{V}\lra \wt{X}^2, \qquad \wt{v}\lra\big(\io(\wt{v}),\io(\psi(\wt{v}))\big),\EE
being disjoint from the diagonal $\De_{\wt{X}}\!\subset\!\wt{X}^2$.
The second condition in~\eref{iopsiprop_e1} is equivalent to the injectivity of this~map.
The three conditions are illustrated in Figure~\ref{wtVpsi_fig}.\\

\noindent
For a compatible pair $(\io,\psi)$ as in~\eref{iopsidfn_e}, define
\BE{iopsiprop_e2}
X_{\io,\psi}=\wt{X}\big/\!\!\sim, \quad
\wt{x}\sim\wt{x}' ~~~\hbox{if}~~~\wt{x}=\wt{x}'~~\hbox{or}~~ 
\psi\big(\io^{-1}(\wt{x})\big)\cap\io^{-1}(\wt{x}')\neq\eset.\EE
Since $\psi$ is an involution, the relation~$\sim$ above is symmetric.
In light of the first two conditions in~\eref{iopsiprop_e1},
the last condition in~\eref{iopsiprop_e1} is equivalent to the transitivity of this relation.
Let
$$q\!:\wt{X} \lra X_{\io,\psi}$$
be the quotient projection.
By~\eref{iopsiprop_e2}, the projection of~$\wt{X}_q^{(2)}$ to $\wt{X}^2$ is 
the image of~\eref{wtXX2_e}.
For every $k\!\in\!\Z^{\ge0}$, the~map
\BE{tiXprtX_e}
\phi_k\!:\wt{V}_{\io}^{(k)}\lra \wt{X}_q^{(k+1)}\!\subset\!X_{\io,\psi}\!\times\!\wt{X}^{k+1},
~~
\phi_k\big(\wt{x},\wt{v}_1,\ldots,\wt{v}_k\big)= 
\big(q(\wt{x}),\wt{x},\io(\psi(\wt{v}_1)\!),\ldots,\io(\psi(\wt{v}_k)\!)\!\big),\EE
is thus a bijection.
We note that the diagram
\BE{phikiio_e}\begin{split}
\xymatrix{\wt{V}_{\io}^{(k)}  \ar[rr]^{\phi_k}  \ar[rd]^{\io_k} \ar[dd]_{\io_{k;k-1}^{(i)}}
&& \wt{X}_q^{(k+1)} \ar[ld]|{q_{k+1;1}} \ar[dd]^{q_{k+1;k}^{(i+1)}}\\
& \wt{X} & \\ 
\wt{V}_{\io}^{(k-1)} \ar[rr]_{\phi_{k-1}} \ar[ru]|{\io_{k-1}} && 
\wt{X}_q^{(k)} \ar[ul]^{q_{k;1}\!\!}}
\end{split}\EE
commutes for all $i\!\in\![k]$ and $k\!\in\!\Z^+$.

\begin{lmm}\label{iopsi_lmm}
Suppose $(\io,\psi)$ is a compatible pair as in~\eref{iopsidfn_e}
so that $\io$ is a closed transverse immersion and $\psi$ is smooth.
\begin{enumerate}[label=(\arabic*),leftmargin=*]
 
\item For every $k\!\in\!\Z^{\ge0}$, the bijection $\phi_k$ is a homeomorphism.

\item\label{wtXqproj_it} The projection $\wt{X}_q^{(k+1)}\!\lra\!\wt{X}^{k+1}$ is an embedding
with respect to the smooth structure on $\wt{X}_q^{(k+1)}$ induced by~$\phi_k$.

\end{enumerate}
\end{lmm}

\begin{proof}
Since $\io$ is a closed transverse immersion and $\psi$ is smooth, the map
\BE{iopsi_e3}\wt{V}_{\io}^{(k)}\lra \wt{X}^{k+1}, \quad
\big(\wt{x},\wt{v}_1,\ldots,\wt{v}_k\big)\lra
\big(\wt{x},\io(\psi(\wt{v}_1)\!),\ldots,\io(\psi(\wt{v}_k)\!)\!\big),\EE
is a closed immersion.
Since this map is injective, it follows that it is a topological embedding.
Along with \cite[Theorem~1.32]{Warner}, this implies that~\eref{iopsi_e3} is a smooth embedding.
Since $q$ is continuous, the projection in~\ref{wtXqproj_it} is a topological embedding.
The last two conclusions imply both claims of the lemma.
\end{proof}

\noindent
By Lemma~\ref{iopsi_lmm}, the smooth structure on $\wt{X}_q^{(k+1)}$ induced by $\phi_k$ 
is $\bS_{k+1}$-invariant.
Furthermore, the~maps
\begin{alignat*}{2}
q_{k+1}^{(i)}\!:\wt{X}_q^{(k+1)}&\lra\wt{X}, &\quad 
q_{k+1}^{(i)}(x,\wt{x}_1,\ldots,\wt{x}_{k+1})&=\wt{x}_i,\\
q_{k+1;k}^{(i)}\!:\wt{X}_q^{(k+1)}&\lra\wt{X}_q^{(k)},&\quad
q_{k+1;k}^{(i)}\big(x,\wt{x}_1,\ldots,\wt{x}_{k+1}\big)&= 
\big(x,\wt{x}_1,\ldots,\wt{x}_{i-1},\wt{x}_{i+1},\ldots,\wt{x}_{k+1}\big),
\end{alignat*}
are immersions for all $k\!\in\!\Z^{\ge0}$ and $i\!\in\![k\!+\!1]$.\\

\begin{figure}
\begin{pspicture}(-4,-4)(11,4.5)
\psset{unit=.3cm}
\psline[linewidth=.08](-1,14)(-1,8)\pscircle*(-1,11){.2}\rput(-2.1,11.7){\sm{$\wt{v}_{12}$}}
\psline[linewidth=.08,linestyle=dashed](-3,9)(1,13)\rput(2.5,12.5){\sm{$\cN\io_{2;1}^{(1)}$}}
\psline[linewidth=.05]{<->}(3,11)(10,11)\rput(6.5,11.7){$\psi$}
\psline[linewidth=.08](14,14)(14,8)\pscircle*(14,11){.2}\rput(13,11.9){\sm{$\wt{v}_{21}$}}
\psline[linewidth=.08,linestyle=dashed](11,11)(17,11)\rput(17,11.8){\sm{$\cN\io_{2;1}^{(1)}$}}
\psline[linewidth=.08](27,9)(31,13)\pscircle*(29,11){.2}\rput(28.3,11.9){\sm{$\wt{v}_{31}$}}
\psline[linewidth=.08,linestyle=dashed](26,11)(32,11)\rput(33,11.8){\sm{$\cN\io_{2;1}^{(1)}$}}
\psline[linewidth=.09,linestyle=dashed](-1,6)(-1,0)\pscircle*(-1,3){.2}\rput(0.3,2.6){\sm{$\wt{v}_{13}$}}
\psline[linewidth=.08](-3,1)(1,5)\rput(-2.8,5){\sm{$\cN\io_{2;1}^{(2)}$}}
\psline[linewidth=.05]{<->}(3,5.3)(25,8.3)\rput(15,7.7){$\psi$}
\psline[linewidth=.05,linestyle=dashed]{<->}(0,.5)(13,.5)\rput(6.5,1.3){$D_2\psi_2$}
\psline[linewidth=.05]{->}(-1,-1.5)(-1,-5.5)\rput(-.3,-3.5){$\io$}
\psline[linewidth=.08,linestyle=dashed](14,6)(14,0)\pscircle*(14,3){.2}\rput(15.1,2.2){\sm{$\wt{v}_{23}$}}
\psline[linewidth=.08](11,3)(17,3)\rput(12.2,5){\sm{$\cN\io_{2;1}^{(2)}$}}
\psline[linewidth=.05]{->}(14,-1.5)(14,-5.5)\rput(14.7,-3.5){$\io$}
\psline[linewidth=.05]{<->}(18,3)(25,3)\rput(21.5,3.7){$\psi$}
\psline[linewidth=.07,linestyle=dashed](27,1)(31,5)\pscircle*(29,3){.2}\rput(30,2.2){\sm{$\wt{v}_{32}$}}
\psline[linewidth=.07](26,3)(32,3)\rput(29.2,5.5){\sm{$\cN\io_{2;1}^{(2)}$}}
\psline[linewidth=.05]{->}(29,-1.5)(29,-5.5)\rput(29.7,-3.5){$\io$}
\psline[linewidth=.09](-1,-7)(-1,-13)\pscircle*(-1,-10){.2}\rput(-1.8,-9.4){\sm{$\wt{x}_1$}}
\psline[linewidth=.08](-3,-12)(1,-8)
\psline[linewidth=.08](14,-7)(14,-13)\pscircle*(14,-10){.2}\rput(13.2,-9.2){\sm{$\wt{x}_2$}}
\psline[linewidth=.08](11,-10)(17,-10)
\psline[linewidth=.07](27,-12)(31,-8)\pscircle*(29,-10){.2}\rput(28.3,-9.2){\sm{$\wt{x}_3$}}
\psline[linewidth=.07](26,-10)(32,-10)
\end{pspicture}
\caption{The solid lines in the bottom row represent the image of a transverse immersion
\hbox{$\io\!:\wt{V}\!\lra\!\wt{X}$} around three double points $\wt{x}_i\!\in\!\wt{X}$.
The solid lines above each~$\wt{x}_i$ are the preimages of the two branches at~$\wt{x}_i$
in~$\wt{V}$; they are interchanged by the involution~$\psi$ as indicated.
The top and middle points in each column represent an element of $\wt{V}_{\io}^{(2)}$,
with the top point being its first $\wt{V}$-component.}
\label{wtVpsi_fig}
\end{figure}
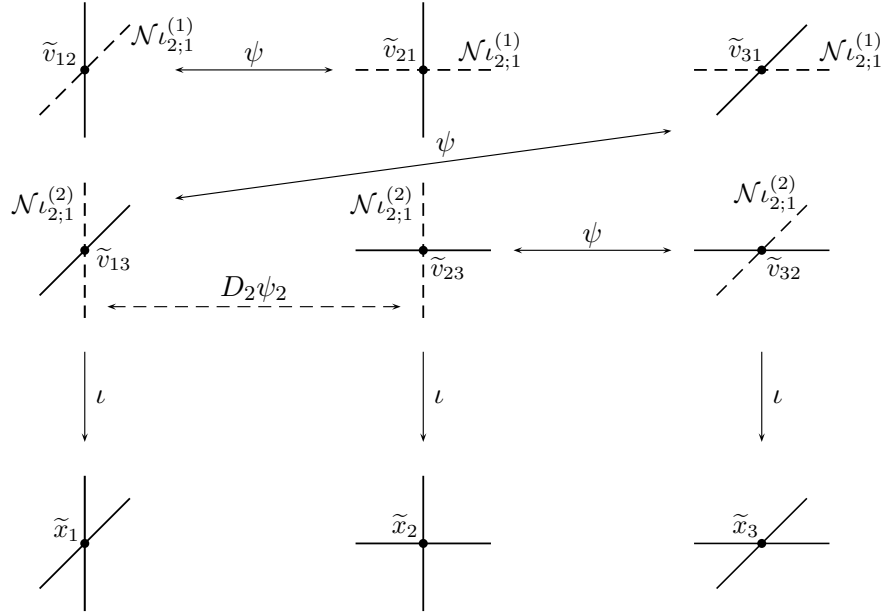

\noindent
By the commutativity of the diagram in~\eref{phikiio_e}, 
the identity on $q_{k+1;1}^*T\wt{X}$ and~$\nd\phi_{k-1}$ induce bundle
isomorphisms
$$D_k\!:\cN\io_k\lra \cN q_{k+1;1} \qquad\hbox{and}\qquad
D_i\phi_k\!:\cN\io_{k;k-1}^{(i)}\lra \cN q_{k+1;k}^{(i+1)}$$
covering~$\phi_k$.
By the commutativity of the right triangle in~\eref{phikiio_e} with $i$ replaced by 
$i\!-\!1$, $\nd q_{k;1}$ induces a bundle homomorphism
\BE{cNqdecompmaps_e} \cN q_{k+1;k}^{(i)} \lra \cN q_{k+1;1}\EE
over $\wt{X}_q^{(k+1)}$ for every $i\!\in\![k\!+\!1]\!-\!\{1\}$.
For all $\si\!\in\!\bS_k$ and $i\!\in\![k]$, the homomorphism~$\nd\si_i$ of 
the second diagram in~\eref{Vkdiag_e} with~$\io$ replaced by~$q$ induces an isomorphism
$$D_i{\si}\!:\cN q_{k;k-1}^{(i)}\lra \cN q_{k;k-1}^{(\si(i))}$$
covering the action of~$\si$ on~$\wt{X}_q^{(k)}$.\\

\noindent
For $k\!\ge\!2$, let 
\begin{gather*}
\cN q_k=\bigoplus_{i\in[k]}\cN q_{k;k-1}^{(i)}\,, \quad
\cN_i q_k=\!\!\bigoplus_{\begin{subarray}{c}j\in[k]\\ j\neq i\end{subarray}}\!
\cN q_{k;k-1}^{(j)}\subset\cN q_k\,
~~\forall\,i\!\in\![k].
\end{gather*}
By the commutativity of~\eref{phikiio_e}, the diagram
$$\xymatrix{ \bigoplus\limits_{i\in[k]}\!\!\cN\io_{k;k-1}^{(i)}  \ar[d]
\ar[rrr]^{\bigoplus\limits_{i\in[k]}\!\!\!D_i\phi_k} &&& \cN_1q_{k+1}\ar[d] \\
  \cN\io_k\ar[rrr]^{D_k} &&& \cN q_{k+1;1} }$$
commutes for every $k\!\in\!\Z^+$; the left and right vertical arrows above
are the isomorphisms induced by the homomorphisms~\eref{iodecompmaps_e} 
and~\eref{cNqdecompmaps_e}, respectively.
For $k\!\ge\!2$ and $\si\!\in\!\bS_k$, let
\BE{DsiXdfn_e}D\si\!\equiv\!\bigoplus_{i\in[k]}\!D_i\si: \cN q_k\lra \cN q_k\,.\EE
This bundle isomorphism preserves the splitting of~$\cN q_k$ and
takes~$\cN_i q_k$ to~$\cN_{\si(i)}q_k$.

\begin{dfn}\label{NCTransConfregul_dfn}
Suppose $(\io,\psi)$ is a compatible pair as in~\eref{iopsidfn_e}
so that $\io$ is a transverse closed immersion and $\psi$ is smooth.
Let $q\!:\wt{X}\!\lra\!X_{\io,\psi}$ be the associated quotient.
A \sf{regularization} for  $(\io,\psi)$ is a tuple $(\Psi_{k;i})_{k\in\Z^+,i\in[k]}$,
where each $\Psi_{k;i}$ is a regularization for the immersion~$q_k^{(i)}$,
such~that the tuple $(\Psi_{k+1;1}\!\circ\!D_k)_{k\in\Z^{\ge0}}$ 
is a refined regularization for~$\io$ as in Definition~\ref{NCTransCollregul_dfn}, 
\begin{alignat}{2}
\label{NCSCCregCond_e0}
q\!\circ\!\Psi_{k;i_1}\big|_{\cN_{i_2}q_k\cap\Dom(\Psi_{k;i_1})}
&=q\!\circ\!\Psi_{k;i_2}\big|_{\cN_{i_1}q_k\cap\Dom(\Psi_{k;i_2})}
&\qquad &\forall~i_1,i_2\!\in\![k],~k\!\in\!\Z^+, \\
\label{NCSCCregCond_e2}
\Psi_{k;i}&=\Psi_{k;\si(i)}\!\circ\!D\si\big|_{\Dom(\Psi_{k;i})}  &\qquad
&\forall~i\!\in\![k],~\si\!\in\!\bS_k,~k\!\ge\!2\,.
\end{alignat}
\end{dfn}

\begin{dfn}\label{NCSCCregul_dfn}
Suppose $(\wt{X},\wt\om)$ is a symplectic manifold, 
$(\io,\psi)$ is a compatible pair as in~\eref{iopsidfn_e}
so that $\io$ is a closed transverse immersion of codimension~2, $\psi$ is smooth,
$\io_k^*\wt\om$ is a symplectic form on~$\wt{V}_{\io}^{(k)}$ for all $k\!\in\!\Z^+$,
and $\psi^*\io^*\wt\om\!=\!\io^*\wt\om$.
An \sf{$\wt\om$-regularization for~$(\io,\psi)$} is a~tuple 
\BE{NCSCCregul_e}\fR\equiv \big(\cR_k)_{k\in\Z^+}\equiv
\big(\rho_{k;i},\na^{(k;i)},\Psi_{k;i}\big)_{k\in\Z^+,i\in[k]}\EE
such that $(\Psi_{k;i})_{k\in\Z^+,i\in[k]}$ is a regularization for~$(\io,\psi)$,
the~tuple 
\BE{NCSCCregul_e0}\big(\big( \{D_i\phi_k\}^*(\rho_{k+1;i+1},\na^{(k+1;i+1)})\big)_{i\in[k]}\big),
\Psi_{k+1;1}\!\circ\!D_k\big)_{k\in\Z^{\ge0}}\EE
is a refined $\wt\om$-regularization for the immersion~$\io$ as in Definition~\ref{NCSCDregul_dfn},
and 
\BE{NCSCCregul_e3}\big(\rho_{k;i},\na^{(k;i)}\big) = 
\{D_i\si\}^*\big(\rho_{k;\si(i)},\na^{(k;\si(i))}\big) \quad
\forall\,i\!\in\![k],\,\si\!\in\!\bS_k,\, k\!\in\!\Z^+.\EE
\end{dfn}

\vspace{.2in}

\noindent
The condition~\eref{NCSCCregCond_e0} replaces~\eref{SCCregCond_e0}.
By Proposition~\ref{NCD_prp}, the first condition on~$\wt\om$ in Definition~\ref{NCSCCregul_dfn}
is equivalent to $V\!\equiv\!\io(\wt{V})$ being an NC symplectic divisor in~$(\wt{X},\wt\om)$.
For a smooth family $(\wt\om_t)_{t\in B}$ of symplectic forms on~$\wt{X}$ satisfying
the conditions of Definition~\ref{NCSCCregul_dfn},
the notion of $\wt\om$-regularization naturally extends to a notion of 
\sf{$(\wt\om_t)_{t\in B}$-family of refined regularizations for~$(\io,\psi)$}.
We show in Section~\ref{NCCcomp_subs} that the NC symplectic varieties
of Definition~\ref{NCC_dfn} correspond to the compatible pairs~$(\io,\psi)$ and
symplectic forms~$\wt\om$ satisfying the conditions of Definition~\ref{NCSCCregul_dfn}.
An $\om$-regularization in the sense of Definition~\ref{NCCregul_dfn} likewise
corresponds  to an $\wt\om$-regularization for~$(\io,\psi)$.

\subsection{Local vs.~global perspective}
\label{NCCcomp_subs}

\noindent
We now show that the local and global notions of NC symplectic variety 
are equivalent.
If $(\wt{X},\wt\om)$ and $(\io,\psi)$ correspond to an NC variety~$(X,\om)$ as 
in Corollary~\ref{NCC_crl} below, 
then an \hbox{$\wt\om$-regularization} for~$(\io,\psi)$ in the sense of Definition~\ref{NCSCCregul_dfn}
determines an $\om$-regularization for~$X$ in the sense of 
Definition~\ref{NCCregul_dfn}\ref{NCCregul_it1}.
Conversely, the local regularizations constituting an $\om$-regularization for~$X$
can be patched together to form an $\wt\om$-regularization for~$(\io,\psi)$.
These relations also apply to families of regularizations.

\begin{lmm}\label{NCC_lmm} 
Suppose $\wt{X}$ and $\wt{V}$ are smooth manifolds and
$(\io,\psi)$ is a compatible pair as in~\eref{iopsidfn_e}
so that~$\io$ is a closed transverse immersion of codimension~2 and $\psi$ is 
a smooth involution.
The associated quotient space~$X_{\io,\psi}$ is then canonically
an NC variety.
Furthermore, every NC variety is isomorphic to $X_{\io,\psi}$ for some~$(\io,\psi)$
as above.
\end{lmm}

\begin{proof}
(1) Let $k\!\in\!\Z^+$, $V_{\io}^{(k)}\!\subset\!\wt{X}$ be as in~\eref{Xiotak_e}, and
$\wt{x}\!\in\!V_{\io}^{(k)}\!-\!V_{\io}^{(k+1)}$.
In light of~\eref{iopsiprop_e1}, we can write
\begin{equation*}\begin{split}
&\{\wt{x}\}\!\cup\!\io\big(\psi(\io^{-1}(\wt{x})\!)\big)\equiv\big\{\wt{x}_i\!:i\!\in\![k\!+\!1]\big\} 
\quad\hbox{and}\quad
\io^{-1}(\wt{x}_i)\equiv\big\{\wt{v}_{ij}\!:\,j\!\in\![k\!+\!1]\!-\!i\big\}~~\forall\,i\!\in\![k\!+\!1]\\
&\hspace{.8in}\hbox{with}\qquad \wt{x}_i\neq\wt{x}_{i'}~~\forall\,i\!\neq\!i', \quad
\wt{v}_{ij}\neq\wt{v}_{ij'}~~\forall\,j\!\neq\!j',\quad
\psi\big(\wt{v}_{ij}\big)=\wt{v}_{ji}\,;
\end{split}\end{equation*}
see Figure~\ref{wtVpsi_fig}.
By \cite[Proposition~1.35]{Warner}, the closedness of~$\io$, and
the continuity of~$\psi$,
there exist open neighborhoods $U_i\!\subset\!\wt{X}$ of~$\wt{x}_i$ with $i\!\in\![k\!+\!1]$ 
and neighborhoods $\wt{V}_{ij}\!\subset\!\wt{V}$ of~$\wt{v}_{ij}$ with 
$i,j\!\in\![k\!+\!1]$, $i\!\neq\!j$, such~that
$$\io^{-1}(U_i)=\bigsqcup_{j\in[k+1]-i}\!\!\!\!\!\!\!\!\wt{V}_{ij}\subset\wt{V},\quad
U_i\!\cap\!U_j=\eset, ~~~
\psi\big(\wt{V}_{ij}\big)=\wt{V}_{ji} \quad\forall~i,j\!\in\![k\!+\!1],~i\!\neq\!j,$$
and the restriction of $\io$ to each $\wt{V}_{ij}$ is injective.
Let $U\!=\!q(U_1)\!\cup\!\ldots\!q(U_{k+1})$.
Since
$$q^{-1}\big(q(U)\big)=\bigsqcup_{j\in[k+1]}\!\!\!\!\!U_j$$
by the above, $U$ is an open neighborhood of $q(\wt{x})$ in~$X_{\io,\psi}$.\\

\noindent
For $i\!\in\!I\!\subset\![k\!+\!1]$ and $j\!\in\![k\!+\!1]\!-\!i$, let 
$$U_{ij}=\io(\wt{V}_{ij})\subset U_i, \qquad
U_I=\bigcap_{j'\in I-i}\!\!\!\!U_{ij'}\subset U_i\,.$$
For each $i\!\in\![k\!+\!1]$, $\{U_{ij}\}_{j\neq i}$ 
is a transverse collection of closed submanifolds of~$U_i$
of codimension~2 such~that 
$$\io(\wt{V})\cap U_i= \bigcup_{j\in[k+1]-i}\!\!\!\!\!\!\!\!U_{ij}\,.$$
The diffeomorphisms
$$\io\!: \wt{V}_{ij}\lra U_{ij}, \qquad \psi\!: \wt{V}_{ij}\lra\wt{V}_{ji},
\qquad\hbox{and}\qquad \io\!: \wt{V}_{ji}\lra U_{ji}$$
with $i\!\neq\!j$ induce an identification $\psi_{ij}\!:U_{ij}\!\lra\!U_{ji}$
restricting to identifications on the submanifolds $U_I\!\subset\!U_{ij},U_{ji}$
whenever $i,j\!\in\!I\!\subset\![k\!+\!1]$.
By~\eref{iopsiprop_e2},
$$U=\bigg(\bigsqcup_{j\in[k+1]}\!\!\!\!\!U_j\bigg)\!\!\bigg/\!\!\!\sim
\qquad\hbox{with}\qquad 
U_{ij}\ni \wt{x}' \sim  \psi_{ij}(\wt{x}')\in U_{ji}
~~\forall\,\wt{x}'\!\in\!U_{ij},\,i\!\neq\!j\,.$$
Thus, $\X\!\equiv\!\{U_I\}_{I\in\cP^*(k+1)}$ is a transverse configuration 
such that  $U_{ij}$ is a closed submanifold of~$U_i$ of codimension~2
for all $i,j\!\in\![k\!+\!1]$ distinct and the restriction of~$q$ to the union of~$U_i$
descends to a homeomorphism
$$\vph\!: U\lra X_{\eset}\,.$$
The tuple $(U,\X,\vph)$ is then a chart around $q(\wt{x})$ in~$X_{\io,\psi}$
in the sense of Definition~\ref{ImmTransConf_dfn0}\ref{NCchart_it}.
Any two such charts overlap smoothly and thus generate an NC atlas for~$X_{\io,\psi}$.\\

\noindent
(2) Let $X$ be an NC variety.
Choose a locally finite collection 
$$\big(U_y,\X_y\!\equiv\!(X_{y;I})_{I\in\cP^*(N_y)},
\vph_y\!:U_y\!\lra\!X_{y;\eset}\big)_{y\in\cA'}$$
of NC charts covering~$X$. 
Let 
$$\wt{X}=\bigg(\bigsqcup_{y\in\cA'}\bigsqcup_{i\in[N_y]}\!\!\!\{(y,i)\}\!\times\!X_{y;i}
\bigg)\!\Big/\!\!\sim,$$
where we identify $(y,i,z)$ with $z\!\in\!X_{y;i}$
and  $(y',i',z')$ with $z'\!\in\!X_{y';i'}$~if there exist 
$x\!\in\!U_y\!\cap\!U_{y'}$ and an overlap map~$\vph_{yy';x}$ as in~\eref{ImmTransConf_e2}
such~that
\BE{NCvarglue_e}z'\in \vph_{y'}\big(U_{yy';x}\big),\quad
z=\vph_{yy';x}(z'),\quad
\vph_{yy';x}\big(\vph_{y'}\big(U_{yy';x}\big)\!\cap\!X_{y';i'}\big)
=\vph_y\big(U_{yy';x}\big)\!\cap\!X_{y;i}.\EE
Define
$$\wt{V}=\bigg(\bigsqcup_{y\in\cA'}\bigsqcup_{i\in[N_y]}\bigsqcup_{j\in[N_y]-i}\!\!\!\!\!\!
\{(y,i,j)\}\!\times\!X_{y;ij}\bigg)\!\Big/\!\!\sim,$$
where we identify $(y,i,j,z)$ with $z\!\in\!X_{y;ij}$
and  $(y',i',j',z')$ with $z'\!\in\!X_{y';i'j'}$~if 
there exist 
$x\!\in\!U_y\!\cap\!U_{y'}$ and an overlap map~$\vph_{yy';x}$ as in~\eref{ImmTransConf_e2}
such that~\eref{NCvarglue_e} holds~and
$$\vph_{yy';x}\big(\vph_{y'}(U_{yy';x})\!\cap\!X_{y';j'}\big)
=\vph_y\big(U_{yy';x}\big)\!\cap\!X_{y;j}\,.$$
The Hausdorffness of $X$ implies the Hausdorffness of $\wt{X}$ and~$\wt{V}$.
The last two spaces inherit smooth structures from the smooth structures of~$X_{y;i}$ and~$X_{y;ij}$. 
Define 
$$\io\!:\wt{V}\lra\wt{X}, \quad \io\big([y,i,j,z]\big)=[y,i,z], \qquad
\psi\!:\wt{V}\lra\wt{V}, \quad \psi\big([y,i,j,z]\big)=[y,j,i,z].$$
By the assumption on the charts in Definition~\ref{ImmTransConf_dfn}, 
$\io$ is a closed transverse immersion;
the smooth map~$\psi$ is an involution.
It is immediate that the pair~$(\io,\psi)$ is compatible.
The well-defined~map
$$X\lra X_{\io,\psi}, \qquad x\lra \big[y,i,\vph_y(x)\big]\quad
\forall~x\in\vph_y^{-1}(X_{y;i}),~i\!\in\![N_y],~y\!\in\!\cA',$$
is an isomorphism of NC varieties.
\end{proof}

\noindent
Following the usual terminology of algebraic geometry, we call the map
$$q\!: \wt{X}\lra X_{\io,\psi}=X$$
provided by the last statement of Lemma~\ref{NCC_lmm} the \sf{normalization}
of the NC variety~$X$; it is unique up to isomorphism.
By the proof of Lemma~\ref{NCC_lmm}, $q$ pulls back 
a symplectic structure~$\om$ on~$X_{\io,\psi}$ to a symplectic form~$\wt\om$
on~$\wt{X}$.

\begin{crl}\label{NCC_crl}
Suppose $(\io,\psi)$ is a compatible pair as in~\eref{iopsidfn_e}
so that~$\io$ is a closed transverse immersion of codimension~2 and $\psi$ is 
a smooth involution.
A symplectic structure~$\om$ on the NC variety~$X_{\io,\psi}$ corresponds 
to a symplectic form~$\wt\om$ on~$\wt{X}$ such that $\io(\wt{V})$ is an NC symplectic divisor
in~$(\wt{X},\wt\om)$ and $\psi^*\io^*\wt\om\!=\!\io^*\wt\om$.
\end{crl}

\begin{proof}
By the proof of Lemma~\ref{NCC_lmm}, a symplectic structure~$\om$ on~$X_{\io,\psi}$
corresponds to a symplectic form~$\wt\om$ on~$\wt{X}$ such that 
$\io_k^*\wt\om$ is a symplectic form on $\wt{V}_{\io}^{(k)}$ for all $k\!\in\!\Z^+$,
the intersection and \hbox{$\om$-orientations} of~$\wt{V}_{\io}^{(k)}$ are the same,
and $\psi^*\io^*\wt\om\!=\!\io^*\wt\om$.
Thus, the claim follows from Proposition~\ref{NCD_prp}.
\end{proof}

\noindent
The local perspective on NC varieties in Section~\ref{NCCloc_subs} leads
to natural notions of smoothness, immersion, and transverse immersion.
These notions in turn make it possible to adapt the considerations of 
Section~\ref{EquivalentDfn_subs} to NC varieties.
The codimension~2 condition in Definition~\ref{ImmTransConf_dfn} is not material
for these considerations.

\begin{dfn}\label{NCRes2_dfn}
Let $X$ be an NC variety as in Definition~\ref{ImmTransConf_dfn}. 
A~\sf{resolution of~$X$} is a sequence of closed immersions
\BE{absNC2_e}
\ldots \lra X^{(k)}\xlra{f_{k;k-1}}X^{(k-1)}
\xlra{f_{k-1;k-2}}\cdots \xlra{f_{2;1}}  X^{(1)}
\xlra{f_{1;0}} X^{(0)}\!\equiv\!X\EE 
such that
\begin{enumerate}[label=(R\arabic*),leftmargin=*]

\setcounter{enumi}{-1}

\item $f_{1;0}$ is a transverse surjective immersion; 

\item for every $k\!\in\!\Z^+$, $X^{(k)}$ is a manifold with a free $\bS_k$-action
and the codimension of~$f_{k;k-1}$ is~2;

\item\label{Requiv2_it} for all $0\!\le\!k'\!\le\!k$, the map 
\BE{absNC2_e2} 
f_{k;k'}\!\equiv\!f_{k'+1;k'}\!\circ\!\ldots\!\circ\!f_{k;k-1}\!: X^{(k)}\!\lra\!X^{(k')}\EE
is $\bS_{k'}$-equivariant, and
 
\item\label{inverse2_it} for all $0\!\le\!k'\!\le\!k$ and 
$x\!\in\!X^{(k)}\!-\!f_{k+1;k}(X^{(k+1)})$,
$f_{k;k'}^{-1}(f_{k;k'}(x))=\bS_{[k]-[k']}\!\cdot\!x$. 

\end{enumerate}
\end{dfn}

\vspace{.1in}

\noindent
By the reasoning after Definition~\ref{NCRes_dfn}, the assumptions in Definition~\ref{NCRes2_dfn}
imply that the map~\eref{absNC2_e2} is $\bS_{[k]-[k']}$-invariant.
If $(\io,\phi)$ is a compatible pair as in~\eref{iopsidfn_e} with $X\!=\!X_{\io,\psi}$,
then the sequence 
$$\ldots \lra \wt{X}_q^{(k)}\xlra{q_{k;k-1}}\wt{X}_q^{(k-1)}
\xlra{q_{k-1;k-2}}\cdots \xlra{q_{2;1}}  \wt{X}_q^{(1)}
\xlra{q_{1;0}} \wt{X}_q^{(0)}\!\equiv\!X$$
of immersions constructed in Section~\ref{NCCgl_subs} is a resolution of~$X$.
By the reasoning in the proof of Proposition~\ref{2TermToSeq_prp},
for any other resolution~\eref{absNC2_e} of~$X$ 
there exist unique smooth maps~$h_k$ so that the  diagram
\BE{absNC2_e10}\begin{split}
\xymatrix{\ldots\ar[r]& X^{(k)}\ar[d]^{h_k}\ar[rr]^{f_{k;k-1}}&&
X^{(k-1)}\ar[d]^{h_{k-1}}\ar[rr]^{f_{k-1;k-2}}&& \cdots \ar[r]^{f_{2;1}}& X^{(1)} 
\ar[d]^{h_1}\ar[r]^>>>>>{f_{1;0}}& X^{(0)}\!=\!X\ar[d]^{\id}\\
\ldots\ar[r] &\wt{X}_q^{(k)}\ar[rr]^{q_{k;k-1}}&&
\wt{X}_{\io}^{(k-1)} \ar[rr]^{q_{k-1;k-2}}&& \cdots\ar[r]^{q_{2;1}} & 
\wt{X}_{\io}^{(1)} \ar[r]^>>>>>{q_{1;0}}& \wt{X}_q^{(0)}\!=\!X} 
\end{split}\EE
commutes. These maps are $\bS_k$-equivariant diffeomorphisms.

\subsection{Another global perspective}
\label{NCgl2_subs}

\noindent
By~\eref{NCSCCregCond_e2}, a regularization $(\Psi_{k;i})_{k\in\Z^+,i\in[k]}$
for $(\io,\psi)$ is determined by the refined regularization 
$(\Psi_{k+1;1}\!\circ\!D_{k;1})_{k\in\Z^{\ge0}}$ for~$\io$.
By the proof of Proposition~\ref{NCCgl_prp} below,
which re-interprets the notion of $\wt\om$-regularization for~$(\io,\psi)$
provided by Definition~\ref{NCSCCregul_dfn} 
in terms of the notion of refined $\wt\om$-regularization for~$\io$
provided by Definition~\ref{NCSCDregul_dfn},
the condition~\eref{NCSCCregCond_e0} is equivalent to the \hbox{$\psi$-equivariance} 
condition~\eref{NCCgl_e2} on $(\Psi_{k+1;1}\!\circ\!D_{k;1})_{k\in\Z^{\ge0}}$.\\

\noindent
For $k\!\in\!\Z^+$ and $i,j\!\in\![k]$ distinct, let $\si_{k;ij}\!\in\!\bS_k$ 
be the transposition interchanging $i,j\!\in\![k]$.
Let $(\io,\psi)$ be a compatible pair as in~\eref{iopsidfn_e}.
The bijection~$\phi_k$ in~\eref{tiXprtX_e} is equivariant with respect to the inclusion of 
$\bS_k$ into $\bS_{k+1}$ induced by the inclusion 
$$[k]\lra[k\!+\!1], \qquad i\lra i\!+\!1;$$
the first inclusion identifies $\bS_k$ with the subgroup~$\bS_{[k+1]-[1]}$ of~$\bS_{k+1}$.
The action of~$\si_{k+1;12}$ on~$\wt{X}_{\io}^{(k+1)}$ 
corresponds to an involution~$\psi_k$ on~$\wt{V}_{\io}^{(k)}$ via the bijection~$\phi_k$, i.e.
\BE{psikdfn_e}\psi_k\!: \wt{V}_{\io}^{(k)}\lra \wt{V}_{\io}^{(k)}, \qquad
\phi_k\!\circ\!\psi_k\equiv \si_{k+1;12}\!\circ\!\phi_k, \qquad
\psi_k^2=\id_{\wt{V}_{\io}^{(k)}}\,.\EE
This involution extends the natural $\bS_k$-action on~$\wt{V}_{\io}^{(k)}$ to 
an $\bS_{k+1}$-action, under the above identification of~$\bS_k$ with~$\bS_{[k+1]-[1]}$,
so that $\phi_k$ becomes $\bS_{k+1}$-equivariant.\\

\noindent
We note~that 
$$\psi_1\!\approx\!\psi\!: 
\wt{V}_{\io}^{(1)}\!\approx\!\wt{V}\lra\wt{V}\!\approx\!\wt{V}_{\io}^{(1)}$$
and that the diagram 
\BE{iopsidiag_e}\begin{split}
\xymatrix{ \wt{V}_{\io}^{(k)} \ar[rr]^{\psi_k} \ar[d]_{\io_{k;k-1}^{(i)}} 
&& \wt{V}_{\io}^{(k)} \ar[d]^{\io_{k;k-1}^{(i)}}\\
\wt{V}_{\io}^{(k-1)} \ar[rr]^{\psi_{k-1}} && \wt{V}_{\io}^{(k-1)}}
\end{split}\EE
commutes for all $k\!\in\!\Z^+$ and $i\!\in\![k]\!-\!\{1\}$.
By the $\bS_{k+1}$-equivariance of~$\phi_k$ and the commutativity of
the square in~\eref{phikiio_e}, the map~\eref{whiokk_e}
is $\bS_{k'+1}$-equivariant and $\bS_{[k+1]-[k'+1]}$-invariant with respect 
to the extended actions for every $k'\!\in\![k]$.\\

\noindent
Under the assumptions of Lemma~\ref{iopsi_lmm}, the involution~$\psi_k$ is smooth.
For each $i\!\in\![k]\!-\!\{1\}$, the homomorphism
\BE{Djpsi_e} D_i\psi_k\!: \cN\io_{k;k-1}^{(i)}\lra \cN\io_{k;k-1}^{(i)}\EE
induced by the commutativity of~\eref{iopsidiag_e} in this case is an isomorphism.
Let 
$$D\psi_k\!\equiv\!\bigoplus_{i\in[k]-\{1\}}\!\!\!\!\!\!\!D_i\psi_k\!:
\cN_{k;1}\io\!\equiv\!\bigoplus_{i\in[k]-\{1\}}\!\!\!\!\!\!\cN\io_{k;k-1}^{(i)} 
\lra \cN_{k;1}\io\,.$$

\begin{prp}\label{NCCgl_prp}
Suppose $(\wt{X},\wt\om)$ and  $(\io,\psi)$ are as in Definition~\ref{NCSCCregul_dfn}.
An  $\wt\om$-regularization for~$(\io,\psi)$
is equivalent to a refined $\wt\om$-regularization 
\BE{NCCgl_e0}\wh\fR\equiv\big(\wh\cR_k\big)_{k\in\Z^{\ge0}}\equiv
\big((\wh\rho_{k;i},\wh\na^{(k;i)})_{i\in[k]},\wh\Psi_k\big)_{k\in\Z^{\ge0}}\EE
for~$\io$ as in~\eref{NCSCDregul_e} such~that 
\begin{gather}
\label{NCCgl_e1}
\big(\wh\rho_{k;i},\wh\na^{(k;i)}\big)
=\big\{D_i\psi_k\big\}^*\big(\wh\rho_{k;i},\wh\na^{(k;i)}\big) 
\qquad \forall~i\!\in\![k]\!-\!\{1\},~k\!\in\!\Z^+,\\ 
\label{NCCgl_e2}
\psi\!\circ\!\wh\Psi_{k;1}=\wh\Psi_{k;1}\!\circ\!D\psi_k
~~\hbox{on}~~\cN_{k;1}'\io=D\psi_k(\cN_{k;1}'\io)
\qquad \forall~k\!\in\!\Z^+,
\end{gather}
with $\cN_{k;1}'\io\!\equiv\!\cN_{k;1}\io\!\cap\!\Dom(\wh\Psi_k)$ and
$\wh\Psi_{k;1}\!:\cN_{k;1}'\io\!\lra\!\wt{V}_{\io}^{(1)}\!\approx\!\wt{V}$ as in~\eref{Psikkprdfn_e}.
\end{prp}

\begin{proof} Suppose $(\Psi_{k;i})_{k\in\Z^+,i\in[k]}$ is a tuple of 
regularizations for the immersions~$q_k^{(i)}$ as in Definition~\ref{NCTransConfregul_dfn}.
Let $k\!\in\!\Z^+$.
The $\bS_k$-invariance condition~\eref{NCSCCregCond_e2} is equivalent to the condition
\BE{Psik1i_e2}
\Psi_{k;i}=\Psi_{k;1}\!\circ\!D\si\!\circ\!D\si_{k;1i}\big|_{D\si_{k;1i}(\Dom(\Psi_{k;1}))}
\quad\forall\,i\!\in\![k],\,\si\!\in\!\bS_{[k]-[1]}.\EE
Since $D_{k-1}$ is equivariant with respect to the natural identification of~$\bS_{k-1}$ 
with~$\bS_{[k]-[1]}$, the above condition and~\eref{NCSCCregCond_e2} are equivalent~to
\BE{Psik1i_e2b}\begin{aligned}
\big\{\!\Psi_{k;1}\!\circ\!D_{k-1}\big\}
&=\big\{\!\Psi_{k;1}\!\circ\!D_{k-1}\big\}\!\circ\!D\!\si\big|_{D_{k-1}^{\,-1}(\Dom(\Psi_{k;1}))}
&\quad&\forall\,\si\!\in\!\bS_{k-1},\\
\Psi_{k;i}&=\Psi_{k;1}\!\circ\!D\si_{k;1i}\big|_{D\si_{k;1i}(\Dom(\Psi_{k;1}))}
&\quad&\forall\,i\!\in\![k]\,.
\end{aligned}\EE
The condition~\eref{Psik1i_e2} implies \eref{NCSCCregCond_e0} for all $i_1,i_2\!\in\![k]\!-\!\{1\}$.
Suppose $k\!\ge\!2$ and~\eref{Psik1i_e2} holds. 
The full condition~\eref{NCSCCregCond_e0} is then equivalent~to
$$q\!\circ\!\Psi_{k;1}\big|_{\cN_2q_k\cap\Dom(\Psi_{k;1})}
=q\!\circ\!\Psi_{k;1}\!\circ\!D\si_{k;12}
\big|_{\cN_1q_k\cap D\si_{k;12}(\Dom(\Psi_{k;1}))}\,.$$
By the middle statement in~\eref{psikdfn_e} for $k\!-\!1$ instead of~$k$ and by
the denseness of $V_{\io}^{(1)}\!-\!V_{\io}^{(2)}$ in \hbox{$V_{\io}^{(1)}\!\subset\!\wt{X}$},
\eref{iopsiprop_e2} implies that the above condition and~\eref{NCSCCregCond_e0} 
are equivalent~to
\BE{Psik1i_e4}\begin{split}
&\psi\!\circ\!\big\{\!\Psi_{k;1}\!\circ\!D_{k-1}\big\}
\big|_{\cN_{k-1;1}\io\cap D_{k-1}^{\,-1}(\Dom(\Psi_{k;1}))}\\
&\hspace{1in}=\big\{\!\Psi_{k;1}\!\circ\!D_{k-1}\big\}\!\circ\!D\psi_{k-1}
\big|_{D\psi_{k-1}(\cN_{k-1;1}\io\cap D_{k-1}^{\,-1}(\Dom(\Psi_{k;1})))}\,.
\end{split}\EE
Thus, a regularization $(\Psi_{k;i})_{k\in\Z^+,i\in[k]}$ for  $(\io,\psi)$
as in Definition~\ref{NCTransConfregul_dfn} determines a refined regularization 
$(\wh\Psi_k\!\equiv\!\Psi_{k+1;1}\!\circ\!D_{k;1})_{k\in\Z^{\ge0}}$ 
for~$\io$ as in Definition~\ref{NCTransCollregul_dfn} satisfying~\eref{NCCgl_e2}.\\

\noindent
If a tuple $(\wh\Psi_k)_{k\in\Z^{\ge0}}$ is a refined regularization for~$\io$, then
$$\wh\Psi_{k-1}=\wh\Psi_{k-1}\!\circ\!D\!\si\big|_{\Dom(\wh\Psi_{k-1})}
\qquad\forall~\si\!\in\!\bS_{k-1},~k\!\in\!\Z^+\,.$$
Via the second identity in~\eref{Psik1i_e2b} with 
$$\Psi_{k;1}\equiv \wh\Psi_k\!\circ\!D_{k-1}^{-1}\big|_{D_{k-1}(\Dom(\wh\Psi_{k-1}))},$$
the tuple $(\wh\Psi_k)_{k\in\Z^{\ge0}}$ thus determines a tuple $(\Psi_{k;i})_{k\in\Z^+,i\in[k]}$ 
of regularizations for the immersions~$q_k^{(i)}$ satisfying~\eref{NCSCCregCond_e2}. 
If the first tuple satisfies~\eref{NCCgl_e2},
the second tuple satisfies~\eref{Psik1i_e4} and thus~\eref{NCSCCregCond_e0}.\\

\noindent
Suppose $\fR$ is a tuple as in~\eref{NCSCCregul_e} so that $(\Psi_{k;i})_{k\in\Z^+,i\in[k]}$ 
is a regularization for~$(\io,\psi)$.
Let $(\wh\Psi_k)_{k\in\Z^{\ge0}}$ be the corresponding refined regularization for~$\io$ and
$$\big(\wh\rho_{k;i},\wh\na^{(k;i)}\big)
=\big\{D_i\phi_k\big\}^*\big(\rho_{k+1;i+1},\na^{(k+1;i+1)}\big)
\quad\forall\,i\!\in\![k],\,k\!\in\!\Z^+.$$
Since $D_{k-1}$ is equivariant with respect to the natural identification of~$\bS_{k-1}$ 
with~$\bS_{[k]-[1]}$, 
the condition that~\eref{NCSCCregul_e0} is a refined $\wt\om$-regularization implies that
\BE{Psik1i_e12a}
\big(\rho_{k;i},\na^{(k;i)}\big)=\big\{\!D_i\si\big\}^{\!*}
\big(\rho_{k;\si(i)},\na^{(k;\si(i))}\big)
~~\forall\,i\!\in\![k],\,\si\!\in\!\bS_{[k]-[1]},\,k\!\in\!\Z^+\,.\EE
The condition~\eref{NCSCCregul_e3} is equivalent~to
$$\big(\rho_{k;i},\na^{(k;i)}\big) = 
\big\{\!D_i\si_{k;1i}\big\}^{\!*}
\big\{\!D_1\si\big\}^{\!*}\big\{\!D_1\si_{k;12}\big\}^{\!*}
\big(\rho_{k;2},\na^{(k;2)}\big)
\quad\forall\,i\!\in\![k],\,\si\!\in\!\bS_{[k]-[1]},\,k\!\in\!\Z^+.$$
This condition is in turn equivalent~to
\BE{Psik1i_e12b}\begin{aligned}
\big(\rho_{k;2},\na^{(k;2)}\big)
&=\big\{\!D_2\si_{k;12}\big\}^{\!*}\big\{\!D_1\si\big\}^{\!*}
\big\{\!D_1\si_{k;12}\big\}^{\!*}\big(\rho_{k;2},\na^{(k;2)}\big)
&~~&\forall\,\si\!\in\!\bS_{[k]-[1]},\,k\!\in\!\Z^+,\\
\big(\rho_{k;i},\na^{(k;i)}\big)&
=\big\{\!D_i\si_{k;1i}\big\}^{\!*}\big\{\!D_1\si_{k;12}\big\}^{\!*}
\big(\rho_{k;2},\na^{(k;2)}\big)&~~&\forall\,i\!\in\![k],\,k\!\in\!\Z^+\,.
\end{aligned}\EE
Since $\si_{k;12}$ commutes with all elements of $\bS_{[k]-[2]}$ and
$$\si_{k;12}\si_{k;2i}\si_{k;12}=\si_{k;1i}=\si_{k;2i}\si_{k;12}\si_{k;2i}
\quad\forall\,i\!\in\![k]\!-\![2],\,k\!\in\!\Z^+,$$
the condition \eref{Psik1i_e12a} implies that
the first condition in~\eref{Psik1i_e12b} is equivalent to 
$$\big(\rho_{k;i},\na^{(k;i)}\big)=\big\{\!D_i\si_{k;12}\big\}^{\!*}\big(\rho_{k;i},\na^{(k;i)}\big)
\quad\forall\,i\!\in\![k]\!-\![2],\,k\!\in\!\Z^+\,.$$
By the middle statement in~\eref{psikdfn_e} for $k\!-\!1$ instead of~$k$,
the above condition and~\eref{NCSCCregul_e3} are equivalent~to
$$\big(\wh\rho_{k-1;i},\wh\na^{(k-1;i)}\big)
=\big\{D_i\psi_{k-1}\big\}^*\big(\wh\rho_{k-1;i},\wh\na^{(k-1;i)}\big) 
\qquad \forall~i\!\in\![k\!-\!1]\!-\!\{1\},~k\!\in\!\Z^+.$$
Thus, an $\wt\om$-regularization for~$(\io,\psi)$ as in Definition~\ref{NCSCCregul_dfn} 
determines a refined $\wt\om$-regularization~\eref{NCCgl_e0} for~$\io$ satisfying~\eref{NCCgl_e1}.\\

\noindent
Conversely, a refined $\wt\om$-regularization~\eref{NCCgl_e0} 
determines a tuple~\eref{NCSCCregul_e} such that 
 $(\Psi_{k;i})_{k\in\Z^+,i\in[k]}$ is a regularization for~$(\io,\psi)$
and the associated tuple~\eref{NCSCCregul_e0}
is a refined $\wt\om$-regularization for the immersion~$\io$ 
via the second identity in~\eref{Psik1i_e12b} with
$$\big(\rho_{k;2},\na^{(k;2)}\big)\equiv 
\big\{(D_1\phi_{k-1})^{-1}\}^*\big(\wh\rho_{k-1;1},\wh\na^{(k-1;1)}\big).$$
If the tuple~\eref{NCCgl_e0} satisfies~\eref{NCCgl_e1}, then
the tuple~\eref{NCSCCregul_e} satisfies~\eref{NCSCCregul_e3}.
\end{proof}

\subsection{Examples}
\label{eg_subs}

\noindent
We now give some examples of NC divisors and varieties.

\begin{eg}\label{NCD_eg0}
Let $X$ be a manifold and $\{V_i\}_{i\in S}$ be a finite transverse collection of 
closed submanifolds of~$X$ of codimension~2.
The associated normalization 
$$\io\!:\wt{V}\equiv\bigsqcup_{i\in S}\!V_i \lra X$$
as in Lemma~\ref{NCD_lmm} is induced by the inclusions $V_i\!\lra\!X$ and
$$\wt{V}_{\io}^{(k)}=
\bigsqcup_{\begin{subarray}{c}I\subset S\\ |I|=k\end{subarray}}
\bigsqcup_{\tau\in\Aut(I)}
\!\!\!\!\!\!\big\{\big(x,(\tau(i),x)_{i\in I}\big)\!:\,x\!\in\!V_I\big\}$$
is the disjoint union of $k!$ copies of the disjoint union of 
the submanifolds~$V_I$ with $|I|\!=\!k$. 
\end{eg}

\begin{eg}\label{NCC_eg0}
Let $\X\!\equiv\!(X_I)_{I\in\cP^*(N)}$ is an $N$-fold transverse configuration
in the sense of Definition~\ref{TransConf_dfn1} such that 
$X_{ij}$ is a closed submanifold of~$X_i$ of codimension~2 for all $i,j\!\in\![N]$ distinct,
$$\wt{X}=\bigsqcup_{i\in[N]}\!\!\{i\}\!\times\!X_i\,, \qquad
\wt{V}= \bigsqcup_{i\in[N]}\bigsqcup_{j\in[N]-i}\!\!\!\!\{(i,j)\}\!\times\!X_{ij},$$
and $q\!:\wt{X}\!\lra\!X_{\eset}$ be the natural quotient map.
The normalization of the preimage of the singular locus $X_{\prt}\!\subset\!X_{\eset}$
in~$\wt{X}$,
$$\io\!:\wt{V}\lra 
q^{-1}\big(X_{\prt}\big)=\bigsqcup_{i\in[N]}\bigcup_{j\in[N]-i}\!\!\!\!\{i\}\!\times\!X_{ij}\,,$$
is induced by the inclusions $X_{ij}\!\lra\!X_i$. 
The associated involution~is
$$\psi\!:\wt{V}\lra\wt{V}, \qquad (i,j,x)\lra (j,i,x).$$
\end{eg}

\begin{eg}\label{NCDvsC_eg}
An NC symplectic divisor $V$ in $(X,\om)$ gives rise to an NC symplectic variety
as follows.
Let $\io\!:\wt{V}\!\lra\!X$ be the associated closed transverse immersion
as in Lemma~\ref{NCD_lmm},
$$(\wt{X},\wt\om)=(X,\om)\sqcup(\wt{V}\!\times\!\C,\pi_1^*\io^*\om\!+\!\pi_2^*\om_{\C}), 
\qquad
\wt{V}'=\{1\}\!\times\!\wt{V}\sqcup\{2\}\!\times\!\wt{V}\sqcup\wt{V}_{\io}^{(2)}\!\times\!\C,$$
where $\pi_1,\pi_2\!:\wt{V}\!\times\!\C\!\lra\!\wt{V},\C$ are the two projection maps
and $\om_{\C}$ is the standard symplectic form on~$\C$.
We define a closed transverse immersion $\io'\!:\wt{V}'\!\lra\!\wt{X}$ and 
an involution~$\psi$ on~$\wt{V}'$ by
\begin{alignat*}{3} 
\io'(1,\wt{v})&=\io(\wt{v})\in X, &\quad
\io'(2,\wt{v})&=(\wt{v},0)\in\wt{V}\!\times\!\C,&\quad
\io'\big((x,\wt{v}_1,\wt{v}_2),c\big)&=(\wt{v}_1,c)\in\wt{V}\!\times\!\C,\\
\psi(1,\wt{v})&=(2,\wt{v}), &\quad
\psi(2,\wt{v})&=(1,\wt{v}),&\quad
\psi\big((x,\wt{v}_1,\wt{v}_2),c\big)&=\big((x,\wt{v}_2,\wt{v}_1),c\big);
\end{alignat*}
see Figure~\ref{NCDvsC_fig}.
The pair $(\io',\psi)$ satisfies the three conditions in~\eref{iopsiprop_e1} with $\io$
replaced by~$\io'$ 
and thus determines an NC variety~$X_{\io',\psi}$.
Since $\psi^*\io'^*\wt\om\!=\!\io'^*\wt\om$, 
$X_{\io',\psi}$ is an NC symplectic variety by Corollary~\ref{NCC_crl}.
\end{eg}

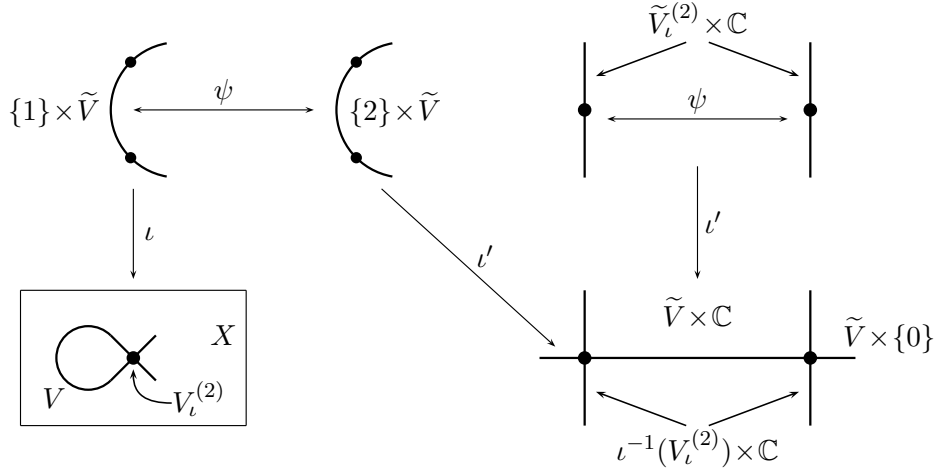
\begin{figure}
\begin{pspicture}(-4,-1.5)(11,4.5)
\psset{unit=.3cm}
\psarc[linewidth=.1](2,11){3}{100}{260}\psarc[linewidth=.1](12,11){3}{100}{260}
\pscircle*(-.12,13.12){.25}\pscircle*(-.12,8.88){.25}
\pscircle*(9.88,13.12){.25}\pscircle*(9.88,8.88){.25}
\rput(-3.5,11){$\{1\}\!\times\!\wt{V}$}\rput(11.6,11){$\{2\}\!\times\!\wt{V}$}
\psline[linewidth=.05]{<->}(0,11)(8,11)\rput(4,11.8){$\psi$}
\psline[linewidth=.05]{->}(0,7.5)(0,3.5)\rput(0.7,5.5){$\io$}
\psline[linewidth=.05]{->}(11,7.5)(18.7,0.5)\rput(15.6,4.5){$\io'$}
\psline[linewidth=.05]{->}(25,8.5)(25,3.5)\rput(25.7,6){$\io'$}
\pscircle*(20,11){.3}\pscircle*(30,11){.3}
\psline[linewidth=.1](20,8)(20,14)\psline[linewidth=.1](30,8)(30,14)
\rput(25,15){$\wt{V}_{\io}^{(2)}\!\times\!\C$}
\pnode(24.5,14){A1}\pnode(20.5,12.5){B1}\ncline[linewidth=.07]{->}{A1}{B1}
\pnode(25.5,14){A1}\pnode(29.5,12.5){B1}\ncline[linewidth=.07]{->}{A1}{B1}
\psline[linewidth=.05]{<->}(21,10.6)(29,10.6)\rput(25,11.4){$\psi$}
\psline[linewidth=.1](-1,-1)(1,1)\psline[linewidth=.1](-1,1)(1,-1)
\pscircle*(0,0){.3}\psarc[linewidth=.1](-2,0){1.41}{45}{-45}
\psline[linewidth=.05](-5,-3)(5,-3)\psline[linewidth=.05](-5,3)(5,3)
\psline[linewidth=.05](-5,-3)(-5,3)\psline[linewidth=.05](5,3)(5,-3)
\rput(4,1){$X$}\rput(-3.5,-1.8){$V$}\rput(2.9,-1.8){$V_{\io}^{(2)}$}
\pnode(0,-.5){B}\pnode(1.7,-2){A}
\nccurve[linewidth=.07,angle=180,angleB=-90,ncurv=1]{->}{A}{B}
\pscircle*(20,0){.3}\pscircle*(30,0){.3}
\psline[linewidth=.1](20,-3)(20,3)\psline[linewidth=.1](30,-3)(30,3)
\psline[linewidth=.1](18,0)(32,0)
\rput(25,2){$\wt{V}\!\times\!\C$}\rput(25,-4){$\io^{-1}(V_{\io}^{(2)})\!\times\!\C$}
\pnode(24.5,-3){A1}\pnode(20.5,-1.5){B1}\ncline[linewidth=.07]{->}{A1}{B1}
\pnode(25.5,-3){A1}\pnode(29.5,-1.5){B1}\ncline[linewidth=.07]{->}{A1}{B1}
\rput(33.5,1){$\wt{V}\!\times\!\{0\}$}
\end{pspicture}
\caption{The normalization $\wt{X}\!\equiv\!X\!\sqcup\!(\wt{V}\!\times\!\C)$
of the NC variety~$X_{\io',\psi}$ associated with an NC divisor $V\!\subset\!X$ 
as in Example~\ref{NCDvsC_eg}.}
\label{NCDvsC_fig}
\end{figure}

\begin{eg}\label{NCC_eg1}
A generalization of the 2-fold SC symplectic configuration of Example~\ref{NCC_eg0}
is obtained by taking two disjoint copies, $V_1$ and~$V_2$, of a smooth symplectic divisor~$V$
in the same symplectic manifold~$(\wt{X},\wt\om)$.
Let $\psi\!:V_1\!\lra\!V_2$ be a symplectomorphism
and $\psi\!:V_2\!\lra\!V_1$ be its inverse; thus, $\psi$ is an involution
on $\wt{V}\!\equiv\!V_1\!\sqcup\!V_2$.
In~this case, the normalization 
$$\io\!:\wt{V}\lra V\!\equiv\!V_1\!\cup\!V_2\subset \wt{X}$$
is just the inclusion into~$\wt{X}$.
The pair~$(\io,\psi)$ satisfies the three conditions in~\eref{iopsiprop_e1}
and thus determines an NC variety~$X_{\io,\psi}$;
it is obtained by identifying $V_1$ with~$V_2$ in~$\wt{X}$ via~$\psi$.
The singular locus~$X_{\prt}$ in this case can be identified 
with~$(V,\om|_V)$.
\end{eg}

\begin{eg}\label{NCC_eg2}
A more elaborate 2-fold NC symplectic variety
is obtained by taking~$V$ in Example~\ref{NCC_eg1} to be 
any closed symplectic submanifold of~$(\wt{X},\wt\om)$ and $\psi\!:V\!\lra\!V$
to be any symplectomorphism without fixed points such that $\psi\!\circ\!\psi\!=\!\id_V$.
The normalization
$$\io\!:\wt{V}\!\equiv\!V\lra V\subset\wt{X}$$
is again just the inclusion.
The pair~$(\io,\psi)$ satisfies the three conditions in~\eref{iopsiprop_e1}
and thus determines an NC variety~$X_{\io,\psi}$;
it is obtained by ``folding" $X$ along~$V$ as directed by~$\psi$. 
The singular locus~$X_{\prt}$ in this case is the quotient of~$V$ 
by the $\Z_2$-action determined by~$\psi$.
\end{eg}

\begin{figure}
\begin{pspicture}(-4,-2.5)(11,2)
\psset{unit=.3cm}
\pscircle[linewidth=.1](0,-3.08){4.14}\rput(8,-5.38){\sm{$\wt{V}$}}
\psarc[linewidth=.06]{->}(0,-3.08){2.5}{60}{240}\rput(-1,-2.5){\sm{$\psi$}}
\pscircle*(0,1.06){.25}\pscircle*(0,-7.22){.25}
\pscircle*(3.59,-1.01){.25}\pscircle*(3.59,-5.15){.25}
\pscircle*(-3.59,-1.01){.25}\pscircle*(-3.59,-5.15){.25}
\rput(4.8,-1){\sm{$\wt{v}_{12}$}}\rput(0.1,2){\sm{$\wt{v}_{13}$}}
\rput(-4.7,-1){\sm{$\wt{v}_{23}$}}\rput(-4.3,-5.6){\sm{$\wt{v}_{21}$}}
\rput(0.1,-8.2){\sm{$\wt{v}_{31}$}}\rput(4.8,-5.2){\sm{$\wt{v}_{32}$}}
\psline[linewidth=.06](3.69,-3.08)(4.49,-3.08)\psline[linewidth=.06](-3.69,-3.08)(-4.49,-3.08)
\psline[linewidth=.06](2.27,.87)(1.77,0)\psline[linewidth=.06](2.27,-6.93)(1.77,-6.16)
\psline[linewidth=.06](-2.27,.87)(-1.77,0)\psline[linewidth=.06](-2.27,-6.93)(-1.77,-6.16)
\rput(4.9,-3.08){\sm{$r$}}\rput(-4.9,-2.9){\sm{$\bar{r}$}}
\rput(2.8,1.2){\sm{$\bar{q}$}}\rput(-2.8,-7.26){\sm{$q$}}
\rput(2.7,-7.26){\sm{$\bar{p}$}}\rput(-2.6,1.1){\sm{$p$}}
\psline[linewidth=.1](30,.65)(30,-5.58)\psline[linewidth=.1](30.65,0)(24.42,0)
\psline[linewidth=.1](30.85,-5.23)(24.77,.85)
\psarc(30.65,.65){.65}{270}{180}\psarc(30.5,-5.58){.5}{180}{45}
\psarc(24.42,.5){.5}{45}{270}
\pscircle*(30,0){.3}\pscircle*(30,-4.38){.3}\pscircle*(25.62,0){.3}
\rput(29.3,-.8){\sm{$\wt{x}_1$}}\rput(25.4,-.9){\sm{$\wt{x}_2$}}\rput(29.2,-4.8){\sm{$\wt{x}_3$}}
\psline[linewidth=.06](29.6,-2.19)(30.4,-2.19)\rput(30.9,-2.19){\sm{$r$}}
\psline[linewidth=.06](24.42,.5)(23.5,.88)\rput(23.1,1.1){\sm{$\bar{r}$}}
\psline[linewidth=.06](27.81,-.4)(27.81,.4)\rput(27.81,1){\sm{$p$}}
\psline[linewidth=.06](27.53,-2.47)(28.09,-1.91)\rput(27,-2.7){\sm{$q$}}
\psline[linewidth=.06](30.83,.83)(31.39,1.39)\rput(31.8,1.5){\sm{$\bar{q}$}}
\psline[linewidth=.06](30.5,-5.58)(30.88,-6.5)\rput(31.4,-6.5){\sm{$\bar{p}$}}
\psline[linewidth=.1]{->}(10,-3.08)(20,-3.08)\rput(15,-2.4){$\io$}
\rput(23,-5.38){$V\subset\wt{X}$}
\end{pspicture}
\caption{The normalization of an NC variety with $\io(\wt{v}_{ij})\!=\!\wt{x}_i$
and $\psi(\wt{v}_{ij})\!=\!\wt{v}_{ji}$.}
\label{NCC_fig}
\end{figure}
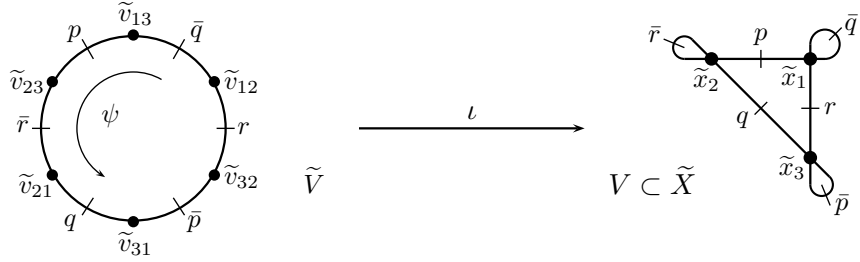

\noindent
A ``3-fold" version of Example~\ref{NCC_eg2} is shown in Figure~\ref{NCC_fig}.
The topological space~$X_{\io,\psi}$ is obtained from~$\wt{X}$ by folding
the NC divisor~$V$ as indicated by the action of~$\psi$ on its normalization~$\wt{V}$.
This folding is not induced by an involution on~$V$ itself;
while most points of~$V$ are identified in pairs, 
the three double points are identified into~one.

\section{On the proof of Theorem~\ref{NCC_thm}}
\label{NCCpf_sec}

\noindent
We now explain why the proof of \cite[Theorem~2.17]{SympDivConf}, 
which is outlined in \cite[Figure~2]{SympDivConf}, extends to Theorem~\ref{NCC_thm}.
This proof revolves around a weaker version
of the notion of regularization of Definition~\ref{TransConfregul_dfn}\ref{SCCreg_it},
which is readily adaptable to Definition~\ref{NCCregul_dfn}.

\begin{dfn}\label{ConfRegulLoc_dfn}
Suppose $\X\!\equiv\!\{X_I\}_{I\in\cP^*(N)}$ is a transverse configuration as in Definition~\ref{TransConf_dfn1},
\hbox{$I^*\!\in\!\cP^*(N)$}, and $U\!\subset\!X_{I^*}$ is an open subset.
A \sf{regularization for~$U$ in~$\X$} is a tuple  $(\Psi_i)_{i\in I^*}$, 
where $\Psi_i$ is a regularization for~$U$ in~$X_i$ 
in the sense of Definition~\ref{smreg_dfn}, such that
\begin{alignat*}{2}
\Psi_i\big(\cN_{I^*;I}\!\cap\!\Dom(\Psi_i)\big)&=X_I\!\cap\!\Im(\Psi_i)
&\quad &\forall\,i\!\in\!I\!\subset\!I^*,\\
\Psi_{i_1}\big|_{\cN_{I^*;i_1i_2}\cap\Dom(\Psi_{i_1})}
&=\Psi_{i_2}\big|_{\cN_{I^*;i_1i_2}\cap\Dom(\Psi_{i_2})}
&\quad &\forall\,i_1,i_2\!\in\!I^*\,.
\end{alignat*}
\end{dfn}

\begin{dfn}\label{LocalRegul_dfn}
Suppose $\X\!\equiv\!\{X_I\}_{I\in\cP^*(N)}$ is a transverse configuration
and $W\!\subset\!X_{\eset}$ is an open subset.
A \sf{weak $(\om_i)_{i\in[N]}$-regularization for} a transverse configuration~$\X$ over~$W$
is a tuple as in~\eref{SCCregdfn_e0} such~that 
\begin{enumerate}[label=$\bu$,leftmargin=*]

\item for every $I\!\in\!\cP^*(N)$, the tuple $(\Psi_{I;i})_{i\in I}$
is a regularization for $X_I\!\cap\!W$ in~$\X$;

\item for all $i\!\in\!I\!\subset\![N]$, the tuple 
$((\rho_{I;j},\na^{(I;j)})_{j\in I-i},\Psi_{I;i})$
is an $\om_i$-regularization for $X_I\!\cap\!W$ in~$X_i$
in the sense of Definition~\ref{sympreg1_dfn};

\item for all $i\!\in\!I'\!\subset\!I\!\subset\![N]$,
the bundle isomorphism~$\fD\Psi_{I;i;I'}$ associated with~$\fD\Psi_{I;i}$
as in~\eref{wtPsiIIdfn_e} is a product Hermitian isomorphism and 
\BE{LocalRegul_e2}
\Psi_{I;i}\big|_{\Dom(\Psi_{I;i})\cap\fD\Psi_{I;i;I'}^{\,-1}(\Dom(\Psi_{I';i}))}
=\Psi_{I';i}\circ\fD\Psi_{I;i;I'}|_{\Dom(\Psi_{I;i})\cap\fD\Psi_{I;i;I'}^{\,-1}(\Dom(\Psi_{I';i}))}\,.\EE
\end{enumerate}
\end{dfn}

\vspace{.1in}

\noindent
Thus, an $(\om_i)_{i\in[N]}$-regularization for~$\X$ in the sense of 
Definition~\ref{TransConfregul_dfn}\ref{SCCreg_it}
is a weak $(\om_i)_{i\in[N]}$-regularization for~$\X$ over $W\!=\!X_{\eset}$
such that 
$$\Dom(\Psi_{I;i})=\fD\Psi_{I;i;I'}^{\,-1}(\Dom(\Psi_{I';i}))
\qquad\forall\,i\!\in\!I'\!\subset\!I\!\subset\![N],~|I'|\!\ge\!2,$$
as required by the first condition in~\eref{overlap_e}.
If $W,W_1,W_2\!\subset\!X_{\eset}$ are open subsets with \hbox{$W\!\subset\!W_1\!\cap\!W_2$}
and $(\fR_t^{(1)})_{t\in B}$ and $(\fR_t^{(2)})_{t\in B}$
are families of weak regularizations for~$\X$ over~$W_1$ and~$W_2$, respectively,
we define
$$\big(\fR_t^{(1)}\big)_{t\in B} \cong_W\big(\fR_t^{(2)}\big)_{t\in B}$$
if the restrictions of the two regularizations to~$W$ agree on
the level of germs;
see the first part of \cite[Section~5.1]{SympDivConf} for a formal definition.\\

\noindent
Let $X$ be an NC variety with an NC atlas  $(U_y,\X_y,\vph_y)_{y\in\cA}$
as in Definition~\ref{ImmTransConf_dfn}.
For each open subset $W\!\subset\!X$, let 
$$\cA(W)=\big\{y\!\in\!\cA\!:\,U_y\!\subset\!W\big\}.$$
In particular, $(U_y,\X_y,\vph_y)_{y\in\cA(W)}$ is an NC atlas for~$W$.
For a symplectic structure~$\om$ on~$X$ as in Definition~\ref{NCCregul_dfn}\ref{NCCregul_it1}
and an open subset $W\!\subset\!X$,
a \sf{weak $\om$-regularization for~$X$ over~$W$} is
a tuple~$(\fR_y)_{y\in\cA(W)}$ as in Definition~\ref{NCCregul_dfn}\ref{NCCregul_it1}
so that 
each~$\fR_y$ is a weak $(\om_{y;i})_{i\in[N_y]}$-regularization for $\X_y$ and 
\eref{NCCregulover_e} holds for all $y,y'\!\in\!\cA(W)$ and $x\!\in\!U_{yy';x}\!\subset\!U_y\!\cap\!U_{y'}$ 
as in Definition~\ref{ImmTransConf_dfn0}\ref{NCatlas_it}.\\

\noindent
Suppose $(\wt{X},\wt\om)$ is a symplectic manifold,
 $(\io,\psi)$ is a compatible pair as in~\eref{iopsidfn_e}
so that~$\io$ is a closed transverse immersion of codimension~2 and $\psi$ is 
a smooth involution, and $W$ is an open subset of the NC variety $X\!=\!X_{\io,\psi}$.
From the global perspective of Definition~\ref{NCSCCregul_dfn}, 
a \sf{weak $\om$-regularization for~$X$ over~$W$} is a tuple as in~\eref{NCSCCregul_e}
satisfying~\eref{NCSCCregul_e3} such~that 
\begin{enumerate}[label=$\bu$,leftmargin=*]

\item  the tuple $(\Psi_{k;i})_{k\in\Z^+,i\in[k]}$ is a regularization for
the restriction of~$(\io,\psi)$ to  $\io^{-1}(q^{-1}(W))$,

\item the tuple~\eref{NCSCCregul_e0} 
is a refined $\wt\om$-regularization for the restriction of~$\io$ to 
$\io^{-1}(q^{-1}(W))$, except the first condition in~\eref{NCoverlap_e} with 
$\Psi_k\!\equiv\!\Psi_{k+1;1}\!\circ\!D_{k;1}$ may not hold 
and the second holds over 
the intersection of the domains of the two sides, 
i.e.~$\Dom(\Psi_k)\!\cap\!\fD\Psi_{k;k'}^{-1}(\Dom(\Psi_{k'}))$.

\end{enumerate}
By Lemma~5.8 and Corollary~5.9 in~\cite{SympDivConf},
weak regularizations and equivalences between them in the simple NC~setting
can be cut down to regularizations and equivalences between regularizations.
The same reasoning applies in the arbitrary NC~setting viewed from
the global perspective of either Section~\ref{NCCgl_subs} or~\ref{NCgl2_subs}.
Thus, it is sufficient to establish Theorem~\ref{NCC_thm}
with {\it regularizations} replaced by {\it weak regularizations} everywhere.\\

\noindent
The last task is readily accomplished by combining the proof of \cite[Theorem~2.7]{SympDivConf}
with the local perspective of Section~\ref{NCCloc_subs} via an inductive construction.
Let 
$$\big(U_y,\X_y\!\equiv\!(X_{y;I})_{I\in\cP^*(N_y)},
\vph_y\!:U_y\!\lra\!X_{y;\eset}\big)_{y\in\Z^+}$$
be a locally finite collection of NC charts covering~$X$ and
$(U_y')_{y\in\Z^+}$ be an open cover of~$X$ such that $\ov{U_y'}\!\subset\!U_y$ for every $y\!\in\!\Z^+$.
For each $y^*\!\in\!\Z^+$, the tuple
$(\fR_{t;y})_{t\in N(\prt B),y\in\cA(U_{y^*})}$
is  an $(\om_t)_{t\in N(\prt B)}$-family of weak regularizations for~$X$ over~$U_{y^*}$
in the local sense.
It determines an $(\om_{t;y^*;i})_{t\in N(\prt B),i\in[N_y^*]}$-family 
$(\fR_{y^*;t})_{t\in N(\prt B)}$ of weak regularizations
for~$\X_{y^*}$ over $X_{y^*;\eset}$
in the sense of Definition~\ref{LocalRegul_dfn}.\\

\noindent
Suppose $y^*\!\in\!\Z^+$ and we have constructed 
\begin{enumerate}[label=(I\arabic*),leftmargin=*]

\item\label{neighbX_it} an open neighborhood~$W_{y^*}$ of
$$\ov{U_{y^*}^<}\equiv\bigcup_{y<y^*}\!\!\ov{U_y'}\subset X\,,$$

\item\label{neighbB_it} a neighborhood $N_{y^*}(\prt B)$ of  $\ov{N'(\prt B)}$ in $N(\prt B)$,

\item  a smooth family $(\mu_{t,\tau})_{t\in B,\tau\in\bI}$ of
1-forms on~$X$ such~that 
\BE{NC_e0}\mu_{t,0}=0, ~~ 
\supp\big(\mu_{\cdot,\tau}\big)\!\subset\! 
\big(B\!-\!N_{y^*}(\prt B)\big)\!\times\!(X\!-\!X^*), ~~
\om_{t,\tau}\!\equiv\!\om_t\!+\!\nd\mu_{t,\tau}\in \Symp^+(X)\EE
for all $t\!\in\!B$ and $\tau\!\in\!\bI$, 

\item an $(\om_{t,1})_{t\in B}$-family $(\fR_{t;y}')_{t\in B,y\in\cA(W_{y^*})}$
of weak regularizations for~$X$ over~$W_{y^*}$ such~that
\BE{NC_e3}\big(\fR'_{t;y}\big)_{t\in N_{y^*}(\prt B),y\in\cA(W_{y^*})}
\cong_{W_{y^*}} \!\!\big(\fR_{t;y}\big)_{t\in N_{y^*}(\prt B),y\in\cA(W_{y^*})}.\EE

\end{enumerate}
This family determines an $(\om_{t,1;y^*;i})_{t\in B,i\in[N_y^*]}$-family 
$(\fR_{y^*;t}')_{t\in B}$ of weak regularizations
for~$\X_{y^*}$ over
$$  X_{y^*;\eset}^<\equiv \vph_{y^*}\big(U_{y^*}\!\cap\!W_{y^*}\big) \subset X_{y^*;\eset}$$
such that
$$\big(\fR'_{y^*;t}\big)_{t\in N_{y^*}(\prt B)}
\cong_{X_{y^*;\eset}^<} \!\!\big(\fR_{y^*;t}\big)_{t\in N_{y^*}(\prt B)}.$$

\vspace{.2in}

\noindent
Let \hbox{$N_{y^*+1}(\prt B)\!\subset\!N_{y^*}(\prt B)$} 
be a neighborhood of~$\ov{N'(\prt B)}$,
$W'\!\subset\!W_{y^*}$ be a neighborhood of $\ov{U_{y^*}^<}$, and
$U'\!\subset\!U''\!\subset\!U_{y^*}$ be neighborhoods of~$\ov{U_{y^*}'}$
such~that 
$$\ov{N_{y^*+1}(\prt B)}\subset N_{y^*}(\prt B),   \qquad 
\ov{W'}\subset W_{y^*}, \qquad \ov{U'}\subset U'', \quad\hbox{and}\quad\ov{U''}\subset U_{y^*}\,.$$
Define
\begin{gather*}
W_{y^*+1}=U'\!\cup\!W',\qquad
X_{y^*;i}''=X_{y^*;i}\cap\!\vph_{y^*}\!(U'')\quad\forall\,i\!\in\![N_{y^*}], \\
X_{y^*;\eset}'^<=\vph_{y^*}\!\big(U_{y^*}'\!\cap\!W\big)\subset X_{y^*;\eset}, \qquad
X_{y^*}^*=\vph_{y^*}\!\big(U_{y^*}\!\cap\!X^*\big)\,.
\end{gather*}
In particular, \ref{neighbX_it} and~\ref{neighbB_it} hold with $y^*$ replaced by $y^*\!+\!1$.
By repeated applications of \cite[Proposition~5.3]{SympDivConf} as in
the proof of \cite[Theorem~2.17]{SympDivConf} at the end of \cite[Section~5.1]{SympDivConf}, 
we obtain
\begin{enumerate}[label=$\bullet$,leftmargin=*]

\item a smooth family $(\mu_{y^*;t,\tau;i}')_{t\in B,\tau\in\bI,i\in[N_{y^*}]}$ of
1-forms on~$X_{y^*;\eset}$ such~that 
\begin{gather*} 
\mu_{y^*;t,0;i}'=0, \quad 
\supp\big(\mu_{y^*;\cdot,\tau;i}'\big)\subset 
\big(B\!-\!N_{y^*+1}(\prt B)\big)\!\times\!\big(X_{y^*;i}''\!-\!X_{y^*;\eset}'^<\!\cup\!X_{y^*}^*\big)
\quad \forall~t\!\in\!B,\,\tau\!\in\!\bI,\,i\!\in\![N_{y^*}],\\
\big(\om_{y^*;t,\tau;i}'\equiv\om_{y^*;t,1;i}\!+\!\nd\mu_{y^*;t,\tau;i}'\big)_{i\in[N_{y^*}]}
\in \Symp^+(\X_{y^*}) \quad \forall~t\!\in\!B,\,\tau\!\in\!\bI,
\end{gather*}

\item an $(\om'_{y^*;t,1;i})_{t\in B,i\in[N_{y^*}]}$-family $(\wt\fR_{y^*;t})_{t\in B}$
of weak regularizations for~$\X_{y^*}$ over $\vph_{y^*}(U')$ such~that
$$\big(\wt\fR_{y^*;t}\big)_{t\in N_{y^*+1}(\prt B)} \cong_{\vph_{y^*}(U')}\! 
\!\!\big(\fR_{y^*;t}\big)_{t\in N_{y^*+1}(\prt B)}\,,
\quad 
\big(\wt\fR_{y^*;t}\big)_{t\in B} \cong_{\vph_{y^*}\!(U'\cap W')} \!\big(\fR_{y^*;t}'\big)_{t\in B}\,\,.$$
\end{enumerate}

\vspace{.1in}

\noindent
Since $X_{y^*;i}''\!\subset\!\vph_{y^*}\!(U'')$,
the support of $\mu_{y^*;t,\tau;i}'$ is contained in $\vph_{y^*}(U'')$.
For all $t\!\in\!B$ and $\tau\!\in\!\bI$, the tuple
$(\mu_{y^*;t,\tau;i}')_{i\in[N_{y^*}]}$ thus determines 
a 1-form  $\mu_{t,\tau}'$ on~$X$ such~that 
\begin{gather}\notag
\vph_{y^*}^{\,*}\mu_{t,\tau}'=\big(\mu_{y^*;t,\tau;i}'\big)_{i\in[N_{y^*}]}\,,\\
\label{NC_e7}
\mu_{t',0}=0, ~~ 
\supp\big(\mu_{\cdot,\tau}'\big)\!\subset\! 
\big(B\!-\!N_{y^*+1}(\prt B)\big)\!\times\!(U_{y^*}\!-\!X^*), ~~
\om_{t,\tau}'\!\equiv\!\om_{t;1}\!+\!\nd\mu_{t,\tau}'\in \Symp^+(X).
\end{gather}
These forms vary smoothly with $t\!\in\!B$ and $\tau\!\in\!\bI$.\\

\noindent
Let $\be_1,\be_2\!:\bI\!\lra\!\bI$ be smooth non-decreasing functions such that
$$\be_1(\tau)=\begin{cases}\tau,&\hbox{if}~1\!-\!\tau\!\ge\!2^{-y^*};\\
1,&\hbox{if}~1\!-\!\tau\!\le\!2^{-y^*-1};
\end{cases} \qquad
\be_2(\tau)=\begin{cases}0,&\hbox{if}~1\!-\!\tau\!\ge\!2^{-y^*-1};\\
1,&\hbox{if}~\tau\!=\!1.
\end{cases}$$
We concatenate the families $(\mu_{t,\tau})_{t\in B,\tau\in\bI}$ and
$(\mu_{t,\tau}\!+\!\mu_{t,\tau}')_{t\in B,\tau\in\bI}$ of \hbox{1-forms} on~$X$ 
into a new smooth family $(\mu_{t,\tau})_{t\in B,\tau\in\bI}$~by
$$\mu_{t,\tau}''=\mu_{t,\be_1(\tau)}+\mu_{t,\be_2(\tau)}'\qquad\forall~t\!\in\!B,\,\tau\!\in\!\bI\,.$$
By~\eref{NC_e7}, 
\BE{NC_e8}\begin{aligned} 
\mu_{t,\tau}''(x)&=\mu_{t,1}''(x) &\qquad
&\forall~x\!\in\!X\!-\!U_{y^*},~1\!-\!\tau\!\le\!2^{-y^*-1}, \\
\mu_{t,\tau}''(x)&=\mu_{t,\tau}(x)  &\qquad &\forall~x\!\in\!X\!-\!U_{y^*}
~\hbox{s.t.}~\mu_{t,\tau'}(x)\!=\!\mu_{t,1}(x)~\forall~1\!-\!\tau'\!\le\!2^{-y^*}\,.
\end{aligned}\EE
Furthermore, \eref{NC_e0} holds with $\mu$ and $N_{y^*}(\prt B)$
replaced by~$\mu''$ and $N_{y^*+1}(\prt B)$, respectively.\\

\noindent
The family $(\wt\fR_{y^*;t})_{t\in B}$ determines an $(\om'_{t,1})_{t\in B}$-family 
$(\wt\fR_{t;y})_{t\in B,y\in\cA(U')}$ of weak regularizations for~$X$ over~$U'$ 
such~that 
\begin{gather*}
\big(\wt\fR_{t;y}\big)_{t\in N_{y^*+1}(\prt B),y\in\cA(U')}\cong_{U'}\!
\big(\fR_{t;y}\big)_{t\in N_{y^*+1}(\prt B),y\in\cA(U')}\,,\\
\quad \big(\wt\fR_{t;y}\big)_{t\in B}\cong_{U'\cap W'}\!
\big(\fR_{t;y}'\big)_{t\in B}\quad\forall\,y\!\in\!\cA(U'\!\cap\!W')\,.
\end{gather*}
Along with $(\fR_{t;y}')_{t\in B,y\in\cA(W_{y^*})}$, it thus determines
an  $(\om'_{t,1})_{t\in B}$-family $(\fR_{t;y}'')_{t\in B,y\in\cA(W_{y^*+1})}$
of weak regularizations for $\X$ over~$W_{y^*+1}$ so~that 
\BE{NC_e9}\big(\fR_{t;y}''\big)_{t\in B} = (\fR_{t;y}')_{t\in B} 
\qquad\forall~y\!\in\!\cA\big(W_{y^*+1}\!-\!\ov{U'}\big) \EE
and~\eref{NC_e3} holds with $\fR'$, $N_{y^*}(\prt B)$, and~$W_{y^*}$ replaced by 
$\fR''$, $N_{y^*+1}(\prt B)$, and~$W_{y^*+1}$, respectively.\\

\noindent
Since the collection $(U_y)_{y\in\Z^+}$ is locally  finite, for every point $x\!\in\!X$
there exist a neighborhood $U_x\!\subset\!X$ of~$x$ and $N_x\!\in\!\Z^+$ such~that 
$$U_x\cap U_{y^*}=\eset \qquad\forall~y^*\!\in\!\Z^+,\,y^*\!\ge\!N_x\,.$$
By~\eref{NC_e8} and~\eref{NC_e9},
$$\mu_{t,\tau}''\big|_{U_x}=\mu_{t,\tau}\big|_{U_x}, \quad 
\fR_{t;y}''=\fR_{t;y}'  \qquad\forall~t\!\in\!B,\,\tau\!\in\!\bI,\,y\!\in\!\cA(U_x),~y^*\!>\!N_x.$$
The inductive construction above thus terminates after finitely many step on a sufficiently small neighborhood $U_x\!\subset\!X$
of each point $x\!\in\!X$ and
provides a smooth family $(\mu_{t,\tau})_{t\in B,\tau\in\bI}$ of 
1-forms on~$X$ satisfying~\eref{SCCom_e} and an $(\om_{t,1})_{t\in B}$-family 
$(\wt\fR_t)_{t\in B}$ of weak regularizations for~$X$ over~$X$ satisfying~\eref{SCCom_e2}.\\


\noindent
{\it Department of Mathematics, the University of Iowa, Iowa City, IA 52242\\
 mohammad-tehrani@uiowa.edu}\\

\noindent
{\it Department of Mathematics, Stony Brook University, Stony Brook, NY 11794\\
markmclean@math.stonybrook.edu, azinger@math.stonybrook.edu}

\clearpage

\end{document}